\documentclass[a4paper,11pt,reqno]{amsart}
\usepackage[margin=3cm]{geometry}

\newcommand{\version}{\today}

\usepackage{amsthm,amsfonts,amsmath, amscd, bbold}
\usepackage{amscd,dsfont}
\usepackage{mathrsfs}
\usepackage{accents}

\usepackage{mathrsfs}
\usepackage{amscd}
\usepackage[active]{srcltx}
\usepackage{verbatim}

\usepackage[latin1]{inputenc}
\usepackage{tikz}
\usetikzlibrary{trees}

\usepackage{mathtools}

\swapnumbers
                              
\theoremstyle{plain}
\newtheorem{thm}{THEOREM}[section]
\newtheorem{lm}[thm]{LEMMA}
\newtheorem{cl}[thm]{COROLLARY}

\theoremstyle{definition}
\newtheorem{defn}[thm]{DEFINITION}
\newtheorem{remark}[thm]{Remark}
\theoremstyle{remark}                       
\newcommand{\upchi}{\raise1pt\hbox{$\chi$}}
\newcommand{\R}{{\mathord{\mathbb R}}}
\newcommand{\C}{{\mathord{\mathbb C}}}
\newcommand{\Z}{{\mathord{\mathbb Z}}}

\newcommand{\tr}{{\rm Tr}}
\numberwithin{equation}{section}
\newcommand{\dd}{\:{\rm d}}

\newcommand{\un}{{\rm 1\kern -2.5pt l}}

 \newcommand{\rb}{\rho_{2}}

\begin{document}
\markboth{\scriptsize{CM \version}}{\scriptsize{CM September 2, 2016}}
\def\Z{{\mathbb Z}}
\def\R{{\mathbb R}}
\def\C{{\mathbb C}}
\def\sg{\sigma}
\def\S{\mathcal{S}}

\newcommand{\E}{{\mathcal E}}
\def\Ex{{\mathbb E}}


\renewcommand{\a}{\alpha}
\renewcommand{\b}{\beta}
\newcommand{\ga}{\gamma}
\newcommand{\Ga}{\Gamma}
\renewcommand{\d}{\delta}
\newcommand{\e}{\varepsilon}
\renewcommand{\l}{\lambda}
\newcommand{\om}{\omega}
\renewcommand{\O}{\Omega}
\newcommand{\Om}{\Omega}
\renewcommand{\th}{\theta}
\newcommand{\VV}{\mathbf{V}}
\newcommand{\UU}{\mathbf{U}}

\makeatletter
\DeclareRobustCommand\widecheck[1]{{\mathpalette\@widecheck{#1}}}
\def\@widecheck#1#2{%
    \setbox\z@\hbox{\m@th$#1#2$}%
    \setbox\tw@\hbox{\m@th$#1%
       \widehat{%
          \vrule\@width\z@\@height\ht\z@
          \vrule\@height\z@\@width\wd\z@}$}%
    \dp\tw@-\ht\z@
    \@tempdima\ht\z@ \advance\@tempdima2\ht\tw@ \divide\@tempdima\thr@@
    \setbox\tw@\hbox{%
       \raise\@tempdima\hbox{\scalebox{1}[-1]{\lower\@tempdima\box
\tw@}}}%
    {\ooalign{\box\tw@ \cr \box\z@}}}
\makeatother


\newcommand{\beq}{\begin{equation}}
\newcommand{\eeq}{\end{equation}}
\newcommand{\bal}{\begin{aligned}}
\newcommand{\eal}{\end{aligned}}
\newcommand{\ben}{\begin{enumerate}}
\newcommand{\beni} {\begin{enumerate}[(i)]}
\newcommand{\een}{\end{enumerate}}
\newcommand{\bit}{\begin{itemize}}
\newcommand{\eit}{\end{itemize}}
\newcommand{\beqw}{\begin{equation*}}
\newcommand{\eeqw}{\end{equation*}}
\newcommand{\bthm}{\begin{theorem}}
\newcommand{\ethm}{\end{theorem}}
\newcommand{\bpr}{\begin{proposition}}
\newcommand{\epr}{\end{proposition}}
\newcommand{\ble}{\begin{lemma}}
\newcommand{\ele}{\end{lemma}}
\newcommand{\blem}{\begin{lemma}}
\newcommand{\elem}{\end{lemma}}
\newcommand{\bpf}{\begin{proof}}
\newcommand{\epf}{\end{proof}}
\newcommand{\bex}{\begin{example}}
\newcommand{\eex}{\end{example}}
\newcommand{\bre}{\begin{example}}
\newcommand{\ere}{\end{example}}
\newcommand{\bma}{\begin{bmatrix}}
\newcommand{\ema}{\end{bmatrix}}
\newcommand\T{{\mathbb T}}

\newcommand{\ddt}{\frac{\mathrm{d}}{\mathrm{d}t}}
\newcommand{\ddtr}{\frac{\mathrm{d}^+}{\mathrm{d}t}}
\newcommand{\ddhr}{\frac{\mathrm{d}^+}{\mathrm{d}h}}
\newcommand{\ddtt}{\frac{\mathrm{d^2}}{\mathrm{d}t^2}}
\newcommand{\ddr}{\frac{\mathrm{d}}{\mathrm{d}r}}


\renewcommand{\aa}{{\boldsymbol\alpha}}
\newcommand{\bb}{{\boldsymbol\beta}}
\renewcommand{\gg}{{\boldsymbol\gamma}}
\renewcommand{\dd}{{\boldsymbol\delta}}
\newcommand{\Dom}{{\mathsf D}}
\newcommand{\wt}{\widetilde}
\newcommand{\one}{{\mathds1}}
\newcommand{\embed}{\hookrightarrow}
\newcommand{\Null}{\mathsf{N}}
\newcommand{\rank}{\mathrm{rank}\;}
\newcommand{\sgn}{{\rm sgn}}
\renewcommand{\hat}{\widehat}
\newcommand{\ot}{\otimes}
\newcommand{\hN}{\widehat\Num}
\newcommand{\G}{\Gamma}
\def\bh{{\bf h}}
\def\tand{\quad{\rm and}\quad}
\def\L{\mathcal{L}}\def\bk{\rangle\langle}
\newcommand{\cA}{{\mathord{\mathcal A}}}
\def\ib{\iota_\beta}
\def\H{\mathcal{H}}
\def\E{\mathcal{E}}
\def\Eb{\mathcal{E}_\beta}
\def\rb{\sigma_\beta}
\newcommand{\bx}{{\bf x}}
\newcommand{\cM}{\mathcal{M}}
\newcommand{\aA}{\mathcal{A}}
\newcommand{\aB}{\mathcal{B}}
\newcommand{\cP}{{\mathord{\mathscr P}}}
\newcommand{\cQ}{{\mathord{\mathscr Q}}}
\newcommand{\cE}{{\mathcal{E}}}
\newcommand{\cL}{{\mathord{\mathscr L}}}
\newcommand{\cB}{{\mathord{\mathscr B}}}
\newcommand{\scH}{{\mathord{\mathscr H}}}
\newcommand{\scC}{{\mathord{\mathscr C}}}
\newcommand{\scA}{{\mathord{\mathscr A}}}
\newcommand{\scJ}{{\mathord{\mathscr J}}}
\newcommand{\scK}{{\mathord{\mathscr K}}}
\newcommand{\scG}{{\mathord{\mathscr G}}}
\newcommand{\scS}{{\mathord{\mathscr S}}}
\newcommand{\scQ}{{\mathord{\mathscr Q}}}
\newcommand{\sE}{{\mathord{\mathscr E}}}
\newcommand{\scL}{{\mathord{\mathscr L}}}
\newcommand{\scM}{{\mathord{\mathscr M}}}
\newcommand{\cJ}{{\mathcal{J}}}
\newcommand{\cH}{{\mathcal{H}}}
\newcommand{\cK}{{\mathcal{K}}}
\newcommand{\cG}{{\mathcal{G}}}
\newcommand{\cD}{{\mathcal{D}}}
\newcommand{\cN}{{\mathcal{N}}}
\newcommand{\cC}{{\mathcal{C}}}
\newcommand{\cU}{{\mathcal{U}}}
\newcommand{\cR}{{\mathcal{R}}}
\newcommand{\cS}{{\mathcal{S}}}
\newcommand{\cT}{{\mathcal{T}}}
\newcommand{\cV}{{\mathcal{V}}}

\newcommand{\B}{{\mathcal{B}}}
\newcommand{\calS}{{\mathcal{S}}}
\newcommand{\F}{{\mathcal{F}}}
\newcommand{\fH}{{\mathfrak{H}}}
\newcommand{\bbA}{{\bf A}}
\newcommand{\Cln}{\mathfrak{C}^n}  
\newcommand{\err}{{r}}

\newcommand{\PX}{\cP(\cX)}
\newcommand{\PXs}{\cP_*(\cX)}
\newcommand{\hrhom}{\widehat{\rho_{-j}}}
\newcommand{\hrhop}{\widehat{\rho_{+j}}}
\newcommand{\hrhopm}{\widehat{\rho_{\pm j}}}
\newcommand{\hrhomp}{\widehat{\rho_{\mp j}}}
\newcommand{\lrho}{\check{\rho}}
\newcommand{\ulrho}{\breve{\rho}}
\newcommand{\ulV}{\breve{V}}
\newcommand*{\bdot}[1]{%
  \accentset{\mbox{\large\bfseries .}}{#1}}
  \newcommand{\bu}{{\bf u}}
    \newcommand{\bv}{{\bf v}}
    \newcommand{\be}{{\bf e}}
     \newcommand{\bq}{{\bf q}}
      \newcommand{\bp}{{\bf p}}

\title
[Fermionic gauge invariant Gaussian maps]
{The structure of gauge invariant Gaussian quantum operations on finite Fermion systems}

\author{Eric A. Carlen}
\address{Department of Mathematics\\ 
Hill Center\\
Rutgers University\\
110 Frelinghuysen Road\\
Piscataway\\
NJ 08854-8019\\
USA}
\email{carlen@math.rutgers.edu}

\dedicatory{ Dedicated to Len Gross in celebration of his 95$^{{\rm th}}$ birthday}

\begin{abstract} Let ${\mathcal H}_1$ be a finite dimensional complex Hilbert space. Let $\psi\mapsto Z(\psi)$ be a canonical anti-commutation relations (CAR) field over ${\mathcal H}_1$ acting irreducibly on a Hilbert space ${\mathord{\mathscr K}}$.  The $*$-algebra ${\mathscr A}_{{\mathcal H}_1}$  generated by   $\{Z(\psi)\ :\ \psi\in {\mathcal H}_1\}$, is simply all operators on 
${\mathscr K}$.
 However,  the CAR field endows ${\mathscr A}_{{\mathcal H}_1}$ with additional structure, and we are concerned with quantum operations that act in harmony with this structure.  In particular, there is a {\em gauge automorphism group}, 
 generated by ``second quantizing''  $\psi \mapsto e^{it}\psi$.  
The fixed point algebra of this group is a  sub-algebra, $\scG_{\cH_1}$, of ${\mathscr A}_{{\mathcal H}_1}$ studied by Araki and Wyss.  It contains the density matrices of an important class of states, the {\em gauge invariant Gaussian states}, ${\mathfrak S}_{GIG}$. 
 
Our focus is on  semigroups $\{e^{t{\mathscr L}}\}_{t\geq 0}$ of quantum operations on ${\mathscr A}_{{\mathcal H}_1}$ that map ${\mathfrak S}_{GIG}$ into itself. Each  
$e^{t{\mathscr L}}$ is one-to-one, and our first main result is a structure theorem for such quantum operations on ${\mathscr G}_{{\mathcal H}_1}$ that 
map ${\mathfrak S}_{GIG}$ into itself. 
 We apply this  to  study  semigroups of quantum operations on 
${\mathscr G}_{{\mathcal H}_1}$  that map ${\mathfrak S}_{GIG}$ into itself. Our second main result is  a   structure theorem showing that they are parameterized by pairs $(G,A)$ where $G$ is a contraction semigroup generator on ${\mathcal H}_1$, and 
$0 \leq A \leq -G -G^*$.  We then show that each of these semigroups  has a natural extension to the full CAR algebra 
${\mathscr A}_{{\mathcal H}_1}$. Further results are obtained under further assumptions on the pair $(G,A)$.  
\end{abstract} 


\maketitle

\medskip
\centerline{Keywords: Quantum operations, Fermions, Gaussian states}

\centerline{Subject Classification Numbers: 46L57, 81V74, 47C15}



\maketitle

\section{Introduction}  

Len Gross pioneered many developments in the study of Fermion systems and the quantum operations; i.e., completely positive trace preserving (CPTP) maps, that act upon them. His work, particularly the papers \cite{Gr72,Gr75}, has inspired much work of my own. This work of Len Gross  was motived by problems in quantum field theory. More recently, developments in quantum information theory and quantum computing have led to new questions about quantum operations on Fermion systems, already in the case of finitely many degrees of freedom. In this paper we answer some such questions and discuss others that remain open. 

Systems of finitely many Fermionic degrees of freedom  are specified by a finite dimensional  complex  Hilbert space $\cH_1$, called the {\em single particle space},  and a conjugate linear map $Z$ from $\cH_1$ to $\cB(\scK)$, the bounded operators another Hilbert space $\scK$,  satisfying
\begin{equation}\label{FIELDS1}
Z(\psi)Z^*(\phi)+Z^*(\phi) Z(\psi)   = \langle \psi,\phi\rangle \one 
\end{equation}
together with
\begin{equation}\label{FIELDS2}
Z(\psi)Z(\phi)+ Z(\phi)Z(\psi) = 0  \quad{\rm and}\quad Z^*(\psi)Z^*(\phi)+ Z^*(\phi)Z^*(\psi) = 0 \ 
\end{equation}
for all $\phi,\psi\in \cH_1$.

The equations \eqref{FIELDS1} and \eqref{FIELDS2} are known as the {\em canonical anti-commutation relations} (CAR) and the algebra generated  the operators $Z(\psi)$ and 
their adjoints is known as the CAR algebra over $\cH_1$,  denoted  by $\scA_{\cH_1}$.  A CAR field is also known as a {\em representation} of the abstract CAR algebra which 
can be constructed in the tensor algebra over 
$\cH_1$; see e.g. \cite{AW64}.

Let $Z$ be a CAR field over $\cH_1$ acting on a Hilbert space $\scK$. Let $\cH$ be any subspace of $\cH_1$. Then the restriction of $Z$  to $\cH$ defines a 
CAR field over $\cH$. Define  $\scA_{\cH}$ to be the $*$-sub-algebra of $\scA_{\cH_1}$ generated by $\{Z(\psi)\ :\ \psi\in \cH\}$.  We denote this by writing
\begin{equation}\label{AHDEF}
\scA_{\cH} = W^*(\{Z(\psi)\ :\ \psi\in \cH\})\ .
\end{equation}

It is well-known that if the dimension of $\cH_1$ is $N$  and  no non-trivial subspace of $\scK$  is invariant under all of the operators in $\scA_{\cH_1}$ (so that $Z$ acts irreducibly on 
$\scK$), then $\scK$ is $2^N$-dimensional and $\scA_{\cH_1} = \cB(\scK)$.  Thus, every CPTP map on $\cB(\scK)$ is a CPTP map on $\scA_{\cH_1}$, and the well-developed general theory of completely positive maps (see \cite{Paul02}, or in the finite dimensional case, \cite{C25}) is immediately applicable.

However, $\cB(\scK)$, 
viewed from the perspective of a particular CAR  field that generates it, has physically significant structure built out of the CAR field, and we shall be interested in CPTP maps that 
preserve aspects of this structure.

We now introduce this structure, and start by explaining the relation between any two CAR fields over $\cH_1$. The first construction of a CAR field over a finite dimensional 
Hilbert space was given by Jordan and Wigner in 1928 \cite{JW28}, shortly after the beginning of modern quantum mechanics. Fock later gave another construction \cite{F32}. Both of these constructions have been reviewed many times; Derezi\'nski \cite{D06} provides a clear and recent  exposition of this and more. 

In finite 
dimensions, the relation between any two constructions is quite simple.
A classical theorem on the structure of  Clifford algebras can be re-formulated (see \cite{D06}) to show that when $\cH_1$ has finite dimension $N$,  and $Z$ and $\widetilde{Z}$  are
two CAR fields over 
$\cH_1$ acting on irreducibly  on Hilbert spaces $\scK$ and $\widetilde{\scK}$ respectively,  then both $\scK$ and $\widetilde{\scK}$ have dimension $2^N$, and the fields 
(representations) are unitarily equivalent in the sense that there exists 
a unitary $\cU:\scK \to \widetilde{\scK}$ such that for all $\psi\in \cH_1$, 
$$
\widetilde{Z}(\psi) = \cU Z(\psi) \cU^*\ .
$$
Moreover, because $\scK$ is a finite dimensional Hilbert space, any $*$-automorphism $\alpha$ of $\scA_{\cH_1} =\cB(\scK)$ is inner and 
implemented by a unitary conjugation \cite{BW}. That is, there exists a unitary $\cU_\alpha^{\phantom{*}}:\scK\to \scK$ such that for all $A\in \scA_{\cH_1}$,
$$
\alpha(A) := \cU_\alpha^{\phantom{*}} A \cU^*_\alpha\ ,
$$
and the unitary $\cU_\alpha$ is uniquely determined by $\alpha$ up to a phase.  

Throughout the following discussion, we assume that there is given  some CAR field $Z$  over $\cH_1$ acting irreducibly on $\scK$. 
If $U$ is any unitary on $\cH_1$, define
\begin{equation}\label{SECQUNTU}
\alpha_U(Z(\psi)) :=  Z(U\psi)
\end{equation}
for all $\psi\in \cH_1$. One readily checks that $\alpha_U(Z)$ is a CAR field, and hence $\alpha_U$ extends to a $*$-automorphism of $\scA_{\cH_1}$ implemented by a unitary 
$\cU_U$, determined up to a phase.  

A concrete construction of $\cU_U$  determines a preferred  phase. The first step is to introduce the {\em number  operator}.
Let $\{\psi_1,\dots,\psi_N\}$ be any orthonormal basis of $\cH_1$.  Define the number operator $\cN$ by
\begin{equation}\label{NUMBEROPDEFI}
\cN := \sum_{j=1}^N Z^*(\psi_j)Z(\psi_j)\ .
\end{equation}
Using the CAR, one readily computes the commutators
\begin{equation}\label{NUMOPNT3I}
[\cN,Z(\phi)] = -Z(\phi)\quad{\rm and\ hence}\quad  [\cN,Z^*(\phi)] = Z^*(\phi)
\end{equation}
for all choices of $\phi$.
 
Evidently $\cN$ is positive semi-definite, and if $\Phi$ is an eigenvector of $\cN$ with eigenvalue $\lambda$, it follows from  \eqref{NUMOPNT3I} that $\cN Z(\psi_j)\Phi = (\lambda -1)Z(\psi_j)\Phi$,
and then as is well known, it follows from this that the spectrum of $\cN$ is $\{0,1,\dots,N\}$. The eigenspaces for the eigenvalues $0$ and $N$ are one dimensional. Choose unit vectors $\varOmega_0$ and $\varOmega_1$ in the $0$-eigenspace and the $N$-eigenspace respectively. Define two states on $\scA_{\cH_1}$ by
\begin{equation}\label{vacantivacst}
\omega_0(A) = \langle \varOmega_0,A\varOmega_0\rangle \quad{\rm and}\quad  \omega_1(A) = \langle \varOmega_1,A\varOmega_1\rangle\ .
\end{equation}
These are the {\em vacuum state} and the {\em anti-vacuum states} respectively. While the vectors $\varOmega_0$ and $\varOmega_1$, called the {\em vacuum vector} and the {\em anti-vacuum vector} are determined only up to a phase, $\omega_0$ and  $\omega_1$ are uniquely determined by the CAR field $Z$.

For $\aa = (\alpha_1,\dots,\alpha_N)\in \{0,1\}^N$, define $\varPhi_\aa\in \scK$ by
\begin{equation}\label{PHIKDEF}
\varPhi_{\aa} := (Z^*(\psi_1))^{\alpha_1} \cdots (Z^*(\psi_N))^{\alpha_N}\varOmega_0\ .
\end{equation}
Then it is easy to see, using the CAR, that $\{\varPhi_\aa\}_{\aa\in \{0,1\}^N}$  is an orthonormal basis of $\scK$.  Each $\varPhi_\aa$ is an eigenvector of each $Z^*(\psi_j)Z(\psi_j)$ with
$$
Z^*(\psi_j)Z(\psi_j) \varPhi_\aa  = \alpha_j \varPhi_\aa\ .
$$
Thus, $\{Z^*(\psi_j)Z(\psi_j)\ :\ 1 \leq j \leq N\}$ is a complete set of mutually commuting observables consisting of orthogonal projections. Physically, $Z^*(\psi_j)Z(\psi_j)$ is the observable whose eigenvalue $1$ corresponds to an observation that ``the $j^{{\rm th}}$ Fermionic mode is occupied'', and  whose eigenvalue $0$ corresponds to an observation that ``the $j^{{\rm th}}$ Fermionic mode is unoccupied''.   Moreover, with $|\aa| = \sum_{\j=1}^N\alpha_j$, $\cN \varPhi_{\aa} = |\aa|\varPhi_\aa$.  This yields a spectral decomposition of $\cN$: 
For $n\in \{0,\dots,N\}$, define the orthogonal projections $E_n$ by
\begin{equation}\label{SPECDECNUM}
E_n = \sum_{\aa\ :\ |\aa| = n} |\varPhi_\aa\rangle\langle \varPhi_\aa| \quad{\rm so \ that }\quad \cN = \sum_{n=0}^N n E_n\ .
\end{equation}
Thus, a measurement of $\cN$  in the pure state $\varPsi$ yields $n$ with probability $\langle \varPsi, E_n\varPsi\rangle$, in which  case after the measurement the state is a 
superposition of states in which exactly $n$ of the Fermionic modes are occupied.  It is for this reason that $\cN$ is called the number operator.

Returning to the automorphisms of $\scA_{\cH_1}$ induced by   unitary operators on $\cH_1$,  let $U$ be any such unitary. Then  $\{U\psi_1,\dots,U\psi_N\}$ is another orthonormal basis  of $\cH_1$.
Define 
\begin{equation}\label{PHIKDEF2}
\widetilde{\varPhi}_{\aa} := (Z^*(U\psi_1))^{\alpha_1} \cdots (Z^*(U\psi_N))^{\alpha_N}\varOmega_0\ 
\end{equation}
for each $\aa$. Then  $\{\widetilde{\varPhi}_\aa\}_{\aa\in \{0,1\}^N}$ is another  orthonormal basis of $\scK$.   Define a unitary $\cU_U$ on $\scK$ by
\begin{equation}\label{PHIKDEF3}
\cU_U \varPhi_\aa  = \widetilde{\varPhi}_{\aa} 
\end{equation}
for each $\aa$, noting that $\cU_U$ is independent of the choice of phase that specifies $\varOmega_0$.  It is easy to check, using the CAR, that for all $\phi\in \cH_1$, $\cU_U Z(\phi)\cU_U^* = Z(U\phi)$, so that 
$\cU_U$ implements that automorphism $\alpha_U$ specified in \eqref{SECQUNTU}. Note that the phase is such that $\cU_U\varOmega_0 = \varOmega_0$. 

If $\{U_t\}_{t\in \R}$ is any one parameter group of unitaries on $\cH_1$, then $\alpha_{U_t}$ is a one parameter group of 
$*$-automorphisms of $\scA_{\cH_1}$.   This lifting of unitary groups acting on $\cH_1$ to unitary groups acting on $\scK$ is known as {\em second quantization}, \cite{Cook53,S56}.
The simplest example, $U_t\psi = e^{it}\psi$ for all $\psi\in \cH_1$,  is physically important. The corresponding group of automorphisms is called the 
{\em gauge group}, and denoted $\{\alpha_{{\rm G},t}\}_{t\in \R}$. 

It follows directly from \eqref{PHIKDEF}, \eqref{PHIKDEF2} and \eqref{PHIKDEF3} that 
\begin{equation}\label{NUMOPDIAG}
\cU_{e^{it}}\varPhi_\aa = e^{it|\aa|}\varPhi_{\aa} = e^{it\cN}\varPhi_\aa\ ,
\end{equation}
and hence  for all $t\in \R$ and all $A\in \scH_{\cH_1}$,
\begin{equation}\label{NUMOPDEF2B}
\alpha_{{\rm G},t}(A) = e^{it\cN}Ae^{-it\cN} .
\end{equation}
While the definition \eqref{NUMBEROPDEFI} of $\cN$ may have appeared to depend on the choice of the basis $\{\psi_1,\dots,\psi_N\}$, it is now clear that this is not the case because for any choice of the basis, $e^{it\cN}$ implements the gauge automorphism group, and $e^{it\cN}\varOmega_0 = \varOmega_0$.

It  follows from \eqref{NUMOPNT3I} that 
\begin{equation}\label{NUMOPDEF2I}
e^{it\cN}Z(\psi)e^{-it\cN} = Z(e^{it}\psi) \quad{\rm and\ hence}\quad  e^{it\cN}Z^*(\psi)e^{-it\cN} = Z^*(e^{-it}\psi)\ ,
\end{equation}
and this is another way  to prove \eqref{NUMOPDEF2B}.

Let $\{\psi_1,\dots,\psi_N\}$ be any orthonormal basis of $\cH_1$. 
For $\bb,\gg\in \{0,1\}^N$ define 
 \begin{equation}\label{POLYBAS}
 Z_{\bb,\gg} = \prod_{j=1}^N (Z^*(\psi_j))^{\beta_j}(Z(\psi_j))^{\gamma_j}\ ,
 \end{equation}
 where the terms in the product are arranged so that the indices increase from left to right. It is easy to see that $\{Z_{\bb,\gg}\ :\ \bb,\gg\in \{0,1\}^N\}$ is a basis of $\scA_{\cH_1}$.
 
 Define $|\bb| = \sum_{j=1}^N \beta_j$ and likewise for $|\gg|$.  Then
 \begin{equation}\label{POLYBASW1}
 \alpha_{{\rm G},t}(Z_{\bb,\gg} ) = e^{it(|\bb|-|\gg|)}Z_{\bb,\gg}~,
 \end{equation}
 so that each $Z_{\bb,\gg}$ is an eigenvector of $\alpha_{{\rm G},t}$ for each $t$.  Taking $t=\pi$, \eqref{POLYBASW1} becomes
 $\alpha_{{\rm G},\pi}(Z_{\bb,\gg} ) = e^{i\pi (|\bb|+|\gg|)}Z_{\bb,\gg}$
 since $e^{i2\pi |\gg|} =1$.  Therefore.
 \begin{equation}\label{POLYBASW2}
 \alpha_{{\rm G},\pi}(Z_{\bb,\gg} ) = \begin{cases}  \phantom{-}Z_{\bb,\gg} & |\bb|+|\gg|\ {\rm is\ even}\\ 
 -Z_{\bb,\gg} & |\bb|+|\gg|\ {\rm is\ odd} \end{cases}\ .
 \end{equation}
 Evidently $\alpha_{{\rm G},\pi}^2 = \alpha_{{\rm G},2\pi} = \one$, and  the spectrum  of $\alpha_{{\rm G},\pi}$ is $\{-1,1\}$.

 \begin{defn} $A\in \scA_{\cH_1}$ is called {\em even} in case $\alpha_{{\rm G},\pi}(A) = A$ and is called {\em odd} in case  $\alpha_{{\rm G},\pi}(A) = -A$. The automorphism 
 $\alpha_{{\rm G},\pi}$ is called the {\em principle automorphism} of the CAR algebra $\scA_{\cH_1}$.
 \end{defn}
 
We shall frequently encounter the unitary $e^{i\pi \cN}$ that implements the principle automorphism.  Define
 \begin{equation}\label{POLYBASW3}
 W := e^{i\pi \cN}\ ,
 \end{equation}
 and note that $W$ is self-adjoint since the eigenvalues of $\cN$ are integers. Then by 
\eqref{NUMOPDEF2B}, for all $A\in \scA_{\cH_1}$, $\alpha_{{\rm G},\pi}(A) = WAW$. In particular $W$ commutes with every even element of $\scA_{\cH_1}$, and it anti-commutes 
with every odd element of $\scA_{\cH_1}$. It can be used to ``paste together'' two CAR fields; see Appendix~\ref{GTPA} on the graded tensor product construction.

There is another useful way to express the generators of the unitary groups $\{\cU_{e^{itH}}\}_{t\in \R}$  produced by second quantization that builds directly on the construction of $\cN$.
 Let $X = \sum_{j=1}^n |\psi_j\rangle\langle \phi_j|$ be any operator on $\cH_1$, written as a sum of rank one projections, which is always possible using, for 
example, the singular value decomposition.  Define the operator $\widehat{X}$ on $\cK$ by $\widehat{X} = \sum_{j=1}^n Z^*(\psi_j)Z(\phi_j)$. The transform
\begin{equation}\label{AWTRANSFORM}
X = \sum_{j=1}^n |\psi_j\rangle\langle \phi_j| \ \mapsto \ \widehat{X} = \sum_{j=1}^n Z^*(\psi_j)Z(\phi_j)\ 
\end{equation}
was studied by Araki and Wyss \cite{AW64} who showed that 
that $\widehat{X}$ does not depend on the particular representation of $X$ as a sum or rank one operators, and that $\|\widehat{X}\| \leq \|X\|_1$ where the 
norm on the left is the operator norm on $\cB(\scK)$ and the norm on the right is the trace norm. Moreover, by a simple computation \cite{AW64}, for any two operators $X,Y$ on $\cH_1$
\begin{equation}\label{AWCOMM}
[\widehat{X},\widehat{Y}] = \widehat{[X,Y]}\ .
\end{equation}
For further results on the transform \eqref{AWTRANSFORM}, see \cite{A87,Lu76}.

 If $H$ is self-adjoint with spectral decomposition
${\displaystyle H = \sum_{j=1}^N \lambda_j |\psi_j\rangle\langle \psi_j|}$, then 
$$
\widehat{H} = \sum_{j=1}^N \lambda_j Z^*(\psi_j)Z( \psi_j)\ .
$$
Because $\{\psi_1,\dots,\psi_N\}$ is orthonormal it follows that $\{Z^*(\psi_1)Z( \psi_1), \dots,Z^*(\psi_N)Z( \psi_N)\}$ is a set of $N$ mutually commuting orthogonal projections. 
Taking $H$ to be the identity on $\cH_1$ yields the  number operator $\cN$; that is;  by $\cN = \widehat{\one}$, recovering the formula \eqref{NUMBEROPDEFI}.

For a $*$-automorphism of $\scA_{\cH_1}$, let $\widehat{\alpha}$ be its dual, acting on the set of states on $\scA_{\cH_q}$ by
$$
\widehat{\alpha}(\rho) = \rho \circ \alpha\ .
$$

\begin{defn} A quantum operation $\Phi$ on $\scA_{\cH_1}$ is {\em gauge covariant} in case for all $t$
\begin{equation}\label{GAUGECOVDEF}
\widehat{\alpha}_{{\rm G},t}\circ \Phi = \Phi \circ \widehat{\alpha}_{{\rm G},t}\ ,
\end{equation}
and is {\em gauge contravariant} in case for all $t$
\begin{equation}\label{GAUGECONVDEF}
\widehat{\alpha}_{{\rm G},t}\circ \Phi = \Phi \circ \widehat{\alpha}_{{\rm G},-t}\ ,
\end{equation}
\end{defn}

The dual $\widehat{\alpha}$ of every $*$-automorphism $\alpha$  of $\scA_{\cH_1}$ is a quantum operation. Therefore, we say that  $\alpha$ is {\em gauge covariant} if and only if 
$\widehat{\alpha}\circ \widehat{\alpha}_{{\rm G},t} = \widehat{\alpha}_{{\rm G},t} \circ \widehat{\alpha}$ for all $t\in \R$, and that 
$\alpha$ is {\em gauge contravariant} if and only if 
$\widehat{\alpha}\circ \widehat{\alpha}_{{\rm G},t} = \widehat{\alpha}_{{\rm G},-t} \circ \widehat{\alpha}$ for all $t\in \R$.

This brings us to the second way of viewing  the automorphism $\alpha_U$  specified in \eqref{SECQUNTU}  in terms  of the unitary $U$ on $\cH_1$. 
Write $U$ in the form $U = e^{-iH}$ where $H$ is self-adjoint, which is always possible by the Spectral Theorem. Then  for all $\psi\in \scH_1$, 
\begin{equation}\label{SecondQuant}
e^{i\widehat{H}} Z(\psi) e^{-i\widehat{H}}  = Z(e^{-iH}\psi) = Z(U\psi)\ .
\end{equation}

To prove \eqref{SecondQuant}, let $\{\psi_1,\dots,\psi_N\}$ be an orthonormal basis of $\cH_1$ consisting of eigenvectors of $H$ so that 
 $H = \sum_{j=1}^N\lambda_j |\psi_j\rangle\langle \psi_j|$. For $j=1,\dots,N$, define $Z_j := Z(\psi_j)$ so that $\widehat{H} = \sum_{j=1}^N\lambda_j Z_j^*Z_j^{\phantom{*}}$.
 
For $t\in \R$,  temporarily define $Z_k(t) := e^{it\widehat{H}} Z_k e^{-it\widehat{H}}$. Then by the CAR,
\begin{eqnarray*}
\frac{{\rm d}}{{\rm d}t} Z_k(t) = ie^{it\widehat{H}} [\widehat{H}, Z_k] e^{-it\widehat{H}} &=&  \sum_{j=1}^N i\lambda_j e^{it\widehat{H}} [Z_j^*Z_j^{\phantom{*}},Z_k]e^{-it\widehat{H}} \\
&=& i\lambda_k e^{it\widehat{H}} [Z_k^*Z_k^{\phantom{*}},Z_k]e^{-it\widehat{H}} = -i\lambda_k Z_k(t)\ .
\end{eqnarray*}
Therefore,
$$
e^{it\widehat{H}} Z(\psi_k) e^{-it\widehat{H}} = Z_k(t) = e^{it\lambda_k}Z_k(0) = Z(e^{-it H}\psi_k)\ ,
$$
and since this is true for each $k$, \eqref{SecondQuant} follows by linearity. 

The identity \eqref{SecondQuant} shows that  $\cU_U := e^{it\widehat{H}}$ is a unitary on $\scK$ implementing $\alpha_U$, and since $\widehat{H}\varOmega_0 = \varOmega_0$, $\cU_U\varOmega_0 = \varOmega_0$.

\begin{defn}\label{QuasiFreeDEF} An automorphism of $\scA_{\cH_1}$ is {\em quasi-free} in case for some self-adjoint operator $H$ on $\cH_1$
\begin{equation}\label{QuasiFree}
\alpha(A) = e^{i\widehat{H}} A e^{-i\widehat{H}}  
\end{equation}
for all $A\in \scA_{\cH_1}$.
\end{defn}

Since $\cN = \widehat{\one}$, it follows from \eqref{AWCOMM} that for all $H$, $[\widehat{H},\cN] =0$, and hence $e^{i\widehat{H}}$ commutes with $\cN$.  In particular, {\em every quasi-free automorphism of $\scA_{\cH_1}$ is gauge covariant}. 

Because $\cN$ is self-adjoint, the commutant of $\cN$ is a von Neumann sub-algebra of $\scA_{\cH_1}$  that was first studied by Araki and Wyss \cite{AW64}. 
\begin{defn} The {\em Araki-Wyss} algebra $\scG_{\cH_1}$ is the commutant of the number operator $\cN$.
\end{defn}

Evidently, for all $A\in \scG_{\cH_1}$ and all $t\in \R$, $\alpha_{{\rm G},t}(A) = A$. For this reason the Araki-Wyss algebra is sometimes called the
{\em gauge invariant CAR algebra}, abbreviated as the GICAR algebra.

\begin{defn} For any state $\rho$ on $\scA_{\cH_1}$, define a quadratic form $Q_\rho(\phi,\psi)$ on $\scH_1$ by 
\begin{equation}\label{SYMDEF}
Q_\rho(\phi,\psi) =\rho(Z^*(\psi)Z(\phi))\ ,
\end{equation}
and define a positive semi-definite operator $Q_\rho$ on $\cH_1$ by $\langle \phi,Q_\rho \psi\rangle  = Q_\rho(\phi,\psi)$. 
The operator $Q_\rho$ is called the {\em symbol} of $\rho$. If  $\rho(Z(\psi)) =0$ for all $\psi$ in $\cH_1$, it is called the {\em covariance} of $\rho$, or in the physics literature, the  {\em two-point function} of $\rho$. 
\end{defn}

For example, consider the vacuum state $\omega_0$ defined in \eqref{vacantivacst}. Since $\omega_0(Z^*(\psi)Z(\phi)) = 0$ for all $\psi,\phi$, $Q_{\omega_0} = 0$. Likewise, it is easy to see that
$Q_{\omega_1} = \one_{\cH_1}$. 

Every state $\rho$ on $\cB(\scH)$ is determined by its moments
$$
\rho(Z^*(\psi_1)\cdots Z^*(\psi_m)Z(\varphi_1) \cdots Z(\varphi_n))\ ,
$$
where $0 \leq m,n \leq N$ and $\{\psi_1,\dots,\psi_m,\varphi_1,\dots,\varphi_n\} \subset \cH_1$
because  the general element of $\cB(\scH)$ can be written as a linear combination of elements of this form.

\begin{defn} A state $\rho$ on $\scA_{\cH_1}$ is a
 {\em gauge invariant Gaussian state} in case all of its moments are determined by its second moments (those for $m=n =1)$, and hence by $Q_\rho$, in the following way:
\begin{equation}\label{GaussChar}
\rho (Z^*(\psi_1)\cdots Z^*( \psi_m ) Z( \phi_n)\cdots  Z( \phi_1)) = 
\begin{cases} 0 & m\neq n\\
\det\left( \langle  \phi_j ,Q_\rho   \psi_k\rangle \right)  & m=n\ .\end{cases}
\end{equation}
The set of all gauge invariant Gaussian states is denoted ${\mathfrak S}_{GIG}$. 
\end{defn}
It is evident from $Q_{\omega_0} = 0$ that $\omega_0\in {\mathfrak S}_{GIG}$, and almost as evident from 
$Q_{\omega_1} = \one_{\cH_1}$ that $\omega_1\in {\mathfrak S}_{GIG}$. Thus $ {\mathfrak S}_{GIG}$ is non-empty, and we shall soon explicitly describe its structure.

 Monomials of the form $Z^*(\psi_1)\cdots Z^*(\psi_n)Z(\varphi_1) \cdots Z(\varphi_n)$, with equal numbers of $Z^*$ and $Z$ terms are invariant under gauge automorphisms $\alpha_{{\rm G},t}$ by \eqref{NUMOPDEF2I} and \eqref{NUMOPDEF2B}. It follows that each $\rho\in {\mathfrak S}_{GIG}$
and each $t\in \R$,
\begin{equation}\label{GIPROP}
\widehat\alpha_{{\rm G},t}(\rho) = \rho
\end{equation}
This is the gauge invariant property of states in ${\mathfrak S}_{GIG}$, while their Gaussian nature lies in the fact that they are completely determined by their second moments. 
They are also physically natural, and arose in the BCS theory of superconductivity \cite{BCS57,B58} in a way that had noting to do with looking for a Fermionic analog of classical Gaussian probability measures.  Likewise, in the mathematical physics  literature they arose in the work of Shale and Stinespring \cite{SS65} on symmetry properties of states on CAR algebras.  Shale and Stinespring sought, and found, a Fermionic counterpart to a Bosonic construction of Segal \cite{S61}, which did explicitly involve Gaussian probability densities, but in \cite{SS65} they did not explicitly refer to Gaussian states.  For other early work on Gaussian states, see \cite{BV68,DA68,PS70}.

We now turn to the structure of  ${\mathfrak S}_{GIG}$.    By the Cauchy-Schwarz inequality associated to the GNS inner product $\langle A,B\rangle_{GNS,\rho} := \rho (A^*B)$,
\begin{equation}\label{QQQ}
|\langle \phi,Q_\rho \psi\rangle|^2   \leq \rho(Z^*(\phi)Z(\phi))\rho(Z^*(\psi)Z(\psi))  \leq \|\phi\|^2\|\psi\|^2\ .
\end{equation}
and it follows that $0 \leq Q_\rho \leq \one$.

Define $\scQ$ by
\begin{equation}\label{SQDEF}
\scQ := \{ Q\in \cB(\cH_1)\ :\  0 \leq Q \leq \one\}\ .
\end{equation} 
By \eqref{QQQ}, the symbol of every state belongs to $\scQ$. It turns out the every $Q\in \scQ$ is the symbol of a unique state in ${\mathfrak S}_{GIG}$. That is,
the states in ${\mathfrak S}_{GIG}$ stand in one-to-one correspondence with the operators $Q\in \scQ$. This is simplest to explain in terms of density matrices of states. 

Let $Z$ be a CAR field over $\cH_1$ that acts irreducibly on $\scK$, so that  $\scK$ is $2^N$ dimensional.  The {\em normalized trace} is the state $\tau$ on $\scA_{\cH_1}$ given by
\begin{equation}\label{NORMTRACEDEF}
\tau(A) = \frac{1}{2^N}\tr[A]\ .
\end{equation}
Every state $\rho$ on $\scA_{\cH_1}$ is represented by a density matrix, {\em also denoted} $\rho$, and taken {\em with respect to the normalized trace} so that for all $A\in \scA_{\cH_1}$,
$$
\rho(A) = \tau(A\rho)\ .
$$
It is convenient to use the same notation for states and their density matrices and we shall generally suppress the distinction where it is unlikely to lead to confusion.

Let $\alpha$ be any $*$-automorphism of $\scA_{\cH_1}$. Writing $\alpha$ as a unitary conjugation, and letting $\rho$ also denote the density matrix of $\rho$ with respect to $\tau$, for all $A\in \scA_{\cH_1}$
$$
\widehat{\alpha}(\rho)(A) = \tau(\rho \alpha(A)) = \tau(\rho \cU_\alpha^{\phantom{*}} A \cU_\alpha^*) = \tau( (\cU_\alpha^*\rho\cU_\alpha^{\phantom{*}}) A)\ .
$$
Therefore, $\widehat{\alpha}(\rho)$ has the density matrix $\cU_\alpha^*\rho \cU_\alpha^{\phantom{*}}$.  For example, consider the gauge group $\alpha_{{\rm G},t}$. If $\rho$ is any state (with density matrix also denoted by $\rho$), then 
$\widehat{\alpha}_{{\rm G},t}(\rho)$ has the density matrix  $e^{-it\cN}\rho e^{it\cN}$.  Hence a quantum operation $\Phi$ (now considered as is usual to be acting on density matrices)  is gauge covariant if and only if for all $t$, and all density matrices $\rho$,
$\Phi(e^{-it\cN}\rho e^{it\cN}) = e^{-it\cN}\Phi(\rho) e^{it\cN}$ which is equivalent to 
\begin{equation}\label{GAUGECOV1}
[\cN,\Phi(\rho)] = \Phi([\cN,\rho])\ .
\end{equation}
In the same way, one see that $\Phi$ is gauge contravariant  if and only if for all density matrices $\rho$ in $\scA_{\cH_1}$, $[\cN,\Phi(\rho)] = -\Phi([\cN,\rho])$. 
 
Fix a CAR field over $\cH_1$ acting irreducibly on $\scK$.
For $Q\in \scQ$, let ${\displaystyle Q = \sum_{j=1}^N \mu_j |\psi_j\rangle\psi_j|}$ be a 
 spectral decomposition of $Q$.  Define the density matrix $\rho_Q$ by 
  \begin{equation}\label{GAUSSRQ}
 \rho_Q  = \prod_{j=1}^N \bigl((1-\mu_j)2Z(\psi_j)Z^*(\psi_j) + \mu_j 2Z^*(\psi_j)Z(\psi_j)\bigr)\ ,
 \end{equation}
 noting that all of the factors commute so that the order is immaterial.  
 
Simple computations show that indeed the symbol of $\rho_Q$ is $Q$, and that the moments of $\rho_Q$ are given by \eqref{GaussChar}. Since the moments determine the state, there can be at most one state with such moments, and we have just displayed one. Hence $Q \mapsto \rho_Q$ is one-to-one from $\scQ$ onto ${\mathfrak S}_{GIG}$. We have already observed that
 states in ${\mathfrak S}_{GIG}$ are invariant under gauge transformations.  It follows that if $\rho\in {\mathfrak S}_{GIG}$, then its density matrix, also denoted $\rho$, belongs to $\scG_{\cH_1}$, and moreover, is invariant under gauge transformations:
 $ \alpha_{{\rm G},t}(\rho) = \rho$.
 
 More generally. consider any self-adjoint $K$ on $\cH_1$ and any $Q\in \scQ$. Then with $\rho_Q$ given by \eqref{GAUSSRQ}, it follows from \eqref{SecondQuant} that 
   \begin{equation}\label{GAUSSRQOP1}
e^{-it\widehat{K}} \rho_Q e^{it\widehat{K}}   = \prod_{j=1}^N ((1-\mu_j)2Z(e^{-itK}\psi_j)Z^*(e^{-itK}\psi_j) + \mu_j 2Z^*(e^{-itK}\psi_j)Z(e^{-itK}\psi_j))\ ,
 \end{equation}
 Therefore, $e^{-it\widehat{H}} \rho_Q e^{it\widehat{H}}\in {\mathfrak S}_{GIG}$ and
   \begin{equation}\label{GAUSSRQOP2}
 e^{-it\widehat{K}} \rho_Q e^{it\widehat{K}}  = \rho_{e^{-itK} Q e^{itK}}\ .
 \end{equation}

  Suppose for example that  $Q = \mu\one$ for some $\mu\in [0,1]$. Then   
 $$
 \rho_{\mu\one} = 2^N\prod_{j=1}^N \left( (1-\mu) Z(\psi_j)Z^*(\psi_j)  + \mu Z^*(\psi_j)Z(\psi_j)\right)\ .
 $$
 By \eqref{FIELDS1},   the density matrix $\rho_{\frac12\one}  = \one$, and hence the state $\rho_{\frac12\one}= \tau$. 
 It is easy to see that $\prod_{j=1}^N Z(\psi_j)Z^*(\psi_j) \Psi = 0$ unless $\Psi$ is a multiple of $\varOmega_0$, and that 
 $\prod_{j=1}^N Z^*(\psi_j)Z(\psi_j) \Psi = 0$ unless $\Psi$ is a multiple of $\varOmega_1$. Therefore, 
 the density matrix $\rho_{0}  = 2^N  |\varOmega_0\rangle\langle\varOmega_0|$, and hence the state $\rho_{0}= \omega_0$, 
 and  the density matrix $\rho_{\one}  = 2^N  |\varOmega_1\rangle\langle\varOmega_1|$, and hence the state $\rho_{\one}= \omega_1$, which is consistent with our earlier observation
$\omega_0$, $\omega_1$GIG states, and shows that  and $\tau$ is also a GIG state.  Segal was the first to emphasize that $\tau$ is a natural non-commutative analog of the normal Gaussian law of classical probability in his approach to non-commutative integration \cite{S53,S65}.
 
 There is another way to write \eqref{GAUSSRQ} that reveals another aspect of the Gaussian character of $\rho_Q$. Given $Q\in \scQ^\circ$, the interior of $\scQ$, define 
 \begin{equation}\label{HTOQFORM}
 H = -\log\left(\frac{Q}{\one - Q}\right)\ .
\end{equation}
Then 
 \begin{equation}\label{HTOQFORM2}
 \rho_Q  =  \frac{1}{\tau(e^{-\widehat{H}})}e^{-\widehat{H}}\ .
\end{equation}
Indeed, $H= -\sum_{j=1}^N \log\left(\frac{\mu_j}{1-\mu_j}\right)|\psi_j\rangle\langle \psi_j|$  if  $Q = \sum_{j=1}^N \mu_j |\psi_j\rangle\langle \psi_j|$ is a spectral decomposition of $Q$, simple computations show that \eqref{HTOQFORM2} is the same as \eqref{GAUSSRQ} for $Q\in \scQ^\circ$.  We summarize this discussion in a lemma.

 \begin{lm}\label{STATESYMCLM}   Let $Z$ be a CAR field over a finite dimensional Hilbert space  $\cH_1$, acting irreducibly on a Hilbert space $\scK$. 
 
 \smallskip
 
  \noindent{\it (1)} 
The map 
$\rho   \mapsto Q_\rho$ is one-to-one from  ${\mathfrak S}_{GIG}$ onto $\scQ$. 

 \smallskip

 \noindent{\it (2)}  The duals of quasi-free automorphisms map 
${\mathfrak S}_{GIG}$ into itself, and  the  change in the symbol is given by \eqref{GAUSSRQOP2}. 

\smallskip

 \noindent{\it (1)}  Let $H$ be self-adjoint on $\cH_1$.
Then $H$ and the symbol $Q$ of the gauge invariant Gaussian state $\rho := \frac{1}{\tau(e^{-\widehat{H}})}e^{-\widehat{H}}$ are related by \eqref{HTOQFORM}.

\end{lm} 
 
 \begin{defn} A quantum operation $\Phi$ on $\scA_{\cH_1}$ is a {\em gauge invariant Gaussian quantum operation} (GIG quantum operation) in case it is gauge covariant and maps 
 ${\mathfrak S}_{GIG}$
 into itself. An automorphism $\alpha$  is a {\em gauge invariant Gaussian quantum automorphism} in case its dual $\widehat{\alpha}$, considered a quantum operation, 
 is a GIG quantum operation.
 \end{defn}
 
 One obvious example of a GIG quantum operation is
 \begin{equation}\label{AWALGCE}
 \Upsilon(A) = \frac{1}{2\pi}\int_0^{2\pi} e^{it\cN} A e^{-it\cN} {\rm d}t = \lim_{T\to\infty}\frac{1}{2T}\int_{-T}^{T} e^{it\cN} A e^{-it\cN} {\rm d}t 
 \end{equation}
 which is the orthogonal projection onto $\scG_{\cH_1}$ in with respect to the inner product $\tau(A^*B)$, and as such, is the tracial conditional expectation onto $\scG_{\cH_1}$ in $\scA_{\cH_1}$.
 (See Appendix~\ref{CONDEXAPP}).  
 If $\rho_Q$ is the density matrix of the state in  ${\mathfrak S}_{GIG}$ with symbol $Q$, then evidently $\Upsilon(\rho_Q) = \rho_Q$. Consequently if $\Phi$ is any GIG quantum operation,
 $\Phi\circ \Upsilon$, $\Upsilon\circ \Phi$ and $\Upsilon\circ \Phi\circ\Upsilon$ are also GIG quantum operations on $\scA_{\cH_1}$, and they all  have the same restriction to 
 $\scG_{\cH_1}$ as $\Phi$.

The following theorem is fundamentally important in what follows. 

\begin{thm}[Araki-Wyss Theorem]\label{AWTHM} The complex linear span of ${\mathfrak S}_{GIG}$ is all of $\scG_{\cH_1}$. 
\end{thm}

A close relative of Theorem~\ref{AWTHM} is proved in \cite[Lemma 8]{AW64}. Theorem~\ref{AWTHM} is proved in Appendix~\ref{AWA}.

\begin{cl}\label{AWCOR}  Let $\scG_{\cH_1}^{{\rm s.a.}}$ denote the real linear space of self-adjoint operators in Ayaki-Wyss algebra $\scG_{\cH_1}$.
Any  linear map $\Lambda$  on $\scA_{\cH_1}$ that maps ${\mathfrak S}_{GIG}$ into itself  maps  $\scG_{\cH_1}^{{\rm s.a.}}$  into itself.  The real linear span of 
${\mathfrak S}_{GIG}$ is $\scG_{\cH_1}^{{\rm s.a.}}$,  and any real linear map $\Gamma$ on 
$\scG_{\cH_1}^{{\rm s.a.}}$ is completely determined by its action on  ${\mathfrak S}_{GIG}$.
\end{cl}

\begin{proof} By Theorem~\ref{AWTHM}, every $A\in \scG_{\cH_1}$ can be written as $A = \sum_{j=1}^n z_j \rho_{Q_j}$ where each $z_j = x_j+iy_j\in \C$ and each $\rho_{Q_j}\in {\mathfrak S}_{GIG}$. If $\Lambda$ is a linear map on $\scA_{\cH_1}$ that maps ${\mathfrak S}_{GIG}$ into itself, then 
$$
\Lambda(A) = \Lambda\left(\sum_{j=1}^n z_j \rho_{Q_j}\right) = \sum_{j=1}^n z_j \Lambda(\rho_{Q_j}) \in \scG_{\cH_1}\ .
$$

If $A$ is self-adjoint, then also $A  =A^*  = \sum_{j=1}^n \overline{z_j} \rho_{Q_j}$, and hence $A = \sum_{j=1}^n x_j \rho_{Q_j}$. Thus  $\scG_{\cH_1}^{{\rm s.a.}}$ is the real linear span of 
${\mathfrak S}_{GIG}$.  Finally, let $\Gamma$ is any real linear map on $\scG_{\cH_1}^{{\rm s.a.}}$, and let $A\in \scG_{\cH_1}^{{\rm s.a.}}$ so that $A$ can be written as
$A = \sum_{j=1}^n x_j \rho_{Q_j}$ wih each $x_j\in \R$. Then 
$$
\Gamma(A) = \sum_{j=1}^n x_j \Gamma(\rho_{Q_j}) \in \scG_{\cH_1}
$$
\end{proof}
 
Let ${\Phi}$ be any quantum operation on $\scA_{\cH_1}$ that maps ${\mathfrak S}_{GIG}$ into itself. By the first part of Corollary~\ref{AWCOR}, $\Phi$ maps 
$\scG_{\cH_1}$ into itself, and since $\Phi$ is positive, it is Hermitian, and hence maps  $\scG_{\cH_1}^{{\rm s.a.}}$ into itself. By the second part of Corollary~\ref{AWCOR}, the restriction of 
$\Phi$ as a real linear map to 
$\scG_{\cH_1}^{{\rm s.a.}}$ is completely determined by the action of $\Phi$ on ${\mathfrak S}_{GIG}$. But then so is the restriction of $\Phi$ as a complex linear map on 
$\scG_{\cH_1}$ because every real linear map on $\scG_{\cH_1}^{{\rm s.a.}}$ has a unique complex linear extension to $\scG_{\cH_1}$. 
In view of these considerations, Corollary~\ref{AWCOR} motivates the following definition:

\begin{defn}\label{SYNMAPDEF} Let $\Gamma$ be a real linear map on  $\scG_{\cH_1}^{{\rm s.a.}}$.  Let $\rho_Q$ denote (the density matrix of) the GIG state on 
$\scA_{\cH_1}$ with symbol $Q$.  The {\em symbol map of $\Gamma$} is the map 
$\gamma_{\Gamma}^{\phantom{I}}:\scQ \to \scQ$ given by 
\begin{equation}\label{SYNMAPDEF1} 
\langle \phi,\gamma_{\Gamma}^{\phantom{I}}(Q)\psi\rangle  = \tau(\Gamma(\rho_Q)(Z^*(\psi)Z(\phi))\ .
\end{equation}
\end{defn}

\begin{lm}\label{ARWYCOR} Let $\Phi$ be a quantum operation on $\scA_{\cH_1}$ that maps ${\mathfrak S}_{GIG}$ into itself. Then $\Phi\big|_{{\scG}_{\cH_1}}$, the restriction of $\Phi$
to $\scG_{\cH_1}$, 
 is completely determined by its symbol map. 
 
 Furthermore, let $\Phi$ be any quantum operation on $\scG_{\cH_1}$ that maps 
${\mathfrak S}_{GIG}$ into itself.
Then $\Phi$ has at least one extension to a GIG quantum operation on $\scA_{\cH_1}$, namely $\Phi\circ \Upsilon$. 
\end{lm} 

\begin{proof} Because of the one-to-one correspondence between GIG states and their symbols, and because of Corolary~\ref{AWCOR}, 
$\gamma_{\Phi}^{\phantom{I}}$, fully specifies the action of $\Phi$ on $\scG_{\cH_1}$. 
This proves the first statement.

Now consider a quantum operation on $\scG_{\cH_1}$ that maps ${\mathfrak S}_{GIG}$ into itself. Then $\Phi\circ\Upsilon$ is a quantum operation on $\scA_{\cH_1}$ that whose restriction to $\scG_{\cH_1}$ is $\Phi$.  Since $\Upsilon$ is gauge-covariant, and since each  $\alpha_{{\rm G},t}$ acts trivially on $\scG_{\cH_1}$, $\Phi\circ\Upsilon$ is gauge covariant. 
\end{proof}

The following questions arise: 

\smallskip

\noindent{$\bullet$} What is the structure of the set of quantum operations on $\scG_{\cH_1}$ that map 
${\mathfrak S}_{GIG}$ into itself?  

\smallskip

\noindent{$\bullet$} What is the structure of the set of quantum operations on $\scA_{\cH_1}$ that map 
${\mathfrak S}_{GIG}$ into itself?

\smallskip

\noindent{$\bullet$} Given a quantum operation $\Phi$ on $\scG_{GIG}$ that maps ${\mathfrak S}_{GIG}$ into itself,  what is the structure of the set of all GIG quantum operations on $\scA_{\cH_1}$ extend $\Phi$?    

\medskip

 To answer the first question, it suffices by  Lemma~\ref{ARWYCOR} to  determine the class of maps on $\scQ$ that arise as symbol maps of quantum operations on $\scG_{\cH_1}$ that map 
${\mathfrak S}_{GIG}$ into itself, and  then show how to construct $\Phi$, as a quantum operation on $\scG_{\cH_1}$ from the data determining its symbol map. 

Towards an answer to  the second question, note that by Lemma~\ref{ARWYCOR}, one can always extend a quantum operation $\Phi$ on $\scG_{\cH_1}$ that maps ${\mathfrak S}_{GIG}$ into itself to a GIG quantum operation on
$\scA_{\cH_1}$, by the extension provided there, namely $\Phi\circ \Upsilon$. 

However, $\Upsilon$ is not one-to-one; it has a large null-space (for ${\rm dim}(\cH_1) > 1$).  Later on, we shall be primarily concerned with continuous semigroups 
$\{\cP_t\}_{t\geq 0}$ of quantum operations, each of which maps ${\mathfrak S}_{GIG}$ into itself. The continuity requirement is that $\lim_{t\to 0}\cP_t = \one$, and it has the consequence that each $\cP_t$ is one-to-one. 
If $\{\cP_t\}_{t\geq 0}$ is such a group of quantum operations on $\scG_{\cH_1}$, then $\{\cP_t\circ \Upsilon\}_{t\geq 0}$ is a family of GIG quantum operations such that for all $s,t\geq 0$, $(\cP_t\circ \Upsilon)(\cP_t\circ \Upsilon) = 
\cP_{s+t}\circ \Upsilon$, but is not a continuous semigroup because $\lim_{t\to 0}\cP_t\circ \Upsilon = \Upsilon \neq \one$. Fortunately, we shall see that when $\Phi$ is a one-to-one  quantum operation on $\scG_{\cH_1}$, it always has one-to-one extensions to all of $\scA_{\cH_1}$.

There are a number of variations on these questions taking into the account the notions of gauge covariant and gauge contravariant quantum operations. 
The first work in this direction is a classic result of Hugenholtz and Kadison \cite{HK75} on the automorphisms of $\scA_{\cH_1}$ that map  ${\mathfrak S}_{GIG}$ into itself.  Before stating their result, we must introduce another fundamental automorphism of the CAR algebra, namely the {\em particle-hole automorphism}. 

While many treatments of the CAR are phrased from the beginning in terms of ``annihilation operators'' and ``creation operators'', which arise directly in the construction of 
Fock \cite{F32}, 
the canonical  anti-commutation relations  \eqref{FIELDS1} and \eqref{FIELDS2} are symmetric in $Z(\psi)$ and $Z^*(\phi)$. However $Z$ is conjugate linear, and hence 
$Z^*$ is linear. As emphasized by Araki \cite{A87}, this near symmetry is fundamentally important, and can be completed
by introducing one more piece of structure on $\cH_1$:  a complex conjugation $C$.

A complex conjugation $C$ on $\cH_1$  is a conjugate linear map on $\cH_1$ with the 
property that $C^2 = \one$. Regarding $\cH_1$ as a real Hilbert space $\cH_{1,\Re}$ with the inner product $\Re \langle \cdot,\cdot \rangle$, $C$ is  an orthogonal transformation 
with spectrum $\{-1,1\}$, and $C\psi = \psi$ if and only if $C(i\psi) = -i\psi$.  Hence  $\psi\mapsto i\psi$ is invertible from the $1$-eigenspace of $C$ to the $-1$-eigenspace of $C$, and 
both eigenspaces are $N$ dimensional. It follows that if $\{\varphi_1,\dots,\varphi_N\}$ is any orthonormal basis for the $1$-eigenspace of $C$ in $\cH_{1,\Re}$, then 
$\{\varphi_1,\dots,\varphi_N\}$ is also an orthonormal basis of the complex Hilbert space $\cH_1$.  It will be convenient to denote $C \psi$ by $\overline{\psi}$. 
Vectors $\psi$ in $\cH_1$ such that $\overline{\psi} = \psi$ are called {\em real}.  Given a self-adjoint operator $H$ on $\cH_1$, define $\overline{ H}$, its complex conjugate, by 
$$
\overline{H}\psi = \overline{ H \bar \psi}\ .
$$
Finally it is easy to see that if $\{\varphi_1,\dots,\varphi_N\}$ any real orthonormal basis of $\cH_1$, then for all operators $X$ on $\cH_1$, 
$$X^T := \sum_{j,k=1}^N \langle \varphi_j,X \varphi_k\rangle |\varphi_k\rangle\langle \varphi_j|$$
is independent of the particular real orthonormal basis. The operator $X^T$ is the {\em transpose} of $T$. By the invariance just mentioned, it depends only on the complex conjugation $C$. 
Moreover, if $X = \sum_{\ell=1}^r\sigma_\ell |\psi_\ell\rangle\langle \phi_\ell|$ is a singular value decomposition of $X$, then 
$$
X^T := \sum_{j,k=1}^N \sum_{\ell =1}^r \sigma_\ell \langle \varphi_j,\psi_\ell\rangle \langle \phi_\ell, \varphi_k\rangle |\varphi_k\rangle\langle \varphi_j| =  \sum_{\ell =1}^r \sigma_\ell 
 |\overline{\psi_\ell}\rangle\langle \overline{\phi_\ell}|\ .
$$
Therefore, an equivalent way to express the trasnposition map is
\begin{equation}\label{TRPOMP}
X = \sum_{\ell=1}^r\sigma_\ell |\psi_\ell\rangle\langle \phi_\ell|  \mapsto X^T = \sum_{\ell =1}^r \sigma_\ell 
 |\overline{\psi_\ell}\rangle\langle \overline{\phi_\ell}|\ .
\end{equation}

Complex conjugations arise physically in connection with time-reversal \cite{W32}. 
If $\overline{H} = H$, then $H$ is called {\em real}. In elementary quantum mechanics one encounters real Hamiltonians of the form $H = -\Delta + V$ acting on $L^2(\R^3)$, and for 
$\psi\in L^2(\R^3)$
we define $C \psi$ by $C\psi(x) = \overline{\psi(x)}$, the complex conjugate of the number $\psi(x)$. Then $H$ is real, and if $\psi(x,t)$ solves
${\displaystyle i\frac{\partial}{\partial t}\psi(x,t) = H \psi(x,t)}$, then $\overline{\psi(x,t)}$ satisfies ${\displaystyle i\frac{\partial}{\partial t}\overline{\psi(x,-t)} = H \overline{\psi(x,-t)}}$. 
(See  the discussion of the quantum free evolution in \cite{JM05}). 
Thus, complex conjugation is associated with time reversal, though in the presence of spin and magnetic fields, there are other effects of time reversal. For physical electrons, which are spin $\frac12$ particles, time reversal will also reverse the spin \cite{W32}. A clear discussion of time reversal and complex conjugation for the Dirac equation can be found in Section 5 of  \cite{Gr72}. (See especially equation (5.6)). For finite Fermion systems arising through the Jordan-Wigner transform applied to a qubit system \cite{LSM61}, a different complex conjugation will be natural. 
However, in what follows, the precise nature of the complex conjugation $C$ will not matter. Henceforth we assume that $\cH_1$ is equipped with a preferred complex conjugation $C$.

We now define the particle-hole automorphism. 
Given a CAR field over $\cH_1$, define the {\em particle-hole automorphism}, $\alpha_{{\rm ph}}$,  of $\scA_{\cH_1}$ by 
\begin{equation}\label{PHADEF}
\alpha_{{\rm ph}}(Z(\psi)) =  Z^*(\overline \psi)
\end{equation}
Since as is easily checked, $\psi\mapsto Z^*(\overline \psi)$ is a CAR field over $\cH_1$, there is a unitary $\cU_{{\rm ph}}$ such that for all $\psi\in \cH_1$,
\begin{equation}\label{PHADEFA1}
\cU_{{\rm ph}}Z(\psi)\cU^*_{{\rm ph}} = Z^*(\overline \psi)\ ,
\end{equation}
and hence $\alpha_{{\rm ph}}$ is indeed a $*$-automorphism.

It is easy to see that $\alpha^2_{{\rm ph}}= \one$, and consequently we may choose the phase so that  $\cU^*_{{\rm ph}} = \cU_{{\rm ph}}$.  Therefore, if $\rho$ is a state on $\scA_{\cH_1}$, the density matrix of 
$\widehat{\alpha}_{{\rm ph}}(\rho)$ is $\alpha_{{\rm ph}}$ applied to the density matrix of $\rho$.  For any $\psi\in \cH_1$,
\begin{equation}\label{PHSYMBM1}
\alpha_{{\rm ph}}(Z^*(\psi)Z(\psi))  = Z(\bar\psi)Z^*(\bar\psi) \quad{\rm and}\quad \alpha_{{\rm ph}}(Z(\psi)Z^*(\psi))  = Z^*(\bar\psi)Z(\bar\psi) 
\end{equation}
Thus $\alpha_{{\rm ph}}$ transforms the projector for ``state $\psi$ is occupied''  into the projector for ``state $\bar\psi$ is unoccupied'', and hence then name ``particle hole automorphism''. 
Because the particle-hole automorphism interchanges ``occupied'' and ``unoccupied'', it swaps the vacuum and the anti vacuum states.  Indeed,  it 
follows from the CAR that  $\alpha_{{\rm ph}}(\cN) = N\one - \cN$. Consequently,
$$
\widehat{\alpha}_{{\rm ph}}(\omega_0) = \omega_1 \quad{\rm and}\quad \widehat{\alpha}_{{\rm ph}}(\omega_1) = \omega_0\ .
$$
There is a more general expression of this. 

\begin{lm}\label{PHSYMBM} The symbol map of $\widehat{\alpha}_{{\rm ph}}$ is  $Q \mapsto \one -Q^T$  where $Q^T$ denotes the transpose of $Q$ computed using any orthonormal basis of $\cH_1$ consisting of real vectors.
\end{lm}

\begin{proof}
Applying \eqref{PHSYMBM1}  to the formula \eqref{GAUSSRQ} for  density matrices $\rho$ of states in 
${\mathfrak S}_{GIG}$, one sees that  for all $Q$ in $\scQ$, $\alpha_{{\rm ph}}(\rho_Q)\in {\mathfrak S}_{GIG}$, and that the symbol of $\alpha_{{\rm ph}}(\rho_Q)$ is 
$\one - Q^T$.
\end{proof}

While $\alpha_{{\rm ph}}$ maps ${\mathfrak S}_{GIG}$ into itself, it is not gauge covariant. Indeed for any three vectors $\psi_1$, $\psi_2$ and $\psi_3$ in $\cH_1$,
$$
\alpha_{{\rm ph}}\bigl(Z^*(\psi_1)Z^*(\psi_2)(Z(\psi_3)\bigr) = Z(\overline{\psi_1}) Z(\overline{\psi_2}) Z^*(\overline{\psi_3})\ .
$$
However, by \eqref{NUMOPDEF2I}
$$
\alpha_{{\rm G},t}\bigl(Z^*(\psi_1)Z^*(\psi_2)(Z(\psi_3)\bigr) = e^{-it}Z^*(\psi_1)Z^*(\psi_2(Z(\psi_3)$$
while
$$
\alpha_{{\rm G},t}\bigl(Z(\overline{\psi_1}) Z(\overline{\psi_2}) Z^*(\overline{\psi_3})\bigr) = e^{it} Z(\overline{\psi_1}) Z(\overline{\psi_2}) Z^*(\overline{\psi_3})\ .
$$
Therefore, $\alpha_{{\rm ph}}$ is not gauge covariant. In fact, a simple extension of this argument shows that $\alpha_{{\rm ph}}$ is {\em gauge contravariant}.  

A related operation will also figure in our results. 
Second quantization ``lifts'' the complex conjugation $C$ on $\cH_1$ to a complex conjugation $\cC$ on $\scK$: Let $\{\varphi_1,\dots,\varphi_N\}$ be any real orthonormal basis of $\cH_1$.  Then as in \eqref{PHIKDEF}, this induces an orthonormal basis $\{\varPhi_\aa\}_{\aa\in \{0,1\}^N}$ through
\begin{equation}\label{PHIKDEF5}
\varPhi_{\aa} := (Z^*(\varphi_1))^{\alpha_1} \cdots (Z^*(\varphi_N))^{\alpha_N}\varOmega_0\ 
\end{equation}
for all $\aa$. Define the conjugate-linear unitary $\cC:\cK\to \cK$ by
\begin{equation}\label{SQCOMCON} {\textstyle
\cC\left( \sum_{\aa \in \{0,1\}^N} z_\aa \varPhi_\aa \right) = \sum_{\aa \in \{0,1\}^N} \overline{z_\aa} \varPhi_\aa\ .}
\end{equation}
Evidently, $\cC$ is a complex conjugation on $\scK$, and hence self-adjoint.  It is easy to check that for all $\phi\in \cH_1$, 
$$
\cC Z(\phi)\cC = Z(\bar\phi)\ .
$$
(Note that $\phi\mapsto Z(\bar\phi)$ is not a CAR field on $\cH$ because it is linear, not conjugate linear.)

It  follows from \eqref{GAUSSRQ} that if $\rho_Q$ is the density matrix of the GIG state with symbol $Q$, 
  \begin{equation}\label{GAUSSRQCC}
 \cC \rho_Q \cC   = \prod_{j=1}^N ((1-\mu_j)2Z(\overline{\psi_j})Z^*(\overline{\psi_j}) + \mu_j 2Z^*(\overline{\psi_j})Z(\overline{\psi_j}))\ ,
 \end{equation}
 
 \begin{defn}\label{FERMTRADEF} The map $\Lambda_\cC$ on $\scA_{\cH_1}$ is defined by 
  $$\Lambda_\cC(A) := \cC A \cC\ .$$
  This map is called  the {\em transpose map} on $\scA_{\cH_1}$.
  \end{defn}
  
  \begin{lm}\label{FERMTRALM} The transpose map $\Lambda_\cC$ maps ${\mathfrak S}_{GIG}$ into itself, and the symbol map of 
 $\Lambda_\cC$  is $Q \mapsto \overline{Q} = Q^T$.  However, the transpose map is not CP. 
  \end{lm}
  
  \begin{proof} The fact that $\Lambda_\cC$ maps ${\mathfrak S}_{GIG}$ into itself is evident from \eqref{GAUSSRQCC}, from which it also follows
  that for $Q = \sum_{j=1}^N\mu_j|\psi_j\rangle\langle \psi_j|$, the symbol of $\Lambda_\cC(\rho_Q)$ is $\sum_{j=1}^N\mu_j|\bar\psi_j\rangle\langle \bar\psi_j|$,
  which is the transpose of $Q$  by \eqref{TRPOMP}.. 
  Moreover, for any self-adjoint  $X\in \scA$,  the matrix representing $\Lambda_\cC(X)$ with respect to the real basis $\{\varPhi_\aa\:\ \aa\in \{0,1\}^N\}$  is the transpose of the matrix 
  representing $X$ with respect to this basis. Therefore, $\Lambda_\cC$ is not CP.  
  \end{proof}

 The following fundamental characterization of quasi-free automorphisms (Definition~\ref{QuasiFreeDEF})    was proved in \cite{HK75}.

\begin{thm}[Hugenholtz and Kadison]\label{HugKadThm} Let $Z$ be a CAR field over $\cH_1$ acting irreducibly on $\scK$. Let $\alpha$ be an automorphism of $\scA_{\cH_1}$ such 
that 
$\widehat{\alpha}$ maps ${\mathfrak S}_{GIG}$ into itself. Then either $\widehat{\alpha}(\omega_0) = \omega_0$ or else  $\widehat{\alpha}(\omega_0) = \omega_1$. 
In the first case, $\alpha$ is a quasi-free automorphism $\alpha_U$ of $\scA_{\cH_1}$ induced by some unitary $U$ on $\cH_1$.  In the second case, $\alpha\circ \alpha_{{\rm ph}}$ 
is a quasi-free automorphism $\alpha_U$ of $\scA_{\cH_1}$ induced by some unitary $U$ on $\cH_1$.
\end{thm}

Because every quasi-free automorphism is gauge covariant, and because $\alpha_{{\rm ph}}$ is gauge contravariant, it follows that every automorphism $\alpha$ whose dual maps 
${\mathfrak S}_{GIG}$ into itself is either gauge covariant, which is the case of $\widehat{\alpha}(\omega_0) = \omega_0$, or 
gauge contravariant, which is the case of $\widehat{\alpha}(\omega_0) = \omega_1$.

In their 1975 paper \cite{HK75}, Hugenholtz and Kadison also gave examples of quantum operations that map ${\mathfrak S}_{GIG}$ into itself. These were constructed in a ``functorial'' 
manner, generalizing the  ``second quantization functor'' $U\mapsto \alpha_U$. They  constructed a map from contractions $S$ on $\cH_1$ into quantum operations on $\scA_{\cH_1}$ that reduces to the functorial map 
$U \mapsto \widehat{\alpha}_U$ when $U$ is unitary. The ingredients of their construction,  recalled in Section~\ref{EHKC}, are a unitary dilation of $S$ onto $\cH_1\oplus \cH_1$ 
and the vacuum state. They showed that their quantum operations map ${\mathfrak S}_{GIG}$ into itself. A few year later in 1979, Evans \cite{E79} generalized this construction 
using the same unitary dilation of $S$, but any GIG state $\rho_{Q_0}$ in place of $\omega_0$. As Evans showed, this construction results in a quantum operation that maps $\rho_Q$, the 
GIG state with symbol $Q$ to the GIG state with symbol
\begin{equation}\label{EVANSSYM}
SQS^* + (\one - SS^*)^{1/2}Q_0 (\one - SS^*)^{1/2}\ .
\end{equation}
As a symbol of a state in ${\mathfrak S}_{GIG}$, the operator specified in \eqref{EVANSSYM} must belong to $\scQ$. It follows that
\begin{equation}\label{EVANSSYM1}
0 \leq (\one - SS^*)^{1/2}Q_0 (\one - SS^*)^{1/2} \leq \one - SS^*\ ,
\end{equation}
and that for an appropriate choice of $Q_0\in \scQ$, $R := (\one - SS^*)^{1/2}Q_0 (\one - SS^*)^{1/2}$ can be any operator on $\cH_1$ satisfying $0 \leq R \leq \one - SS^*$.

In summary, given any contraction $S$ on $\cH_1$ and an operator $R$ on $\cH_1$ satisfying 
\begin{equation}\label{RSCOMPAT}
0 \leq R \leq \one - SS^*\ ,
\end{equation}
the construction of  Evans-Hugenholtz-Kadison yields  a quantum operation $\Phi_{S,R}$  with symbol map $\gamma_{\Phi_{S,R}}$ given by
\begin{equation}\label{EHKSM1}
\gamma_{\Phi_{S,R}}(Q) = SQS^* + R\ .
\end{equation}
Moreover, for every choice of $R$, the map $\Phi_{S,R}$ is an ``extension'' of the second quantization map $U\mapsto \widehat\alpha_U$ from unitaries on $\cH_1$ to $*$-automorphisms of $\scA_{\cH_1}$ in that 
$$
\Phi_{S,R}^\dagger(Z(\psi)) = Z(S\psi)
$$
for all $\psi\in \scH_1$, which extends \eqref{SECQUNTU}. When $S$ is a unitary, \eqref{RSCOMPAT} requires that  $R= 0$.  

\begin{defn}\label{COMPPAIR} A {\em compatible pair} $(S,R)$ for Evans-Hugenholtz-Kadison construction consists of a contraction $S$ on $\cH_1$ and $R\in \scQ$ such that \eqref{RSCOMPAT}
is satisfied, in which case $\Phi_{S,R}$ denotes the quantum operation produced by the Evans-Hugenholtz-Kadison construction.
\end{defn}

The Evans-Hugenholtz-Kadison construction is reviewed in Section~\ref{EHKC}, and it is shown there that for all compatible pairs $(S,R)$, $\Phi_{S,R}$ is gauge-covariant. 
From their construction, one immediately obtains a second family of examples by composing $\Phi_{S,R}$ with $\alpha_{{\rm ph}}$. Then by Lemma~\ref{PHSYMBM}, for all $\rho_Q\in {\mathfrak S}_{GIG}$, 
$\Phi_{S,R}\circ\alpha_{{\rm ph}}(\rho_Q)$
has the symbol
\begin{equation}\label{EHKSM2}
S^*(\one -Q^T)S + R\ .
\end{equation}
The symbol maps resulting from the Evans-Hugenholtz-Kadison construction, namely
\begin{equation}\label{EHKSM3}
\gamma_{\Phi_{S,R}}^{\phantom{I}}(Q) = S^*QS + R \quad{\rm and}\quad \gamma_{\Phi_{S,R}\circ \widehat{\alpha}_{{\rm ph}}}^{\phantom{I}}(Q) = S^*(\one -Q^T)S + R
\end{equation}
are {\em affine maps}. This is remarkable because the map $Q \mapsto \rho_Q$ is highly non-linear; as it must be: ${\mathfrak S}_{GIG}$ is far from convex when $N>1$.  
Indeed, a non-trivial  convex combination of two distinct  classical Gaussian probability densities on the real line is never Gaussian. The situation in ${\mathfrak S}_{GIG}$ 
is more subtle, and was fully clarified by Wolfe \cite{W75} who proved the following lemma:

\begin{lm}[Wolfe's Lemma]\label{WOLFELEM} Let $Q_0,Q_1\in \scQ$, and let $\rho_{Q_0}$ and $\rho_{Q_1}$ be the corresponding GIG states. Let $0 < \lambda < 1$. 
Then $(1-\lambda)\rho_{Q_0} + \lambda \rho_{Q_1}\in {\mathfrak S}_{GIG}$
 if and only if for some vector $\eta\in \cH_1$,
\begin{equation}\label{WOLFELEM1}
Q_{1} - Q_{0} = \pm |\eta\rangle\langle \eta|\ .
\end{equation}
\end{lm}

Wolfe used Lemma~\ref{WOLFELEM} to give a new proof of Theorem~\ref{HugKadThm}, and an extension of it to ``non gauge invariant Gaussian states''.  This broader class of Gaussian states is also of interest; see the recent papers \cite{B05,Pr08,Pr}, all in the finite dimensional context. Many of the questions discussed here in the gauge invariant case are also of interest in the more general setting, and we plan to return to these questions later.  Wolfe's formulation of Lemma~\ref{WOLFELEM} is slightly different; he works in the setting of general Gaussian states. However, a simple adaptation of his reasoning yields a proof of  Lemma~\ref{WOLFELEM} as stated here, and this  is included in Appendix~\ref{Wolfe} for the reader's convenience.

We now use Wolfe's Lemma to construct  quantum operations that map ${\mathfrak S}_{GIG}$ into itself but which have different symbol maps than those provided by  the Evans-Hugenholtz-Kadison construction.  

\begin{thm}\label{NONAFF} There exist quantum operations on $\scA_{\cH_1}$ that map ${\mathfrak S}_{GIG}$ into itself whose symbol maps are not affine. 
\end{thm}

\begin{proof}
Let $\{E_0,E_1\}$ be a binary POVM on $\scK$. That is, $E_0$ and $E_1$ are positive operators on $\scK$ such that $E_0+E_1 = \one$.  Let $Q_0,Q_1\in \scQ$ satisfy 
\eqref{WOLFELEM1} for some non-zero vector $\eta$  and some choice of the sign.

Define $\Phi:\scA_{\cH_1}\to \scA_{\cH_1}$ by 
\begin{equation}\label{NONAFFRX1}
\Phi(X) = \tau(E_0 X) \rho_{Q_0} + \tau(E_1 X) \rho_{Q_1} \ .
\end{equation}
Then $\tau(\Phi(X)) = \tau(E_0 X) + \tau(E_1 X) = \tau(X)$ so that $\Phi$ is trace-preserving. 
It is easy to write down a Krauss representation for $\Phi$: Let $\{\varPhi_1,\dots,\varPhi_{2^N}\}$ be an orthonormal basis of $\scK$  consisting of eigenvectors of $E_0$: $E_0 \varPhi_j = \lambda_j\varPhi_j$ for each $j$. 
For $1 \leq j,k \leq 2^N$, define
$$
V_{j,k} :=\lambda_k^{1/2} |\rho_{Q_0}^{1/2}\varPhi_j\rangle\langle \varPhi_k|\ .
$$
Then evidently
$$
\sum_{j,k=1}^{2^N} V_{j,k} X V_{j,k}^* = \tau(E_0 X) \rho_{Q_0}\ .
$$
The same sort of construction yields a Krauss representation for $X \mapsto \tau(E_1 X) \rho_{Q_1}$. 

Altogether, this yields a Krauss representation of $\Phi$, and shows that it is a quantum operation. 
Physically, it corresponds to performing a binary 
measurement with possible results $\{0,1\}$ such that if the result is $0$, the final  state is $\rho_{Q_0}$, and if the result is $1$, final state is $\rho_{Q_1}$.  
For any input state $\rho$, $\tau(E_0 \rho)$ and $\tau(E_1 \rho)$ are non-negative and sum to $1$, so that by Wolfe's Lemma and the choice of $Q_0,Q_1$, $\Phi(\rho)\in {\mathfrak S}_{GIG}$.  
Evidently the symbol map of $\Phi$, $\gamma_{\Phi}^{\phantom{I}}$, is given by
\begin{eqnarray}
\gamma_{\Phi}^{\phantom{I}}(Q) &=& \tau(E_0 \rho_Q)Q_0 + \tau(E_1 \rho_Q) Q_1\nonumber\\
&=& Q_0 + \tau(E_1\rho_Q)(Q_1 -Q_0) = Q_0 \pm \tau(E_0\rho_Q) |\eta\rangle\langle \eta|\ .\label{NONAFFRX2}
\end{eqnarray}

Since $E_0$ can be  any element of $\scQ$, and the complex span of $\scQ$ is all of $\cB(\cH_1)$,  if $Q \mapsto  \tau(E_0\rho_Q)$ were affine for all choice of $E_0$, then $Q\mapsto \rho_Q$ would be affine, which is not the case. 
\end{proof}

Despite the example of Theorem~\ref{HugKadThm}, it would be many years before anyone took up the problem of determining the structure of 
quantum operations, other than duals of automorphisms,  that map ${\mathfrak S}_{GIG}$ into itself.  While the corresponding problem in the context of the CCR
has attracted more attention (see \cite{GC02} and references therein), the first work on this problem in the context of the CAR that we are aware of is a 2008 paper \cite{DFP08} of 
Dierckx, Fannes and Pogorzelska. Their Proposition 8 (see \cite[Section 5.2]{DFP08}) asserts that a linear map $\Gamma$ on $\scG_{\cH_1}$ maps ${\mathfrak S}_{GIG}$ into 
itself if and only if its symbol map
$\gamma_{\Phi}^{\phantom{I}}$ has the form 
\begin{equation}\label{DFP}
\gamma_{\Phi}^{\phantom{I}}(Q) = R \pm S^*QS   \quad{\rm or}\quad \gamma_{\Phi}^{\phantom{I}}(Q) = R\pm S^*(\one-Q^T)S\ .  
\end{equation}
They further assert that in the case of the plus signs in \eqref{DFP} there exists a quantum operation $\Phi$ with such a symbol map if and only if 
$0 \leq R\leq S^*S$, while they do not address this issue in the case of minus signs. The second statement is correct: The Evans-Hugenholtz-Kadison  construction gives the existence of such
a quantum operation $\Phi$ under the conditions $0 \leq R\leq S^*S$, and Fannes and Rocha \cite{FR80} showed the necessity of this restriction on $R$ and $S$. 

However, the first statement cannot be correct, because by Theorem~\ref{NONAFF} symbol maps, even of quantum operations that map 
${\mathfrak S}_{GIG}$ into itself, are not always affine, while all of the maps in \eqref{DFP} are affine. 

The maps $\Phi$ specified in \eqref{NONAFFRX1}, which provide the counterexamples,    always have a non-trivial null space.  It turns out, as we prove here,  
that whenever a GIG quantum operation $\Phi$ is one-to-one, then its symbol map
$\gamma_{\Phi}^{\phantom{I}}$ is indeed affine. Note that if $\Phi$ is part of a quantum dynamical semigroup, so that it has the form $\Phi = e^{t\cL}$ for some 
Lindblad generator $\cL$ (see \cite{GKS76,Lin76})  then $\Phi$ will indeed be one-to-one. Indeed, the latter half of the present work is largely concerned with semigroups of GIG quantum operations, and 
our results requiring the one-to-one property will be applicable in this context. 

The main theorem proved in Section~\ref{STRUCTURE} is Theorem~\ref{STRUTHM} below. It refers to the GIG quantum operations provided by the Evans-Hugenholz-Kadison construction that is 
reviewed in Section~\ref{EHKC}. For present purposes, all one need to know is the for each compatible pair $(S,R)$, $\Phi_{S,R}$ is a well-defined GIF quantum operation whose symbol map is 
$Q \mapsto R + SQS^*$.  The theorem also refers to the particle hole automorphism $\alpha_{{\rm ph}}$ and to the map $\Lambda_\cC$ defined in \eqref{FERMTRADEF}.

\begin{thm}\label{STRUTHM} Let $\Phi$ be a one-to-one Hermitian map on $\scG_{\cH_1}$ that maps ${\mathfrak S}_{GIG}$ into itself. Then either 
\begin{equation}\label{STRUTHM1}
\gamma_{\Phi}^{\phantom{I}}(\one) \geq  \gamma_{\Phi}^{\phantom{I}}(0)\quad{\rm  or \ else}\quad \gamma_{\Phi}^{\phantom{I}}(0) \geq  \gamma_{\Phi}^{\phantom{I}}(\one)
\end{equation}
but not both. In the first case
 there is a contraction $S$ on $\cH_1$ and an operator $R\in \scQ$ such that
 $SS^* + R \leq \one$  and such that either for all $Q$
\begin{equation}\label{STRUTHM1A}
\gamma_{\Phi}^{\phantom{I}}(Q) = SQS^* + R
\end{equation}
and $\Phi\big|_{\scG_{\cH_1}} = \Phi_{S,R} \big|_{\scG_{\cH_1}}$ so that $\Phi\big|_{\scG_{\cH_1}}$ is CP,
or else for all $Q$,
\begin{equation}\label{STRUTHM1B}
 \gamma_{\Phi}^{\phantom{I}}(Q) = SQ^TS^* + R\ ,
\end{equation}
and and $\Phi\big|_{\scG_{\cH_1}} = \Phi_{S,R} \circ \Lambda_\cC\big|_{\scG_{\cH_1}}$, so that $\Phi\big|_{\scG_{\cH_1}}$ is co-CP, but not CP. 

In the second case, 
there is a contraction $S$ on $\cH_1$ and an operator $R\in \scQ$ such that
 $SS^* + R \leq \one$  and such that either for all $Q$
\begin{equation}\label{STRUTHM1C}
\gamma_{\Phi}^{\phantom{I}}(Q) = S(\one - Q^T)S^* + R
\end{equation}
and $\Phi\big|_{\scG_{\cH_1}} = \Phi_{S,R} \circ \widehat{\alpha}_{{\rm ph}}\big|_{\scG_{\cH_1}}$ so that $\Phi\big|_{\scG_{\cH_1}}$ is CP,
or else for all $Q$,
\begin{equation}\label{STRUTHM1D} 
 \gamma_{\Phi}^{\phantom{I}}(Q) = S(\one -Q)S^* + R\ ,
\end{equation}
and $\Phi\big|_{\scG_{\cH_1}} = \Phi_{S,R} \circ \widehat{\alpha}_{{\rm ph}}\circ \Lambda_\cC\big|_{\scG_{\cH_1}}$  so that $\Phi\big|_{\scG_{\cH_1}}$ is co-CP, but not CP. 
\end{thm}

Theorem~\ref{STRUTHM} is proved in Section~\ref{STRUCTURE}. Then in Section~\ref{EHKC}, the Evans-Hugenholz-Kadison construction is recalled in full detail, and some new results are proved under additional conditions on the symbol map. For instance, when the compatible pair $(S,R)$ specifying $\Phi_{(S,R)}$ is of the form $(e^{-H},R)$ with $H\geq 0$ and $[H,R]=0$,
an explicit diagonalization of $\Phi_{(S,R)}$ in terms of ``Fermionic Hermite polynomials'' is obtained; see Theorem~\ref{EIGTHM}.  The results up to this point provide a complete analysis the class of one-to-one quantum operations on $\scG_{\cH_1}$ that map ${\mathfrak S}_{GIG}$ into itself, and provide extensions of all of them to $\scA_{\cH_1}$. 

The rest of the paper is devoted to continuous semigroups of quantum operations $\{\cP_t\}_{t\geq 0}$ on $\scG_{\cH_1}$ and $\scA_{\cH_1}$ that map ${\mathfrak S}_{GIG}$ into itself. 
The property $\lim_{t\to 0}\cP_t = \one$ implies that each $\cP_t$ is one-to-one, and hence  Theorem~\ref{STRUTHM} is applicable.  

The  main result of 
Section~\ref{SEMIGRCAR} is Theorem~\ref{AWQDS} which establishes a one-to-one correspondence between semigroups $\{\cP_t\}_{t\geq 0}$ of quantum operations on 
$\scG_{\cH_1}$ that map ${\mathfrak S}_{GIG}$ into itself, and pairs $(G,A)$ where $G$ is a contraction semigroup generator on $\cH_1$, and $A$ is an operator on 
$\cH$ satisfying $0 \leq A \leq -G -G^*$.  

The main result of Section~\ref{GIGSGA}, Theorem~\ref{GENMAINTHM}, provides an extension of every 
semigroup of quantum operations on $\scG_{\cH_1}$ that maps ${\mathfrak S}_{GIG}$ into itself to all of $\scA_{\cH_1}$. This improves on a result of Evans \cite{E79} who proved that such an extension exists when $G$ commutes with $A$ (though his results is 
expressed somewhat differently).  Moreover, Theorem~\ref{GENMAINTHM}
  gives an explicit expression for the generator in terms of the data 
$(G,A)$ figuring in Theorem~\ref{AWQDS}.  

Under the commutativity condition of Evans, one can say more. In particular, the case in which $G=G^*$ and $G$ commutes with 
$A$, giving rise to the {\em Fermionic Mehler semigroups}, is then investigated in complete detail, and an explicit diagoanalization of the generators is obtained in this case; 
see Theorem~\ref{EIGTHMCOR} together with the earlier Theorem~\ref{EIGTHM}. 
These semigroups have a natural extension to all of $\scA_{\cH_1}$; see Theorem~\ref{EIGTHMCOR}. 
A simple argument using the Trotter Product Formula shows that the set of GIG semigroup generators is a convex cone; see Theorem~\ref{GIGCONE}. The same reasoning used in the proof of this allows us to compute
symbol maps of semigroups generated by $\cL_1+\cL_2$ when we know the symbol maps of the semigroups generated by  $\cL_1$ and by $\cL_2$; see Theorem~\ref{GIGCONE}.  
We then turn to  the problem of embedding GIG quantum operations in a GIG semigroup; Theorem~\ref{EMBED} gives a necessary condition. Section~\ref{INVSTATES}, presents some simple results on invariant states for GIG quantum operations.  

Several appendices present results that are used in the paper but for which there is not a convenient reference.  Appendix A briefly introduces the graded tensor product construction, 
Appendix B deals with  conditional expectations in the CAR. Appendix C contains a proof of Wolfe's Lemma in the gauge invariant case that is used here, and Appendix D contains a proof of the Araki-Wyss Theorem on the complex span of ${\mathfrak S}_{GIG}$ since what is actually proved in their paper is slightly different. 

We close this introduction with some words on the  analogous problem for Bosons; that is, for the canonical commutation relations (CCR) algebra. Also in this case there is a notion of Gaussian states, and gauge invariance. As here, Guassian states are determined by their symbols, and it is natural to seek a characterization of the quantum operations that map Gaussian states into Gaussian states for Bosons, and likewise in the gauge invariant case.
This has been studied by Giedke and Cirac \cite{GC02}, whose paper includes references to earlier work. However,  these authors consider maps $\Phi$ such that both 
$\Phi$ and $\one\otimes\Phi$, the later acting on a ``doubled'' CCR algebra, take Gaussian states into Gaussian states. One might call these ``completely Gaussian'' 
quantum operations. They find that such ``completely Gaussian''  quantum operations do have affine symbol maps in the Bosonic case. There is an analog of the 
doubling of variables in the Fermionic case, involving a graded tensor product $\one\otimes_F \Phi$; see Appendix~\ref{GTPA} and the discussion in Bravyi's paper \cite{B05}. 
However, proof of Theorem~\ref{NONAFF} yields quantum operations $\Phi$  that ${\mathfrak S}_{GIG}$ into itself,  but such that $\one \otimes_F \Phi$ does not have this property.  
Furthermore, Bravyi \cite{B05} considers a class of quantum operations having ``Gaussian kernels'' that map Gaussian states to Gaussian states (he does not consider gauge invariance) 
which have affine symbol maps. It would be interesting to find necessary and sufficient conditions under which quantum operations mapping ${\mathfrak S}_{GIG}$ into itself have affine 
symbol maps, but it is fortunate for our investigation of quantum dynamical semigroups that being one-to-one is a sufficient condition.

\section{Proof of the structure theorem}\label{STRUCTURE}

This section provides the proof of Theorem~\ref{STRUTHM}.  The proof is somewhat lengthly, and will be broken into several pieces which are results of independent interest. 

Throughout this section, we fix a linear map $\Phi$ on $\scG_{\cH_1}$ and  write $\gamma$  in place of $\gamma_\Phi$ to denote its symbol map.  
 Let $\rho_Q$ denote the GIG density matrix with symbol $Q$. By \eqref{HTOQFORM} and \eqref{HTOQFORM2},  the map $Q\mapsto \rho_Q$ is infinitely differentiable on $\scQ^\circ$, the interior of $\scQ$. 
 Then since $\Phi$ is linear, and the map sending $\rho$ to its symbol is linear, $\gamma$ is infinitely differentiable on $\scQ^\circ$. 
  Let $D\gamma_Q$ denote its derivative at $Q$. That is, for self-adjoint $H$ on $\cH_1$,
\begin{equation}\label{GAMMQDER}
D\gamma_Q (H) = \lim_{t\to 0}\frac{1}{t} \bigl(\gamma(Q+tH) - \gamma(Q)\bigr)\ .
\end{equation}

 We begin with a lemma 
  due to \cite{DFP08}, and we follow their proof. 

\begin{lm}\label{HOMDERLEM}  Let $Q_0\in \scQ^\circ$ and $\psi\in \cH_1$. Let  $a>0$ be such that  $Q_0 + t|\psi\rangle\langle\psi|\in \scQ^\circ$ for all $t\in (-a,a)$.
 Then there  is an $\eta\in \cH_1$ such that for all $t\in (-a,a)$,
\begin{equation}\label{HOMDERLEMA1}
D\gamma_{Q_0 + t|\psi\rangle\langle\psi|}(|\psi\rangle\langle\psi|) = D\gamma_{Q_0} (|\psi\rangle\langle\psi|) = \pm|\eta\rangle\langle \eta|\ ,
\end{equation}
and consequently
\begin{equation}\label{HOMDERLEMA2}
\gamma(Q_0 + t|\psi\rangle\langle\psi|) = \gamma(Q_0) \pm t |\eta\rangle\langle \eta|\ .
\end{equation}
Moreover, the sign in \eqref{HOMDERLEMA1} and  \eqref{HOMDERLEMA2}  is independent of $t$.
\end{lm}

\begin{proof}  For $t\in (-a,a)$ define $Q_t := Q_0 + t|\psi\rangle\langle\psi|$. It follows from Wolfe's Lemma that for all $0 \leq \lambda \leq 1$, 
$$
\rho_{(1-\lambda)Q_0 + \lambda Q_t} = (1-\lambda)\rho_{Q_0} + \lambda \rho_{Q_t}\ .
$$
Since $\Phi$ maps ${\mathfrak S}_{GIG}$ into itself, $\Phi(\rho_{(1-\lambda)Q_0 + \lambda Q_t})$, $\Phi(\rho_{Q_0})$ and $\Phi(\rho_{Q_t})$  all belong to ${\mathfrak S}_{GIG}$, 
and since $\Phi$ is linear,
\begin{equation}\label{HOMDERLEMA3}
\Phi(\rho_{(1-\lambda)Q_0 + \lambda Q_t}) = (1-\lambda)\Phi(\rho_{Q_0}) + \lambda \Phi(\rho_{Q_t})\ . 
\end{equation}
Therefore, by Wolfe's Lemma, 
\begin{equation}\label{HOMDERLEMA4}
 Q_{\Phi(\rho_{Q_t})} -  Q_{\Phi(\rho_{Q_0})} = \gamma(Q_t) - \gamma(Q_0) = \pm|\eta_t\rangle\langle \eta_t|
\end{equation}
for some $\eta_t\in \cH_1$.  Since the map $\rho\mapsto Q_\rho$ is linear, \eqref{HOMDERLEMA3} yields
$$
\gamma(Q_0 + \lambda t |\psi\rangle\langle \psi|) = (1-\lambda) \gamma(Q_0) + \lambda \gamma(Q_t)\ .
$$
We may rewrite this  as
${\displaystyle \gamma(Q_t) - \gamma(Q_0) = \frac{1}{\lambda}(\gamma(Q_0 + \lambda t |\psi\rangle\langle \psi|)) - \gamma(Q_0))}$.
Taking the limit $\lambda\downarrow 0$ yields
\begin{equation}\label{HOMDERLEMA5}
  \gamma(Q_t) - \gamma(Q_0) = tD\gamma_{Q_0} (|\psi\rangle\langle \psi|) \ .
\end{equation}
Combining \eqref{HOMDERLEMA4} and  \eqref{HOMDERLEMA5}
\begin{equation}\label{HOMDERLEMA6}
 tD\gamma_{Q_0} (|\psi\rangle\langle \psi|) =\pm |\eta_t\rangle\langle \eta_t| \ .
\end{equation}

It follows that $D\gamma_{Q_0} (|\psi\rangle\langle \psi|)$ is a multiple (possibly zero) of a rank one projection, and hence either 
$D\gamma_{Q_0} (|\psi\rangle\langle \psi|) \geq 0$ or else $D\gamma_{Q_0} (|\psi\rangle\langle \psi|)\leq 0$.  In the first case,
for some $\eta\in \cH_1$, $|\eta_t\rangle\langle \eta_t| = t |\eta\rangle\langle \eta|$ for all $t\in (-a,a)$, and then 
\eqref{HOMDERLEMA4} becomes \eqref{HOMDERLEMA2} with a positive sign. In the second case,
for some $\eta\in \cH_1$, $|\eta_t\rangle\langle \eta_t| = -t |\eta\rangle\langle \eta|$ for all $t\in (-a,a)$, and then 
\eqref{HOMDERLEMA4} becomes \eqref{HOMDERLEMA2} with a negative sign. .
\end{proof}

At this point in the analysis of \cite{DFP08}, the authors make a vague invocation of linearity and jump to their erroneous conclusion.  As we have seen, there are examples for which the map
symbol map $\gamma$ is not affine.  We now show that under the further assumption that $\Phi$ is one-to-one, the symbol map is in fact affine. 

\begin{lm}\label{HOMDERLEMP2} Under the further assumption that $\Phi$ is one-to-one,  its symbol map $\gamma$ is affine, and in particular, $D\gamma_Q$ does not depend on $Q\in 
\scQ^\circ$.
\end{lm}

\begin{proof} 
Because $\Phi$ is one-to-one, the map $\gamma$ is one-to-one. Indeed, $\gamma(Q) = \gamma(Q')$ if and only if $\Phi(\rho_Q) = \Phi(\rho_{Q'})$, and since $\Phi$ is one-to-one,
this entails $\rho_Q = \rho_{Q'}$, and then by the final part of Lemma~\ref{STATESYMCLM},  $Q = Q'$. 

Fix any $Q_0\in \{ Q\ :\ 0 < Q < \one\}$ and any non-zero $\psi,\phi\in \cH_1$. Let $r > 0$ be  such that $Q_0 + s|\psi\rangle\langle\psi| + t|\phi\rangle\langle\phi|\in 
\{ Q\ :\ 0 < Q < \one\}$ for $s^2 + t^2 < r^2$.  Define
\begin{equation}\label{HOMDERLEMP2A7}
\pm |\eta_t\rangle\langle\eta_t| :=  D\gamma_{Q_0+  t|\phi\rangle\langle\phi|} (|\psi\rangle\langle\psi|)\quad{\rm and}\quad 
\pm |\xi_s\rangle\langle\xi_s| := D\gamma_{Q_0+  s|\psi\rangle\langle\psi|} (|\phi\rangle\langle\phi|)
\end{equation}
which is justified because Lemma~\ref{HOMDERLEM}  says that for any $Q$, $D_Q$ applied to a rank one projector yields a multiple of  a rank one projector.

Using \eqref{HOMDERLEMA2} twice
\begin{eqnarray}
\gamma(Q_0+  s|\psi\rangle\langle\psi| + t|\phi\rangle\langle\phi|) &=&  \gamma\bigl(Q_0+  s|\psi\rangle\langle\psi|\bigr)  + t D\gamma_{Q_0+  s|\psi\rangle\langle\psi|} (|\phi\rangle\langle\phi|)\nonumber\\
&=& \gamma(Q_0) + sD\gamma_{Q_0}(|\psi\rangle\langle\psi|) + t D\gamma_{Q_0+  s|\psi\rangle\langle\psi|} (|\phi\rangle\langle\phi|)\nonumber\\
&=& \gamma(Q_0) \pm s|\eta_0\rangle\langle\eta_0| \pm t |\xi_s\rangle\langle\xi_s|\ ,\label{HOMDERLEMP2A9}
\end{eqnarray}
where the first sign is independent of $s$, and the second is independent of $t$. 
In the same way, 
\begin{eqnarray}
\gamma(Q_0+  s|\psi\rangle\langle\psi| + t|\phi\rangle\langle\phi|) &=&  \gamma\bigl(Q_0+  t|\phi\rangle\langle\phi|\bigr)  + s D\gamma_{Q_0+  t|\phi\rangle\langle\phi|} (|\psi\rangle\langle\psi|)\nonumber\\
&=& \gamma(Q_0) + tD\gamma_{Q_0}(|\phi\rangle\langle\phi|) + s D\gamma_{Q_0+  t|\phi\rangle\langle\phi|} (|\psi\rangle\langle\psi|)\nonumber\\
&=& \gamma(Q_0) \pm t|\xi_0\rangle\langle\xi_0| \pm s |\eta_t\rangle\langle\eta_t|\ , \label{HOMDERLEMP2A10}
\end{eqnarray}
where the first sign is independent of $t$, and the second is independent of $s$. 
By continuity then, the signs that appear in 
\eqref{HOMDERLEMP2A7} do not depend on $s$ or $t$ for $s,t$ sufficiently small.
It follows that for all such $s,t$,
\begin{equation}\label{HOMDERLEMP2A4}
s(|\eta_0\rangle\langle\eta_0|  - |\eta_t\rangle\langle\eta_t|)  = \pm t(|\xi_0\rangle\langle\xi_0| - |\xi_s\rangle\langle\xi_s|)\
\end{equation}
One possibility is that 
\begin{equation}\label{HOMDERLEMP2A5}
|\eta_t\rangle\langle\eta_t|= |\eta_0\rangle\langle\eta_0|\quad{\rm and}\quad |\xi_s\rangle\langle\xi_s| = |\xi_0\rangle\langle\xi_0|
\end{equation}
 for all sufficiently small $s,t$. 
In fact, this is the only possibility.  To see this, 
divide  both sides of  \eqref{HOMDERLEMP2A4} by $s$ and take the limit $s\to 0$  to conclude
$$
 |\eta_t\rangle\langle\eta_t| = |\eta_0\rangle\langle\eta_0| \pm t \frac{{\rm d}}{{\rm d}s}  |\xi_s\rangle\langle\xi_s|\bigg|_{s=0}  \ .
$$
The derivative on the right side  has rank less than $2$ if and only if for some real-valued  differentiable function $f(s)$,   $|\xi_s\rangle\langle\xi_s| = f(s) |\xi_0\rangle\langle\xi_0|$. 
Since the left side of \eqref{HOMDERLEMP2A4} is linear in $s$, it must be that  for some $a\in \R$, $f(s) = 1 +as$. 
In the same way one sees that for some 
$b\in \R$, $|\eta_t\rangle\langle\eta_t| = (1+bt)|\eta_0\rangle\langle\eta_0|$. Now \eqref{HOMDERLEMP2A9} and \eqref{HOMDERLEMP2A10} become
\begin{eqnarray}\label{HOMDERLEMP2A11}
\gamma(Q_0+  s|\psi\rangle\langle\psi| + t|\phi\rangle\langle\phi|)  &=& \gamma(Q_0) \pm s|\eta_0\rangle\langle \eta_0| + t(1+ as)|\xi_0\rangle\langle \xi_0|\nonumber\\
 &=& \gamma(Q_0) \pm s(1+tb)|\eta_0\rangle\langle \eta_0| + t|\xi_0\rangle\langle \xi_0|\ .
\end{eqnarray}
It follows that $b|\eta_0\rangle\langle\eta_0|    = a|\xi_0\rangle\langle\xi_0|$.

If either one of $a$ or $b$, is zero, then both are, and \eqref{HOMDERLEMP2A5} is satisfied. Assume this is not the case. Then with $c = |a/b|$,  
$|\eta_0\rangle\langle\eta_0| = c|\xi_0\rangle\langle\xi_0| $. Going back to \eqref{HOMDERLEMP2A9}, we can now rewrite it as
\begin{equation}\label{HOMDERLEMP2A10BB}
\gamma(Q_0+  s|\psi\rangle\langle\psi| + t|\phi\rangle\langle\phi|) = 
\gamma(Q_0) \pm \left(cs\pm t (1+ as)\right) |\xi_0\rangle\langle\xi_0|\ .
\end{equation}
This would mean that $\gamma(Q_0+  s|\psi\rangle\langle\psi| + t|\phi\rangle\langle\phi|) = \gamma(Q_0)$ whenever $t = \frac{cs}{1+as}$, and since $c\neq 0$, this would mean that $\gamma$ is not one-to-one. Therefore, under the hypothesis that $\Phi$ is one-to-one, \eqref{HOMDERLEMP2A5} is satisfied, and  \eqref{HOMDERLEMP2A9} becomes 
\begin{equation}\label{HOMDERLEMP2A11B}
\gamma(Q_0+  s|\psi\rangle\langle\psi| + t|\phi\rangle\langle\phi|) = 
\gamma(Q_0)  \pm s|\eta_0\rangle\langle\eta_0| \pm t |\xi_0\rangle\langle\xi_0|\ ,
\end{equation}
with signs independent of $s$ and $t$. 
Therefore, the Hessian of $\gamma$ is constant $\scQ^\circ$. 
\end{proof}

Define a real linear transformation $\Psi$ on the space of self-adjoint operators on $\cH_1$ by
$$
\Psi(H) = D_{\tfrac12 \one}(H)\ .
$$
By Lemma~\ref{HOMDERLEMP2}, the choice of the reference point $\tfrac12 \one$ in $\{Q \:\  0 < Q < \one\}$ is arbitrary, we make this choice for convenience.  
By Lemma~\ref{HOMDERLEMP2}, when $\Phi$ is one-to-one, there is a linear operator $\Psi$ acting on the space of self-adjoint operators on $\cH_1$ such that for all $Q_0,Q_1\in 
\scQ$
\begin{equation}\label{GQOPCHAR1}
\gamma(Q_1) = \gamma(Q_0) + \Psi(Q_1-Q_0)\ .
\end{equation}
This is because $\gamma$ is an affine function by Lemma~\ref{HOMDERLEMP2}.

Combining Lemma~\ref{HOMDERLEM} and Lemma~\ref{HOMDERLEMP2}, we have:

\begin{thm}\label{AFFINETHM} Let $\Phi$ be a one-to-one  linear map on $\scG_{\cH_1}$ that maps ${\mathfrak S}_{GIG}$ into itself. Then the symbol map $\gamma_{\Phi}^{\phantom{I}}$ is affine, and hence has the form
\begin{equation}\label{AFFINETHM1}
\gamma_{\Phi}^{\phantom{I}}(Q) = \gamma_{\Phi}^{\phantom{I}}\left(\tfrac12 \one\right) + \Psi\left(Q - \tfrac12 \one\right)
\end{equation}
for some real linear map $\Psi$ on the self-adjoint operators on $\cH_1$.  Moreover, $\Psi$ is one to one, and has the property that for all rank one projections $|\psi\rangle\langle\psi|$,
\begin{equation}\label{AFFINETHM2} 
\Psi(|\psi\rangle\langle\psi|) = \pm|\eta\rangle\langle\eta|\ .
\end{equation}
for some $\eta\in \cH_1$.
\end{thm}

Any real linear map on the self-adjoint operators on $\cH_1$ has a unique extension to a linear operator on $\cB(\cH_1)$ given by
$$
\Psi(H+iK) = \Psi(H) + i\Psi(K)
$$
for self-adjoint $H$ and $K$. Evidently, this extension, also denoted by $\Psi$, is Hermitian; that is for all $X$, $\Psi(X^*) = \Psi(X)^*$.

With   Theorem~\ref{AFFINETHM} at our disposal, it remains to  determine the structure of maps $\Psi$ satisfying
\eqref{AFFINETHM2}. The first question concerns the choice of signs in \eqref{AFFINETHM2}, which we deal with in the Lemma~\ref{SIGNSLEM} below. 
The troublesome aspect of the results proven so far is the possibility of both positive or negative signs, as in \eqref{HOMDERLEMA2}. The next lemma gives conditions under which the signs must be consistent, so that either $\Phi$ is positivity preserving, or $-\Phi$ is positivity preserving. This puts more tools at our disposal such as Kadison's Schwarz inequality for positive maps, which will be used in the proof of the main theorem of this section. 

\begin{lm}\label{SIGNSLEM} Let $\Psi$ be an Hermitian linear map on $\cB(\cH_1)$  that maps rank one projectors onto multiples of rank one projectors. That is, for every unit vector $\psi\in \cH_1$ there is a unit vector $\zeta\in \cH_1$ and a $t\in R$ such that $\Psi(|\psi\rangle\langle \psi|) = t|\zeta\rangle\langle \zeta|$. 
If there are vectors $\psi$ and $\varphi$ such that
\begin{equation}\label{POSNEG0}
\Psi(|\psi\rangle\langle\psi|) =  |\eta\rangle\langle\eta| > 0 \quad{\rm and} \quad  \Psi(|\varphi\rangle\langle\varphi|) =  -|\zeta\rangle\langle\zeta| < 0\ ,
\end{equation}
then there exists a vector $\xi\in \cH_1$ and a self-adjoint operator $A$ on $\cH_1$ such that for all operators $X$ on $\cH_1$,
\begin{equation}\label{POSNEG0B}
\Psi(X) = \tr[A X]|\xi\rangle\langle \xi|\ .
\end{equation}
Consequently, if $\Psi$ does not have the form \eqref{POSNEG0B}, in particular, if $\Psi$ is one-to-one, then either $\Psi$ is  positivity preserving or else $-\Psi$ is positivity preserving.    
\end{lm} 

\begin{proof} For all $t \geq 0$,
\begin{equation}\label{POSNEG1}
\Psi(|\psi+ t\varphi\rangle\langle \psi+ t \varphi|) = |\eta\rangle\langle \eta| - t^2 |\zeta\rangle\langle \zeta|  + t\Psi(|\psi\rangle\langle \varphi| + |\varphi\rangle\langle \psi|)\ .
\end{equation}
Define $f(t) = \tr[\Psi(|\psi+ t\varphi\rangle\langle \psi+ t \varphi|)]$. By hypothesis $f(0) > 0$ while for all $t$ sufficiently large, $f(t) < 0$. By the Intermediate Value Theorem, there exists $t_0\in (0,\infty)$ such that $f(t_0) = 0$, and then since $\Psi(|\psi+ t_0\varphi\rangle\langle \psi+ t_0 \varphi|)$ is self-adjoint and rank  one
\begin{equation}\label{POSNEG2}
\Psi(|\psi+ t_0\varphi\rangle\langle \psi+ t_0 \varphi|) = 0\ .
\end{equation}
Then by \eqref{POSNEG1}, $\Psi(|\psi\rangle\langle \varphi| + |\varphi\rangle\langle \psi|)  =  t_0 |\zeta\rangle\langle \zeta| - \frac{1}{t_0}|\eta\rangle\langle \eta|$, and hence for all $t>0$,
$$
\Psi(|\psi+ t\varphi\rangle\langle \psi+ t \varphi|) = \left(1- \frac{t}{t_0}\right)|\eta\rangle\langle \eta| + (t t_0 - t^2)|\zeta\rangle\langle \zeta|\ .
$$
By hypothesis, this  is rank one for all $t$, which is possible if and only  if  $\zeta$ is a multiple of $\eta$. Therefore, for some constant $c_\varphi$, \eqref{POSNEG0} becomes
\begin{equation}\label{POSNEG3}
\Psi(|\psi\rangle\langle\psi|) =  |\eta\rangle\langle\eta| > 0 \quad{\rm and} \quad  \Psi(|\varphi\rangle\langle\varphi|) =  -c_\varphi|\eta\rangle\langle\eta| < 0\ ,
\end{equation}
and this is true for any vector $\varphi$ for which $\Psi(|\varphi\rangle\langle\varphi|) \leq 0$. 
Then by symmetry we must also have $\Psi(|\phi\rangle\langle\phi|) = c_\phi |\eta\rangle\langle\eta|$ for all $\phi$ such that $\Psi(|\phi\rangle\langle\phi|) > 0$. 

Since $\Psi$ is Hermitian, it is determined by its action on self-adjoint operators $X$. Let $X = \sum_{j=1}^N \lambda_j |\phi_j\rangle\langle\phi_j|$ be an eigenvector-eigenvalue decomposition of $X$.  Then by what we have just proved,
$$
\Psi(X) = \sum_{j=1}^N \lambda_j c_j |\eta\rangle\langle\eta |\ .
$$
The map $X \mapsto \sum_{j=1}^N \lambda_j c_j $ is a linear functional because $\Psi$ is linear, and therefore it can be written in the form $X \mapsto \tr[AX]$ for some operator $A$ on 
$\cH_1$, and $A$ must be self-adjoint because otherwise $\Psi$ would not be Hermitian.   The final statement is clear. 
\end{proof}

By Theorem~\ref{AFFINETHM} and Lemma~\ref{SIGNSLEM}, the symbol map of a one--to--one quantum operation $\Phi$ that maps ${\mathfrak S}_{GIG}$ into itself is affine 
and its constant derivative $\Psi$  is either positivity  preserving or else $-\Psi$  is  positivity  preserving.
Since the symbol map of $\widehat{\alpha}_{{\rm ph}}$ is $Q \mapsto \one - Q^T$, where the transpose is computed with respect to the preferred complex conjugation 
$C$ on $\cH_1$, if the derivative $\Psi$ corresponding to $\Phi$ is not positivity preserving, then that of $\Phi\circ \widehat{\alpha}_{{\rm ph}}$ will be.  Thus, we need only consider the case that $\Phi$ is positivity preserving.  The following theorem completely specifies the structure of $\Psi$ in this case.

\begin{thm}\label{SPECFORM} Let $\Psi$ be a one--to--one  positive linear map on $\cB(\cH_1)$ such that for all $\psi\in \cH_1$, $\Psi(|\psi\rangle\langle\psi|) = |\eta\rangle\langle\eta|$ for some $\eta\in \cH_1$. 
Then there is a linear operator $A$ on $\cH_1$ such that for all operators $X$ on $\cH_1$, 
$$\Psi(X) = AXA^*\ ,$$
or else
for all  operators $X$ on $\cH_1$, 
$$\Psi(X) = A X^TA^*$$
where $T$ is the transpose associated to the complex conjugation on $\cH_1$. 
\end{thm}

\begin{proof}[Proof of Theorem~\ref{SPECFORM}] Let $\{\psi_1,\dots \psi_N\}$ be any real  basis of $\cH_1$.  Then by hypothesis, there are vectors $\{\eta_1,\dots,\eta_N\}$ such that for each $j=1,\dots,N$, 
\begin{equation}\label{STRUCTHMW1}
\Psi(|\psi_j\rangle\langle \psi_j|) = |\eta_j\rangle \langle \eta_j|\ , 
\end{equation}
but these vectors are not uniquely determined; 
we may replace each $\eta_j$ by $e^{i\theta_j}\eta_j$, $\theta_j\in [0,2\pi)$ without affecting  \eqref{STRUCTHMW1}.

Make some choice. We now claim that no matter how the phases have been chosen, for each $j\neq k$, there is a $w_{j,k}\in \C$, $|w| =1$ such that 
\begin{equation}\label{STRUCTHMW2}
\Psi(|\psi_j\rangle\langle \psi_k| +  |\psi_k\rangle\langle \psi_j|) = \overline{w_{j,k}} |\eta_j\rangle \langle \eta_k| +  w_{j,k} |\eta_k\rangle \langle \eta_j|\ . 
\end{equation}

To see this,   for  $t\in \R$  define $H_t := 
|\psi_k+ t\psi_j \rangle\langle \psi_k + t\psi_j|$.  
Then 
\begin{equation}\label{CLASSSTR1}
\Psi(H_t) = |\xi_t\rangle\langle \xi_t|
\end{equation}
where $\xi_0 = \eta_k$, and we fix the phase on $\xi_t$ by requiring  that $\langle \eta_k,\xi_t\rangle \geq 0$.  

$$
\frac{{\rm d}}{{\rm d}t}H_t\big|_{t=0} = \Psi(|\psi_1\rangle\langle \psi_j| + |\psi_j\rangle\langle \psi_1|) =: S\ .
$$
Then since
$(|\psi_k\rangle\langle \psi_j| + |\psi_j\rangle\langle \psi_k|)^2 = |\psi_k\rangle\langle \psi_k| + |\psi_j\rangle\langle \psi_j|$,
Kadison's inequality says 
$$ S^2 = \left(\Psi(|\psi_k\rangle\langle \psi_j| + |\psi_j\rangle\langle \psi_k|)\right)^2 \leq \Psi(|\psi_k\rangle\langle \psi_k| + |\psi_j\rangle\langle \psi_j|) \leq |\eta_k\rangle\langle \eta_k| +
|\eta_j\rangle\langle \eta_j| \ .$$
It follows that the range of $S$ is ${\rm span}(\{\eta_j,\eta_k\})$, which is two dimensional since otherwise $\Psi$ could not be one-to-one.  

Going back to \eqref{CLASSSTR1},
$S = |\xi_0\rangle\langle \xi_0'| +  |\xi'_0\rangle\langle \xi_0|  = |\eta_k\rangle\langle \xi_0'| +  |\xi'_0\rangle\langle \eta_k|$.
It then follows that for some $z,w\in \C$, $\xi'_0 = z\eta_k+ w\eta_j$.  Therefore,
$$
S =  2\Re(z)  |\eta_k\rangle\langle \eta_k| + \bar w |\eta_1\rangle\langle \eta_j| +  w |\eta_j\rangle\langle \eta_k|  \ .
$$
A similar argument with the roles of $\psi_k$ and $\psi_j$ interchanged yields that  for some other $z,w\in \C$, 
$$
S =  2\Re(z)  |\eta_j\rangle\langle \eta_j| + \bar w |\eta_k\rangle\langle \eta_j| +  w |\eta_j\rangle\langle \eta_k|  \ .
$$
By linear independence,  the coefficients of $ |\eta_k\rangle\langle \eta_k| $ and $ |\eta_j\rangle\langle \eta_j| $ must vanish, and so for some $w\in \C$, 
$$
S =   \bar w |\eta_k\rangle\langle \eta_j| +  w |\eta_j\rangle\langle \eta_k|  \ .
$$
This proves \eqref{STRUCTHMW2}. 

Now we leave the phase on $\eta_1$ as it was chosen, but will change the phases, if needed, on $\eta_j$ for $j>1$. 
We have proved that for each $j>1$, there is $w_j\in \C$, $|w_j|=1$ such that 
\begin{equation}\label{STRUCTHMW7}
\Psi(|\psi_1\rangle\langle \psi_j| + |\psi_j\rangle\langle \psi_1|) = \overline{w_j}|\eta_1\rangle\langle \eta_j| + w_j |\eta_1\rangle\langle \eta_1|\ .
\end{equation}
Replace $\eta_j$ by $w_j \eta_j$, and then 
\begin{equation}\label{STRUCTHMW3}
\Psi(|\psi_1\rangle\langle \psi_j| +  |\psi_1\rangle\langle \psi_j|) = |\eta_1\rangle \langle \eta_j| +   |\eta_j\rangle \langle \eta_1|\ . 
\end{equation}

We now claim that with this choice of the phases, 
\begin{equation}\label{STRUCTHMW4}
\Psi(|\psi_k\rangle\langle \psi_j| +  |\psi_k\rangle\langle \psi_j|) = |\eta_k\rangle \langle \eta_j| +   |\eta_k\rangle \langle \eta_j|\ 
\end{equation}
is satisfied for all $j,k$. 
To see this, compute using \eqref{STRUCTHMW2}:
\begin{eqnarray*}
\Psi(|\psi_1+ \psi_j +\psi_k\rangle \langle \psi_1+\psi_j+\psi_k|) &=& |\eta_1\rangle\langle \eta_1| +  |\eta_j\rangle\langle \eta_j| + |\eta_k\rangle\langle \eta_k| \\
&+ & |\eta_1\rangle\langle \eta_j| + |\eta_j\rangle\langle \eta_1| +  |\eta_1\rangle\langle \eta_k| + |\eta_k\rangle\langle \eta_1|\\
&+& \bar w |\eta_j\rangle\langle \eta_k| + w |\eta_k\rangle\langle \eta_j|\\
&=& |\eta_1+\eta_j+w\eta_k\rangle\langle \eta_1+ \eta_j + w\eta_k| \\
&+& (1-w)( |\eta_1\rangle\langle \eta_k| + |\eta_k\rangle\langle \eta_1|)
\end{eqnarray*}
Since $\Psi(|\psi_1+ \psi_j +\psi_k\rangle \langle \psi_1+\psi_j+\psi_k|)$ is rank one, it must be that $w=1$. Thus choosing the relative phases with respect to $\eta_1$ so that \eqref{STRUCTHMW3} is true for all $j$ ensures that \eqref{STRUCTHMW4} is true for all $j$ and $k$.

Now replacing $\psi_j$ by $i\psi_j$, the  argument  leading to \eqref{STRUCTHMW1} yields
$$
\Psi( -i|\psi_k\rangle\langle \psi_j| + i|\psi_j\rangle\langle \psi_k|) = \overline{z_{j,k}} |\eta_k\rangle\langle \eta_j| + z_{j,k}|\eta_j\rangle\langle \eta_k|
$$
for some $z_{j,k}\in \C$ such that $|z_{j,k}| =1$.
Since $\Psi$ is linear,
\begin{equation}\label{HK1}
\Psi( |\psi_k\rangle\langle \psi_j| - |\psi_j\rangle\langle \psi_k|) = i\overline{ z_{j,k}} |\eta_k\rangle\langle \eta_k| + iz_{j,k}|\eta_j\rangle\langle \eta_k|\ ,
\end{equation}
and hence
$$
\Psi( |\psi_k\rangle\langle \psi_j|) = \frac{1+i\overline{z_{j,k}}}{2} |\eta_k\rangle\langle \eta_j| + \frac{1+iz_{j,k}}{2}|\eta_j\rangle\langle \eta_k|
$$
Since $\Psi$ is positive, and hence Hermitian, $\Psi( |\psi_j\rangle\langle \psi_k|) = \Psi( |\psi_k\rangle\langle \psi_j|)^*$, and hence 
$$
\Psi( |\psi_j\rangle\langle \psi_k|) = \frac{1-i z_{j,k}}{2} |\eta_j\rangle\langle \eta_k| + \frac{1-i\overline{z_{j,k}}}{2}|\eta_k\rangle\langle \eta_j|\ .
$$
Now let $w\in \C$, $|w| =1$. Then
\begin{eqnarray*}
\Psi(|\psi_k+ w\psi_j\rangle \langle \psi_k+ w\psi_j|) &=& 
|\eta_k\rangle\langle\eta_k| +  |\eta_j\rangle\langle\eta_j| +\bar w \Psi(|\psi_k\rangle\langle \psi_j|) + w \Psi(|\psi_j\rangle\langle \psi_k|)\\
&=& 
|\eta_k\rangle\langle\eta_k| +  |\eta_j\rangle\langle\eta_j| +\bar w \Psi(|\psi_k\rangle\langle \psi_j|) + w \Psi(|\psi_j\rangle\langle \psi_k|)\\
&=& |\eta_k\rangle\langle\eta_k| +  |\eta_j\rangle\langle\eta_j| + \bar \alpha  |\eta_k\rangle\langle\eta_j|  + \alpha  |\eta_j\rangle\langle\eta_k| 
\end{eqnarray*}
where
$$
\alpha = \Re(w) + \Im(w)z_{j,k}\ .
$$
Since $|w| =1$, $|\alpha| =1$ if and only if $z_{j,k}=\pm i$, and unless $|\alpha| =1$, $\Psi(|\psi_k+ w\psi_j\rangle \langle \psi_k+ w\psi_j|)$ would not be rank one. 
For $z=i$, \eqref{HK1} becomes
\begin{equation}\label{HK2}
\Psi( |\psi_k\rangle\langle \psi_j| - |\psi_j\rangle\langle \psi_k|) =  |\eta_k\rangle\langle \eta_j| - |\eta_j\rangle\langle \eta_k|\ .
\end{equation}
Combining \eqref{STRUCTHMW4} with \eqref{HK2} yields
\begin{equation}\label{HK4}
\Psi( |\psi_k\rangle\langle \psi_j|) =  |\eta_k\rangle\langle \eta_j|\ ,
\end{equation}
and since $\Psi$ is Hermitian, this is the same as $\Psi( |\psi_j\rangle\langle \psi_k|) =  |\eta_j\rangle\langle \eta_k|$. 

For $z=-i$, \eqref{HK1} becomes
\begin{equation}\label{HK5}
\Psi( |\psi_k\rangle\langle \psi_j| - |\psi_j\rangle\langle \psi_k|) =  -|\eta_k\rangle\langle \eta_j| + |\eta_j\rangle\langle \eta_k|\ .
\end{equation}
Combining \eqref{HK5} with \eqref{HK2} yields
\begin{equation}\label{HK6}
\Psi( |\psi_k\rangle\langle \psi_j|) =  |\eta_j\rangle\langle \eta_k|\ .
\end{equation}
and since $\Psi$ is Hermitian, this is the same as $\Psi( |\psi_j\rangle\langle \psi_k|) =  |\eta_k\rangle\langle \eta_j|$.

This shows that all distinct $j,k$, either,
\begin{equation}\label{TWOALTS}
\Psi(|\psi_k\rangle\langle \psi_j|) = |\eta_k\rangle\langle \eta_j| \quad{\rm or}\quad  \Psi(|\psi_k\rangle\langle \psi_j|) =  |\eta_j\rangle\langle \eta_k|\ .
\end{equation}
We now claim that either for any $\ell$ either  $\Psi(|\psi_\ell\rangle\langle \psi_j|) = |\eta_\ell\rangle\langle \eta_j| $ for all $j$ or else 
$\Psi(|\psi_\ell\rangle\langle \psi_j|) = |\eta_j\rangle\langle \eta_\ell| $ for all $j$. Suppose not. Then evidently $N\geq 3$ and  for some $j$ and $k$,
$$\Psi(|\psi_\ell\rangle\langle \psi_j|) = |\eta_\ell\rangle\langle \eta_j| \quad{\rm while}\quad \Psi(|\psi_\ell\rangle\langle \psi_k|) = |\eta_k\rangle\langle \eta_\ell| \ .$$
There are two possibilities for $\Psi(|\psi_j\rangle\langle \psi_k|)$; either it equals $|\eta_j\rangle\langle \eta_k|$ or it equals $|\eta_k\rangle\langle \eta_j|$.
Either way, for $\ell\neq j,k$. 
\begin{eqnarray*}
\Psi(|i\psi_\ell + \psi_j + \psi_k\rangle \langle i\psi_\ell + \psi_j + \psi_k|) &=&  |\eta_\ell\rangle\langle \eta_\ell| + |\eta_j\rangle\langle \eta_j| +|\eta_k\rangle\langle \eta_k| \\
&+& i(|\eta_\ell\rangle\langle \eta_j| -|\eta_j\rangle\langle \eta_\ell|) 
+ i(|\eta_k\rangle\langle \eta_\ell| -  (|\eta_\ell\rangle\langle \eta_k|)\\
 &+& |\eta_j\rangle\langle \eta_k|
 + |\eta_k\rangle\langle \eta_j|\\ 
&=&| i\eta_\ell + \eta_j+\eta_k\rangle\langle i\eta_\ell + \eta_j+\eta_k|\\
 &+& 2i( |\eta_\ell\rangle\langle \eta_k| -|\eta_k\rangle\langle \eta_\ell| )\ ,
\end{eqnarray*}
which is not rank one, 
and this contradiction proves the claim. 

Now suppose that $\Psi(|\psi_1\rangle\langle\psi_2|) = |\eta_1\rangle\langle \eta_2|$. Then by what was just proved, $\Psi(|\psi_1\rangle\langle\psi_\ell|) = |\eta_1\rangle\langle \eta_\ell|$
for all $\ell$. Since $\Psi$ is Hermitian, $\Psi(|\psi_\ell\rangle\langle\psi_1|) = |\eta_\ell\rangle\langle \eta_1|$ for all $\ell$. Then once again by what was proved just above, 
$\Psi(|\psi_\ell\rangle\langle\psi_m|) = |\eta_\ell\rangle\langle \eta_m|$ for all $\ell$ and $m$. The same argument shows that if 
$\Psi(|\psi_1\rangle\langle\psi_2|) = |\eta_2\rangle\langle \eta_1|$, then $\Psi(|\psi_\ell\rangle\langle\psi_m|) = |\eta_m\rangle\langle \eta_\ell|$ for all $\ell$ and $m$.

Define an operator $A$ on $\cH_1$ by
$$
A\varphi  = \sum_{j=1}^N \langle \psi_j,\varphi\rangle \eta_j\ .
$$
The operator $A$ depends on the choice of  overall phase in $\{\eta_1,\dots,\eta_N\}$, but  $|A\varphi\rangle\langle A\varphi|$ does not.  Since
$$
|\varphi\rangle\langle \varphi| = \sum_{j,k=1}^N \langle \varphi,\psi_k \rangle \langle \psi_j,\varphi\rangle |\psi_j\rangle\langle \psi_k|\ ,
$$
if it is the case that $\Psi(|\psi_j\rangle\langle \psi_k|) = |\eta_j\rangle\langle \psi_k|$ for all $j,k$, then
$$
\Psi(|\varphi\rangle\langle \varphi|) = \sum_{j,k=1}^N \langle \varphi,\psi_k \rangle \langle \psi_j,\varphi\rangle |\eta_j\rangle\langle \eta_k| = |A\varphi\rangle\langle A\varphi|  = 
A|\varphi\rangle\langle \varphi|A^*\ .
$$
However, if it is the case that $\Psi(|\psi_j\rangle\langle \psi_k|) = |\eta_k\rangle\langle \psi_j|$ for all $j,k$, then since 
$$
|\bar \varphi\rangle\langle\bar  \varphi| = \sum_{j,k=1}^N \langle \varphi,\psi_k \rangle \langle \psi_j,\varphi\rangle |\psi_k\rangle\langle \psi_j|\ ,
$$
$$
\Psi(|\varphi\rangle\langle \varphi|) = \sum_{j,k=1}^N \langle \varphi,\psi_k \rangle \langle \psi_j,\varphi\rangle |\eta_j\rangle\langle \eta_k| = |A\bar \varphi\rangle\langle A\bar \varphi| =
A(|\varphi\rangle\langle \varphi|)^TA^*\ .
$$
\end{proof} 

\subsection{Proof of Theorem~\ref{STRUTHM}}

\begin{proof}[Proof of Theorem~\ref{STRUTHM}]  By Lemma~\ref{HOMDERLEMP2}, and Lemma~\ref{SIGNSLEM}, $\gamma(Q) = \gamma(0) + \Psi(Q)$ and either $\Psi$ is positivity preserving, or else $-\Psi$ is. In the first case, since the symbol of $\omega_0$ is $0$ and the symbol of $\omega_1$ is $\one$,  $\gamma(\omega_1) \geq \gamma(\omega_0)$, while in the second case,
$\gamma(\omega_1) \leq \gamma(\omega_0)$. If both were true, we would have $\gamma(\omega_1) = \gamma(\omega_0)$, but this is impossible since $\gamma$ is one-to-one. 

Then by Theorem~\ref{SPECFORM}, $\gamma(Q) = \gamma(0) + A Q A^*$ for some operator $A$ on $\cH_1$. Since the range of $\gamma$ lies in $\scQ$, $Q(0) \geq 0$ and 
$Q(0) + AA^* \leq \one$. In particular $A$ is a contraction. Write  $S$ to denote this contraction, and $R$ to denote $\gamma(0)$. Then $(S,R)$ is a compatible pair, and so the Evans-Hugenholtz-Kadison map $\Phi_{S,R}$ for this $R$ and $S$ has $\gamma$ as its symbol map. Because symbol maps completely determine linear maps on $\scG_{\cH_1}$, 
$$
\Phi\big|_{\scG_{\cH_1}} = \Phi_{R,S}\big|_{\scG_{\cH_1}}
$$
is a quantum operation operation on $\scG_{\cH_1}$ which has an extension, namely $\Phi_{S,R}$, to all of $\scA_{\cH_1}$ that is a GIG quantum operation. 

On the other hand, if $\gamma(\omega_1) \leq \gamma(\omega_0)$, then by the same reasoning, $\gamma(Q) = \gamma(0) - A Q^TA^*$. Again, $R := Q(0) \geq 0$ and 
$S := A$ is a contraction with $R + SS^* \leq \one$.  Then $\Phi_{S,R}\circ\Lambda_\cC \big|_{\scG_{\cH_1}}$ is a linear map on $\scG_{\cH_1}$ with this symbol, and hence 
$$
\Phi\big|_{\scG_{\cH_1}} = \Phi_{R,S}\circ\Lambda_\cC\big|_{\scG_{\cH_1}}\ .
$$
As we have seen, $\Lambda_\cC$ is the transpose map, and hence  $\Phi_{R,S}\circ\Lambda_\cC\big|_{\scG_{\cH_1}}$ is  co-CP.
Of course, the composition of a CP map and the transpose can be completely positive: consider the case in which the completely positive map is $X\mapsto \tau(X)\one$,
 which is even a GIG quantum operation. However,  this only happens in case the completely positive map, as in this example, is not one--to--one. 
 
More specifically, in our context,  suppose that $S$ is an invertible contraction. Then $\Phi_{S,R}$ for any compatible $R$ is invertible, and hence so is its Hilbert-Schmidt adjoint $\Phi_{S,R}^\dagger$. Then
 $(\Phi_{S,R}\circ \Lambda_\cC)^\dagger = \lambda_\cC \circ \Phi_{S,R}^\dagger$. Because $\Lambda_\C$ is not completely positive while $\Phi_{S,R}^\dagger$ is completely positive and invertible, $(\Phi_{S,R}\circ \Lambda_\cC)^\dagger$ is not completely positive, and hence neither is $(\Phi_{S,R}\circ \Lambda_\cC)^\dagger$. 
 
The remaining statements follow from this analysis using the fact that the symbol map of $\widehat{\alpha}_{{\rm ph}}$ is $Q \mapsto \one - Q^T$. 
\end{proof}

\section{The Evans-Hugenholtz-Kadison construction}\label{EHKC}

\subsection{The Evans-Hugenholtz-Kadison construction in general}
The first part of this section is expository, serving largely to establish notation and to provide a convenient summary of known results, but the theorems in the second part are new. 

Recall that a {\em compatible pair} $(S,R)$ consists of a contraction $S$ on $\cH_1$ and an operator $R$ on $\cH_1$ satisfying $0 \leq R \leq \one - SS^*$. This section is devoted to the 
 Evans-Hugenholtz-Kadison (EHK) construction of a GIG quantum operation $\Phi_{S,R}$ on the CAR algebra $\scA_{\cH_1}$ that maps ${\mathfrak S}_{GIG}$ into itself and has the symbol map
 $\gamma_{\Phi_{S,R}}(Q) = R + SQS^*$. 

We begin this section with a technical lemma that provides an alternate parameterization of the set of compatible pairs $(S,R)$ that is useful in the EHK construction. 

\begin{lm}\label{SRDEP} For every compatible pair $(S,R)$,  on $\cH_1$,  there is a unique $T\in \scQ$ such that ${\ker}(\one - SS^*) \subseteq{\rm ker}(T)$ and $R = (\one-SS^*)^{1/2}T(\one-SS^*)^{1/2}$. 
\end{lm} 

\begin{proof}  Let $r$ be the rank of $\one - SS^*$, which,  to avoid trivialities, we assume to be at least one, and let $\sum_{j=1}^r \lambda_j |\phi_j\rangle\langle \phi_j|$ be a spectral decomposition of $(\one - SS^*)$.  Then the {\em Moore-Penrose generalized inverse} of $(\one - SS^*)$, $(\one - SS^*)^+$, is defined by
\begin{equation}\label{MPGI}
(\one - SS^*)^+  =  \sum_{j=1}^r \lambda_j^{-1} |\phi_j\rangle\langle \phi_j|\ .
\end{equation}  
Define
\begin{equation}\label{EHKGIG3}
T := \left( (\one - SS^*)^+\right)^{1/2} R  \left( (\one - SS^*)^+\right)^{1/2}  \ 
\end{equation}
so that $R =  (\one -SS^*)^{1/2}T(\one -SS^*)^{1/2}$, and ${\ker}(\one -SS^*) \subseteq {\rm ker}(T)$. Since $R \leq (\one - SS^*)$, ${\ker}(\one -SS^*) \subseteq {\rm ker}(R)$.
Since  the inverse of  $(\one - SS^*)$ restricted to ${\ker}(\one -SS^*)^\perp$ is the restriction of $(\one - SS^*)^+$, the unique operator  $T$ on ${\ker}(\one -SS^*)^\perp$ such that 
$R =  (\one -SS^*)^{1/2}T(\one -SS^*)^{1/2}$  is given by \eqref{EHKGIG3}. It remains to show that $T\in \scQ$. 
Then for any $\psi\in \cH_1$, 
\begin{eqnarray*}
\langle\psi, T\psi\rangle &=& \sum_{j,k=1}^r \langle\psi,\phi_j\rangle\langle\phi_k,\psi\rangle \lambda_j^{-1/2}\lambda_k^{-1/2} \langle R^{1/2}\phi_j, R^{12}\phi_k\rangle \\
&\leq& \sum_{j=1}^r |\langle\psi,\phi_j\rangle|^2 \lambda_j^{-1} \langle \phi_j, R \phi_j\rangle\\
&\leq& \sum_{j=1}^r |\langle\psi,\phi_j\rangle|^2 \lambda_j^{-1} \langle \phi_j, (\one -SS^*) \phi_j\rangle \leq \|\psi\|^2\ .
\end{eqnarray*}
Therefore, $T\in \scQ$.  
\end{proof}

Given a CAR field $Z$ over $\cH_1$, extend it to the CAR field over $\cH_1\oplus \cH_1$ acting on $\scK\otimes \scK$ by
\begin{equation}\label{DOUBLE}
Z(\psi\oplus\eta) = Z(\psi)\otimes W + \one \otimes Z(\eta)\ 
\end{equation}
where $W = e^{i\frac {\pi}{2}\cN}$. (Readers who are not familiar with this sort of construction may wish to consult Appendix~\ref{GTPA} at this point.)
To simplify the notation below, we shall write
\begin{equation}\label{DOUBLE1}
A(\psi) := Z(\psi\oplus 0)\quad{\rm and}\quad B(\psi) := Z(0\oplus\psi)\ ,
\end{equation}
and  write $\scA$ and $\cB$ be the CAR sub-algebras of $\scA_{\cH_1\oplus\cH_1}$ generated by $\{A(\psi) : \psi\in \cH_1\}$ and 
$\{B(\psi) : \psi\in \cH_1\}$ respectively.

Let $S$ be a contraction on $\cH_1$. Let $V(S^*S)^{1/2}$ be a polar decomposition of $S$ where we take $V$ to be unitary in case $S$ has a non-trivial null-space. 
Then $(\one - SS^*)^{1/2} = V(\one - S^*S)^{1/2}V^*$, from which it easily follows that $(\one - SS^*)^{1/2}S = S(\one - S^*S)^{1/2}$. 
Therefore, the block-matrix operator 
\begin{equation}\label{DOUBLE2}
U_S := \left[\begin{array}{cc} -(\one - SS^*)^{1/2} & S \\  S^* &  (\one - S^*S)^{1/2}\end{array}\right]
\end{equation}
is a self-adjoint unitary on $\cH_1\oplus \cH_1$.  Let $J_1:\cH_1\to \cH_1\oplus \cH_1$ be the partial isometry $J_1(\psi) = (\psi,0)$. Likewise define $J_2(\psi) = (0,\psi)$. 
Then $S = J_1^* U_S J_2$.  This dilation of $S$ was used by Evans \cite{E79} who cites Halmos \cite{H67}.

Let $\cU_{U_S}$ be the second-quantization of $U_S$,  and  
let $\rho_T$ be the density matrix in $\scA$ of the GIG  state on $\scA$ with symbol $T$.  These are the basic ingredients of the EHK construction, which proceeds a follows: 
Let $\rho$ be any density matrix in $\cB$.   
Then $\rho_T$ commutes with $\rho$ so that $ \rho_T\rho$ is positive semi-definite. By the product property of the tracial state,  Corollary~\ref{SEGALPRODPROP}, $\tau(\rho_T\rho ) = \tau(\rho_T)\tau(\rho) =1$,   and hence 
$\rho_T\rho $ is a density matrix on $\scK \otimes \scK$. By \eqref{DOUBLE}, as an operator on $\scK \otimes \scK$, $\rho_T\rho$ may be identified with 
$\rho_T\otimes \rho$. Therefore, the operation $\rho \mapsto \rho_T\rho $ corresponds to the quantum operation of adding ancilla in state $\rho_T$.  
We obtain a quantum operation by composing this with two more quantum operations, namely the quasi-free automorphism $\widehat{\alpha}_{U_S}$ on 
$\scA_{\cH_1\oplus\cH_1}$ and the tracial conditional expectation onto $\scA$, ${\mathbb E}_{\cH_1\oplus 0,\tau}$, Literally, this gives us a quantum operation mapping $\cB = \scA_{0\oplus \cH_1}$ to $\scA = \scA_{\cH_1\oplus 0}$, but we may identify both with 
$\scA_{\cH_1}$ as generated by the original CAR field $Z$  with which we started.  

\begin{defn}\label{EHKCONDEF} Let $(S,R)$ be a compatible pair on $\cH_1$. Let $T_{S,R}\in \scQ$ be specified in terms of $S$ and $R$ as in Lemma~\ref{SRDEP}. Then $\Phi_{S,R}$ is the quantum operation on $\scA_{\cH_1}$ given by 
\begin{equation}\label{EHKCONSTR0}
 \Phi_{S,R}(\rho) :=  {\mathbb E}_{\cH_1\oplus 0,\tau}( \widehat{\alpha}_{U_S}(\rho_{T_{S,R}} \rho ))\ ,
 \end{equation}
 and is the {\em EHK quantum operation determined by the compatible pair} $(S,R)$. 
 \end{defn}
 
 To simplify the notation in what follows, we write  ${\mathbb E}_{\scA}$ in place of ${\mathbb E}_{\cH_1\oplus 0,\tau}$.
 This construction is due to Evans \cite{E79}, building on earlier work of Hugenholtz and Kadison \cite{HK75} who considered only the cases $T=0$ corresponding the the vacuum state, and $T=\one$ corresponding to the anti-vacuum state.  Slightly later than Hugenholtz and Kadison, but independently, Schrader and Uhlenbrock \cite{SU75}, building on work of Gross \cite{Gr72}, gave a related construction based on the tracial state. Their work was aimed at constructing ``doubly stochastic'' operations in the sense of Nelson \cite{N73}.  It seems that the idea of extending the ``second quantization functor'' from untaries to contractions was in the air in the early 1970's.

\begin{lm}\label{EHKGCOV} For every compatible pair $(S,R)$,  the EHK quantum operation $\Phi_{S,R}$  defined in \eqref{EHKCONSTR0} is gauge covariant.  
\end{lm}

\begin{proof} To simplify the notation, write $T$ for $T_{S,R}$. Since $\rho_T$ is gauge invariant, so that for all $t$  $\widehat{\alpha}_{e^{it} \one}(\rho_T) = \rho_T$, for all density matrices $\rho\in \cB$, 
$\rho_T\widehat{\alpha}_{e^{it }\one}(\rho) = \widehat{\alpha}_{e^{it} \one\oplus \one}(\rho_T \rho)$. 
Since $U_S$ commutes with $e^{it}\one\oplus \one$, 
$\widehat{\alpha}_{e^{it} \one}$ commutes with
$\widehat{\alpha}_{U_S}$. Therefore,
\begin{eqnarray*}
\Phi_{S,T}(\widehat{\alpha}_{e^{it }\one}(\rho)) &=&  {\mathbb E}_{\scA}( \widehat{\alpha}_{e^{it }\one\oplus \one}( \widehat{\alpha}_{U_S}(\rho_T\rho))) \\
&=&  \widehat{\alpha}_{e^{it} \one}\left({\mathbb E}_{\scA}( \widehat{\alpha}_{e^{it} 0\oplus \one}( \widehat{\alpha}_{U_S}(\rho_T\rho)))\right) \\
&=&  \widehat{\alpha}_{e^{it} \one}\left({\mathbb E}_{\scA}(( \widehat{\alpha}_{U_S}(\rho_T\rho))\right) \ .
\end{eqnarray*}
where in the final equality we have used the fact that $ \widehat{\alpha}_{e^{it} 0\oplus \one}$ is given by a unitary conjugation with a unitary acting only on the first factor in
$\cK\otimes \cK$, and then the partial cyclicity of the partial trace. 
\end{proof}

\begin{lm}\label{EHKGIG} For every compatible pair $(S,R)$,  the EHK quantum operation $\Phi_{S,R}$  defined in \eqref{EHKCONSTR0} 
 maps ${\mathfrak S}_{GIG}$ into itself, and its symbol map is
\begin{equation}\label{EHKGIG1}
Q\mapsto   R + SQ S^*\ .
\end{equation}
\end{lm}

\begin{proof} Let $T$ be given in terms of $S$ and $R$ by Lemma~\ref{SRDEP}.
If $\rho = \rho_Q\in {\mathfrak S}_{GIG}$, then $\rho_T \rho_Q$ is the GIG state on $\scA_{\cH_1\oplus \cH_1}$ with symbol $T\oplus Q$ and then 
$\cU_{U_S} \rho_T \rho_Q \cU_{U_S}$ is the GIG state with symbol $U_S(T \oplus Q ) U_S^*$.  Finally then, the symbol of  $ {\mathbb E}_\scA( \widehat{\alpha}_{U_S}(\rho_T \rho )) $ is
$J_1^*  U_S(T \oplus Q) U_S^* J_1$.  By the form of $U_S$, 
\begin{eqnarray*}
\langle \phi, J_1^*  U_S(Q \oplus T) U_S^* J_1\psi\rangle &=& \left\langle \left(\begin{array}{c} -(\one - SS^*)^{\frac12}\phi\\ S^*\phi\end{array}\right) ,\left[ \begin{array}{cc} T & 0 \\ 0 & Q\end{array}\right]
\left(\begin{array}{c} -(\one - SS^*)^{\frac12}\psi\\ S^*\psi\end{array}\right)\right\rangle \\
&=& \langle \phi, (\one -SS^*)^{\frac12}T(\one -SS^*)^{\frac12} \psi\rangle + \langle \phi,SQS^*\psi\rangle\ .
\end{eqnarray*}
This proves  \eqref{EHKGIG1}.
\end{proof}

Some properties of $\Phi_{S,R}$ are most readily understood in terms of its Hilbert-Schmidt adjoint, $\Phi_{S,R}^\dagger$ with respect to the inner product $\langle X,Y\rangle_{GNS,\tau} := \tau(X^*Y)$. 
The adjoint of ${\mathbb E}_{\scA}$ is the embedding of $\cA= \scA_{\cH_1\oplus 0}$ into $\scA_{\cH_1\oplus \cH_1}$. Since $U_S$ is 
self-adjoint, $\alpha_{U_S}$ is self-adjoint. Finally, the adjoint of the ``adjoining ancilla'' map $\rho \mapsto \rho_{T_{S,R}}\rho$ is the map 
$A \mapsto {\mathbb E}_{0\oplus\cH_1,\tau}( \rho_{T_{S,R}\oplus 0} A) = {\mathbb E}_{0\oplus\cH_1,\tau}(A\rho_{ T_{S,R}\oplus 0 } ) =
{\mathbb E}_{\cB}(A\rho_{ T_{S,R}\oplus 0 } )$.  
Altogether then, $\Phi_{S,R}^\dagger$
 is the 
completely positive unital (CPU) map on $\scA$ given by 
\begin{equation}\label{EHKGIG20}
\Phi_{S,R}^\dagger(X) :=  {\mathbb E}_{\cB}( \alpha_{U_S}(X)\rho_{ T_{R,S}\oplus 0} )
\end{equation}
where on the right side, $X$ is regarded as belonging to the sub-algebra $\scA = \scA_{\cH_1\oplus 0}$ of   $\scA_{\cH_1\oplus \cH_1}$.  Literally, this gives us a map from $\scA$ to $\cB$, but we may identify both with $\scA_{\cH_1}$, and do so. 

For example,
\begin{equation}\label{EHKGIG21}
\alpha_{U_S}(A(\psi)) = - A((\one -SS^*)^{1/2}\psi) + B(S^*\psi)\ ,
\end{equation}
and hence, $\Phi^\dagger_{S,R}(A(\psi)) = B(S^*\psi)$, or, after identifying $\scA$ and $\cB$ with $\scA_{\cH_1}$, 
\begin{equation}\label{PPEHK1}
\Phi^\dagger_{S,R}(Z(\psi)) = Z(S^*\psi)\ ,
\end{equation}
completely independent of $R$. 

Next, consider $A^*(\psi)A(\phi)$. Then using \eqref{EHKGIG21} and $R = (\one -SS^*)^{1/2}T_{S,R}(\one -SS^*)^{1/2}$,
\begin{eqnarray*}
\alpha_{U_S}(A^*(\psi)A(\phi)) &=& (- A^*((\one -SS^*)^{1/2}\psi) + B^*(S^*\psi))(- A((\one -SS^*)^{1/2}\phi) + B(S^*\phi))\\
&=& B^*(S^*\psi)B(S^*\phi) + A^*((\one -SS^*)^{1/2}\psi) A((\one -SS^*)^{1/2}\phi)\\
&+& {\rm terms\ that\ are\ odd\ in}\  \scA \ .
\end{eqnarray*}
Therefore,  ${\mathbb E}_{\cB}\bigl(\rho_{T_{R,S}} \alpha_{U_S}(A^*(\psi)A(\phi))\bigr) = B^*(S^*\psi)B(S^*\phi) + \langle \phi, R \psi\rangle\one$
and hence, after identifying $\scA$ and $\cB$ with $\scA_{\cH_1}$, 
\begin{equation}\label{PPEHK2}
\Phi^\dagger_{S,R}(Z^*(\psi)Z(\phi)) = Z^*(S^*\psi)Z(S^*\phi) + \langle \phi, R \psi\rangle\one\ .
\end{equation}

\subsection{An important special case}

The formulas simplify further if  $S = e^{-H}$, $H\geq 0$, and  $[H,R] =0$, or, what is the same thing, $[H,T_{S,R}] =0$. To simplify the notation below, we write $T$ for $T_{S,R}$.
Then $R = (\one - e^{-2H})T$, and we shall be studying $\Phi^\dagger_{e^{-H}, (\one - e^{-2H})T}$. 

When $S =e^{-H}$  for $H \geq 0$, 
 the dilation   $U_S$  of $S$ given by \eqref{DOUBLE2} simplifies to
\begin{equation}\label{DOUBLE2B}
U_{e^{-H}} := \left[\begin{array}{cc} -(\one - e^{-2H})^{1/2} & e^{-H} \\  e^{-H} &  (\one - e^{-2H})^{1/2}\end{array}\right]\ .
\end{equation}

Let $\{\psi_1,\dots,\psi_N\}$ be an orthonormal basis  of $\cH_1$ consisting of eigenvectors of $H$; $H\psi_j = \lambda_j\psi_j$ for $j=1,\dots,N$. 
In the orthonormal basis $\{\psi_1\oplus 0, 0\oplus\psi_1,\dots, \psi_N\oplus 0, 0\oplus\psi_N\}$, $U_{e^{-H}}$ has a block diagonal form with $2\times 2$ blocks, the $j^{{\rm th}}$ of which is 
$$
\left[\begin{array}{cc} -(1 - e^{-2\lambda_j})^{1/2} & e^{-\lambda_j} \\  e^{-\lambda_j} &  (1 - e^{-2\lambda_j})^{1/2}\end{array}\right]\ .
$$

The formula \eqref{EHKGIG20} becomes
\begin{equation}\label{EHKGIG20V}
\Phi_{e^{-H},(\one - e^{-2H}T)}^\dagger(X) =  {\mathbb E}_{\cB}( \alpha_{U_{e^{-H}}}(X)\rho_{ T \oplus 0} )
\end{equation}
where $X\in \scA_{\cH_1}$ is regarded as an element of $\scA = \scA_{\cH_1\oplus 0}$. 

The following is due to Evans, who proves a slightly more general result. (See the discussion around  (3.3) in \cite{E79}).

\begin{lm}\label{DETBAL} Under the assumption that $H$ and $T$ commute, there is a symmetry between the roles of $\scA_{0\oplus \cH_1}$ and $\scA_{\cH_1\oplus 0}$  in the formula 
\eqref{EHKGIG20V}.  If instead we regard $X\in \scA_{\cH_1}$ as an element of $\cB = \scA_{0\oplus \cH_1}$, we have the equivalent expression 
\begin{equation}\label{EHKGIG20V2}
\Phi_{e^{-H},(\one - e^{-2H}T)}^\dagger(X) =  {\mathbb E}_{\scA}( \alpha_{U_{e^{-H}}}(X)\rho_{ 0 \oplus T} )\ .
\end{equation}

Moreover, for all $X,Y\in \scA_{\cH_1}$, 
\begin{equation}\label{EHKGIG20V3}
\rho_T\bigl( X^*\Phi_{e^{-H},(\one - e^{-2H}T)}^\dagger(Y)\bigr)  = \rho_T\bigl((\Phi_{e^{-H},(\one - e^{-2H}T)}^\dagger(X))^* Y\bigr)\ .
\end{equation}
Consequently, $\Phi_{e^{-H},(\one - e^{-2H}T)}^\dagger$ is self-adjoint with respect to the GNS inner product induced by $\rho_T$. 
\end{lm} 

\begin{remark} The GNS inner product induced by $\rho_T$ is only non-degenerate when $0 < T < \one$, but \eqref{EHKGIG20V3} is valid for all $T\in \scQ$. 
\end{remark} 

\begin{proof} Under the assumption that $S= e^{-H}$, 
 \eqref{EHKGIG21} becomes 
 \begin{equation}\label{EHKGIG21A}
\alpha_{U_{e^{-H}}}(A(\psi)) = - A((\one -e^{-2H})^{1/2}\psi) + B(e^{-H}\psi)\ ,
\end{equation}
and we also have
\begin{equation}\label{EHKGIG21B}
\alpha_{U_{e^{-H}}}(B(\psi)) = A(e^{-H}\psi)  +  B((\one -e^{-2H})^{1/2}\psi)  \ .
\end{equation}
Apart from the the minus sign present in \eqref{EHKGIG21A}, but not present in \eqref{EHKGIG21B},
swapping roles of the  $\scA$ and $\cB$ algebras changes the one formula into the other.

However, this minus sign plays no role in the computation of the right side of either  \eqref{EHKGIG20V} or \eqref{EHKGIG20V2}:
When we apply $\alpha_{U_{e^{-H}}}$ to a monomial $X$ in $\cB$, we replace each term in the monomial according to 
\eqref{EHKGIG21B}.  To compute the right side of \eqref{EHKGIG20V2} we expand, and in each term of the sum, move all of the $B$ terms to the left of all of the $A$ terms, using the CAR, and take the expectation of the the resulting product of the $B$ terms in the state $\rho_T$.  Since $\rho_T$ is a GIG state, the terms $B((\one -e^{-2H})^{1/2}\psi)$ enter only quadratically,
 just as the terms $- A((\one -e^{-2H})^{1/2}\psi)$ enter only quadratically when computing the right side of \eqref{EHKGIG20V}, and hence the minus sign drops out, so that there is complete symmetry.  This proves that \eqref{EHKGIG20V2} is equivalent to \eqref{EHKGIG20V}.
 
 To prove \eqref{EHKGIG20V3}, we compute
 \begin{eqnarray}
\rho_T\left(X^* \Phi^\dagger_{e^{-H}, (\one - e^{-2H})T}(Y)\right) &=& \tau (X^*  {\mathbb E}_{\cB}( \alpha_{U_{e^{-H}}}(Y)\rho_{ T\oplus 0} )  \rho_{0\oplus T} )\nonumber\\
&=& \tau ( {\mathbb E}_{\cB}( X^*  \alpha_{U_{e^{-H}}}(Y)\rho_{T\oplus 0} )  \rho_{0\oplus T} )\nonumber\\
 &=& \tau(\rho_{T\oplus T} X^* \alpha_{U_{e^{-H}}}(Y))\ .\label{DETBAL1}
\end{eqnarray}
Because $U_{e^{-H}}^2 = \one$ and $H$ commutes with $T$, $U_{e^{-H}}(T\oplus T)U_{e^{-H}} = T\oplus T$, and therefore  $\alpha_{U_{e^{-H}}}(\rho_{T\oplus T})= \rho_{T\oplus T}$.  Returning to \eqref{DETBAL1}, and using $\alpha^2_{U_{e^{-H}}} =\one$,

\begin{eqnarray*}
\tau(\rho_{T\oplus T} X^* \alpha_{U_{e^{-H}}}(Y)) &=& \tau(\alpha_{U_{e^{-H}}}(\rho_{T\oplus T} X^* \alpha_{U_{e^{-H}}}(Y))) = 
\tau(\rho_{T\oplus T} \alpha_{U_{e^{-H}}}(X^*) Y)\\
&=&\rho_T\bigl((\Phi_{e^{-H},(\one - e^{-2H}T)}^\dagger(X))^* Y\bigr)\ .
\end{eqnarray*}
\end{proof}

There is more that one can say under these hypotheses. 
Let $\cH$ be any non-trivial subspace of $\scH_1$ that is invariant under $H$. Then, it is clear from \eqref{DOUBLE2B} that $\cH\oplus\cH$ and $\cH^\perp\oplus \cH^\perp$ are invariant under $U_{e^{-H}}$.  It follows that if $X\in \scA_{\cH\oplus \cH}$ and $Y\in \scA_{\cH^\perp\oplus \cH^\perp}$, then 
\begin{equation}\label{FACTORIZE1}
\alpha_{U_{e^{-H}}}(XY) =  \alpha_{U_{e^{-H}}}(X) \alpha_{U_{e^{-H}}}(Y) 
\end{equation}
where $\alpha_{U_{e^{-H}}}(X)\in \scA_{\cH\oplus \cH}$ and $\alpha_{U_{e^{-H}}}(Y) \in \scA_{\cH^\perp\oplus \cH^\perp}$.

Now suppose that $[H,T] =0$, and that $\cH$ is invariant under $T$ as well as $H$, so that $T = T_\cH \oplus T_{\cH^\perp}$ as an operator on $\cH \oplus\cH^\perp$. Then there is a factorization of the GIG density matrix $\rho_T$ of the form
$$
\rho_{T} = \rho_{T_\cH\oplus T_{\cH^\perp}} = \rho_{T_\cH}\rho_{T_{\cH^\perp}}
$$
and $\rho_{T_\cH} = {\mathbb E}_{\cH,\tau}(\rho_T)$ and $\rho_{T_{\cH^\perp}} = {\mathbb E}_{\cH^\perp,\tau}(\rho_T)$.  Note that since $\rho_{T_\cH}$ is even, it belongs to the commutant of
$\scA_{\cH^\perp}$, and likewise, $\rho_{T_{\cH^\perp}}$ belongs to the commutant of $\scA_{\cH}$.  In fact. 
regarded as an element of $\scA_{\cH\oplus\cH}$, $\rho_{\cH}$ belongs to the 
commutant of $\scA_{\cH^\perp\oplus\cH^\perp}$. Likewise, $\rho_{T_{\cH^\perp}}$ belongs to the commutant of $\scA_{\cH\oplus\cH}$.
Then from \eqref{FACTORIZE1}, 
\begin{equation}\label{FACTORIZE2}
\rho_T \alpha_{U_{e^{-H}}}(XY) =  \rho_{T_\cH}\alpha_{U_{e^{-H}}}(X) \rho_{T_{\cH^\perp}}\alpha_{U_{e^{-H}}}(Y) \ .
\end{equation}

This factorization leads easily  to a factorization property for $\Phi^\dagger_{e^{-H},(\one-e^{-2H})T}$:

\begin{thm}\label{FACTORIZTHM}  Let $H$ be positive semidefinite on $\cH_1$, and let $T\in \scQ$ be such that ${\rm \ker}(H) \subseteq{\rm ker}(T)$. Let $\cH$ be a non-trivial subspace of $\cH_1$ invariant under both $H$ and $T$. Let $X\in \scA_{\cH}$ and $Y\in \scA_{\cH^\perp}$.
Then
\begin{equation}\label{FACTORIZTHM1}
\Phi^\dagger_{e^{-H},(\one-e^{-2H})T}(XY) = \Phi^\dagger_{e^{-H},(\one-e^{-2H})T}(X) \Phi^\dagger_{e^{-H},(\one-e^{-2H})T}(Y)\ .
\end{equation}
\end{thm}

\begin{proof} Identify $X\in \scA_{\cH}$ with a element of $\scA_{\cH\oplus 0}$ in the obvious way. Likewise, identify $Y\in \scA_{\cH^\perp}$ with a element of $\scA_{\cH^\perp\oplus 0}$.
By definition, 
$$
\Phi^\dagger_{e^{-H},(\one-e^{-2H})T}(XY)  = {\mathbb E}_{0\oplus \cH_1,\tau}(\rho_T  \alpha_{U_{e^{-H}}}(XY))
$$
Then by \eqref{FACTORIZE2} and then Theorem~\ref{SEGCOND}, 
\begin{eqnarray*}
{\mathbb E}_{0\oplus \cH_1,\tau}(\rho_T  \alpha_{U_{e^{-H}}}(XY)) &=& {\mathbb E}_{0\oplus \cH_1,\tau}(\rho_{T_\cH}\alpha_{U_S}(X) \rho_{T_{\cH^\perp}}\alpha_{U_S}(Y) )\\
&=&{\mathbb E}_{0\oplus \cH,\tau}\left(\rho_{T_{\cH}}\alpha_{U_S}(X)\right)  {\mathbb E}_{0\oplus \cH^\perp,\tau}\left( \rho_{T_{\cH^\perp}}\alpha_{U_S}(Y) \right)
\end{eqnarray*}

Now we claim that 
\begin{equation}\label{FACTORIZPF2}
{\mathbb E}_{0\oplus \cH_1,\tau}{\mathbb E}_{\cH\oplus \cH,\tau} = {\mathbb E}_{0\oplus \cH,\tau}\ 
\end{equation}
This is a consequence of Corollary~\ref{COMCONDEXG}.
Because the symbol of the tracial state $\tau$ is $Q = \frac12\one$,  the requirement of invariance under the symbol $Q$  is automatically satisfied. We apply the Corollary in $\cH_1\oplus \cH_1$ to the pair of subspaces $\cH\oplus \cH$ and $0\oplus \cH_1$. The orthogonal projections onto these subspaces clearly commute, and their intersection is $0\oplus \cH$. Therefore,
\eqref{FACTORIZPF2} is a consequence of Corollary~\ref{COMCONDEXG}.

Then since $\rho_{T_{\cH}}\alpha_{U_S}(X)\in \scA_{\cH\oplus\cH}$ so that $\rho_{T_{\cH}}\alpha_{U_S}(X) =  {\mathbb E}_{\cH\oplus \cH,\tau} (\rho_{T_{\cH}}\alpha_{U_S}(X))$,
$$
{\mathbb E}_{0\oplus \cH,\tau}\left(\rho_{T_{\cH}}\alpha_{U_S}(X)\right)  = {\mathbb E}_{0\oplus \cH_1,\tau}\left(\rho_{T_{\cH}}\alpha_{U_S}(X)\right) 
= \Phi^\dagger_{e^{-H},(\one-e^{-2H})T}(X) \ .
$$
The same sort of reasoning yields
$$
 {\mathbb E}_{0\oplus \cH^\perp,\tau}\left( \rho_{T_{\cH^\perp}}\alpha_{U_S}(Y) \right) =  {\mathbb E}_{0\oplus \cH_1,\tau}\left( \rho_{T_{\cH^\perp}}\alpha_{U_S}(Y) \right) =
 \Phi^\dagger_{e^{-H},(\one-e^{-2H})T}(Y) \ .
$$
\end{proof}

Theorem~\ref{FACTORIZTHM} allows us to easily construct an orthogonal basis of eigenvectors of $\Phi^\dagger_{e^{-H},(\one-e^{-2H})T}$ by taking products of eigenfunctions involving a single Fermionic degree of freedom.

\begin{defn}\label{FermHERMDEF} Let $H$ be positive semidefinite on $\cH_1$, and let $T\in \scQ$ be such that ${\rm \ker}(H) \subseteq{\rm ker}(T)$, and $[H,T] =0$. 
Let  $\{\psi_1,\dots,\psi_N\}$  be an orthonormal basis of $\cH_1$ consisting of simultaneous eigenvectors of $H$ and $T$;
For all $j$, $H\psi_j = \lambda_j\psi_j$ and $T\psi_j = \mu_j\psi_j$.   For each $j=1,\dots,N$, define $Z_j := Z(\psi_j)$ and 
\begin{equation}\label{FERMHERM1}
H_{j,(0,0)} := \one\ ,\ H_{j,(1,0)} := Z^*_j \ ,\ H_{j,(0,1)} := Z_j \quad{\rm and}\quad H_{j,(1,1)} := Z_j^*Z_j - \mu_j\one\ .
\end{equation}
For $\aa\in (\{0,1\}\times\{0,1\})^N$, define
\begin{equation}\label{FERMHERM2}
H_\aa := \prod_{j=1}^N H_{j,\alpha_j}
\end{equation}
where the terms are arranged in the product so that the indices increase form left to right.   The operators $H_\aa$ act on $\scK$, and are  not to be confused with the operator $H$ acting on $\cH_1$, which bear no subscript. 
\end{defn}

The operators $H_{j,(a,b)}$ differ from the operators $K_{j,(a,b)}$ in Lemma~\ref{FERMHERM} only in that they are not normalized. The advantage of this is that we do not need to require $0 < \mu_j < 1$; $\{H_\aa \ :\ \aa\in (\{0,1\}\times \{0,1\})^N\}$ is a basis of $\scA_{\cH_1}$ for all $T$. However, if $T$ is such that $\mu_j =0$ for some $j$, then the $j^{{\rm th}}$ factor of $\rho_T$ is $2Z_jZ_j^*$, and hence $Z_j\rho_T =0$, and hence $\langle Z_j,Z_j\rangle_{GNS,\rho_T} =0$. In fact, for this reason 
$\langle H_\aa,H_\aa\rangle_{GNS,\rho_T} =0$ if the second component of $\alpha_j$ is not zero.  Likewise, if $\mu_j = 1$ for some $j$, 
$\langle H_\aa,H_\aa\rangle_{GNS,\rho_T} =0$ if the first component of $\alpha_j$ is not zero.

  In any case, because the normalization does not affect orthogonality, it follows from
Lemma~\ref{FERMHERM} that  $\{H_\aa \ :\ \aa\in (\{0,1\}\times \{0,1\})^N\}$  is a 
  mutually orthogonal set with respect to the GNS inner product induced by $\rho_T$, whether this inner product is degenerate or not. 
 
 These operators may be regarded as {\em un-normalized Hermite polynomials for the GIG state $\rho_T$}. 
 Indeed, for $\bb,\gg\in \{0,1\}^N$ define 
 \begin{equation}\label{POLYBAS}
 Z_{\bb,\gg} = \prod_{j=1}^N (Z^*(\psi_j))^{\beta_j}(Z(\psi_j))^{\gamma_j}\ ,
 \end{equation}
 where the terms in the product are arranged so that the indices increase from left to right. Define $|\bb| = \sum_{j=1}^N \beta_j$ and likewise for $|\gg|$.  It is evident that $\{Z_{\bb,\gg}\ :\ \bb,\gg\in \{0,1\}^N\}$ is a basis, but not an orthonormal basis, for $\scA_{\cH_1}$. It is natural to refer to the span of 
 $$
 \{Z_{\bb,\gg}\ :\ \bb,\gg\in \{0,1\}^N \ {\rm and}\ |\bb| + |\gg| \leq n\}
 $$
 the space of {\em polynomials of degree $n$ or less}. It is easy to see that this space does not depends on the choice of the orthonormal basis $\{\psi_1,\dots,\psi_N\}$. Now fix a GIG state $\rho_T$ where $0 < T < \one$. 
 If one then applies the Gram-Schmidt procedure to the polynomial basis \eqref{POLYBAS}, one arrives at the orthonormal  basis 
 $\{K_\aa \ :\ \aa\in (\{0,1\}\times \{0,1\})^N\}$ of Lemma~\ref{FERMHERM}. This is exactly how the classical Hermite polynomials arise. The connection between Hermite polynomials and  so-called {\em Wick ordering} is well developed in the case of the CCR for Boson fields where the corresponding procedure leads to the classical Hermite polynomials applied to Boson fields \cite{S74,GJ87}.

\begin{lm}\label{EIG1DEG} Let $H$ be positive semidefinite on $\cH_1$, and let $T\in \scQ$ be such that ${\rm \ker}(H) \subseteq{\rm ker}(T)$, and for $j\in \{1,\dots,N\}$ and $(a,b)\in \{0,1\}\times\{0,1\}$, let $H_{j,(a,b)}$ be defined as in Definition~\ref{FermHERMDEF}. Then
\begin{equation}\label{EIG1DEG1}
\Phi^\dagger_{e^{-H},(\one-e^{-2H})T}(H_{j,(a,b)}) = e^{(a+b)\lambda_j}H_{j,(a,b)}\ .
\end{equation}
\end{lm}

\begin{proof}  Because $\Phi^\dagger_{e^{-H},(\one-e^{-2H})T}$ is unital, $\Phi^\dagger_{e^{-H},(\one-e^{-2H})T}(H_{j,(0,0)}) = H_{j,(0,0)}$.
Under the hypotheses on $S= e^{-H}$ and $R = (\one - e^{-2H})T$, 
\eqref{PPEHK1} becomes
\begin{equation}\label{PPEHK1B}
\Phi^\dagger_{e^{-H},(\one-e^{-2H})T}(Z_j) = e^{-\lambda_j} Z_j\ .
\end{equation}
This is the same as $\Phi^\dagger_{e^{-H},(\one-e^{-2H})T}(H_{j,(0,1)}) = e^{-\lambda_j}H_{j,(0,1)}$. Taking adjoints, one has 
$\Phi^\dagger_{e^{-H},(\one-e^{-2H})T}(H_{j,(1,0)}) = e^{-\lambda_j}H_{j,(1,0)}$.

Likewise,  \eqref{PPEHK2} becomes
\begin{equation}\label{PPEHK2B}
\Phi^\dagger_{e^{-H},(\one-e^{-2H})T}(Z^*_jZ_k) = e^{-\lambda_j-\lambda_k} Z^*_jZ_k + \delta_{j,k} (1- e^{-2 \lambda_j})\mu_j\one\ .
\end{equation}
Incidentally, this shows that for $j\neq k$, $Z^*_jZ_k$ is an eigenvector of $\Phi^\dagger_{e^{-H},(\one-e^{-2H})T}$ with eigenvalue $e^{-\lambda_j-\lambda_k}$. 
We are interested in the case $j=k$. Since $\Phi^\dagger_{e^{-H},(\one-e^{-2H})T}$ is unital, $\one$ is an eigenvector of $\Phi^\dagger_{e^{-H},(\one-e^{-2H})T}$ with eigenvalue $1$. For $k=j$,
\eqref{PPEHK2B} becomes 
\begin{equation}\label{PPEHK2C}
\Phi^\dagger_{e^{-H},(\one-e^{-2H})T}(Z^*_jZ_j) = e^{-2\lambda_j} Z^*_jZ_j +  (1- e^{-2\lambda_j})\mu_j\one\ .
\end{equation}
Therefore,
\begin{eqnarray}\label{PPEHK2E}
\Phi^\dagger_{e^{-H},(\one-e^{-2H})T}(Z^*_jZ_j-\mu_j\one) &=& e^{-2\lambda_j} Z^*_jZ_j +  (1- e^{-2\lambda_j})\mu_j\one - \mu_j\one\nonumber\\
&=& e^{-2\lambda_j}(Z^*_jZ_j - \mu_j\one)\ .
\end{eqnarray}
Therefore, $H_{j,(1,1)} = Z^*_jZ_j-\mu_j\one$ is an eigenvector of $\Phi^\dagger_{e^{-H},(\one-e^{-2H})T}$ with eigenvalue $e^{-2\lambda_j}$.   
\end{proof}

\begin{thm}\label{EIGTHM} Let $T\in \scQ$, and let $H\geq 0$ be an operator on $\cH_1$ such that $[H,T]=0$ and ${\rm ker}(H)\subseteq {\rm ker}(T)$.
Then $(S,R) := (e^{-H}, (\one -e^{-2H})T)$ is a compatible pair, and $T$ is the element of $\scQ$ associated to $S$ and $R$ by  Lemma~\ref{SRDEP}. 
Let $\{\psi_1,\dots,\psi_N\}$ be an orthonormal basis of $\cH_1$ such that for each $j$, $H\psi_j = \lambda_j\psi_j$ and $T\psi_j = \mu_j\psi_j$. 
Let $Z_j = Z(\psi_j)$ for each $j$, and let $H_\aa$ be given by \eqref{FERMHERM1} and \eqref{FERMHERM2}.
For $\alpha = (a,b)\in \{0,1\}\times \{0,1\}$, define $|\alpha| = a+b$. For each $\aa\in (\{0,1\}\times \{0,1\})^N$,
\begin{equation}\label{EIGTHMZ1}
\Phi_{e^{-H},(\one-e^{-2H})T}^\dagger(H_\aa) = e^{-\sum_{j=1}^N|\alpha_j|\lambda_j}H_\aa\ .
\end{equation}\label{EIGTHM1}
Consequently, $\Phi_{e^{-H},(\one-e^{-2H})T}^\dagger$ is  self-adjoint with respect to the GNS inner product induced by $\rho_T$. 

Likewise, for 
each $\aa\in (\{0,1\}\times \{0,1\})^N$,
\begin{equation}\label{EIGTHMZ1B}
\Phi_{e^{-H},(\one-e^{-2H})T}(H_\aa\rho_T ) = e^{-\sum_{j=1}^N|\alpha_j|\lambda_j}H_\aa\rho_T\ .
\end{equation}\label{EIGTHM1B}
\end{thm}

\begin{proof}  Everything up through \eqref{EIGTHMZ1} is a direct consequence of Lemma~\ref{DETBAL}, Theorem~\ref{FACTORIZTHM}, Definition~\ref{FermHERMDEF} and Lemma~\ref{EIG1DEG}.  To prove the self-adjointness, recall that  $\{H_\aa \ :\ \aa\in (\{0,1\}\times \{0,1\})^N\}$  is a basis for $\scA_{\cH_1}$ as well
  mutually orthogonal set with respect to the GNS inner product induced by $\rho_T$.  By \eqref{EIGTHMZ1}, the matrix representing $\Phi_{e^{-H},(\one-e^{-2H})T}^\dagger$ in this basis if real and diagonal. 
To prove the final statement, let $\cP^\dagger$ be any operator on $\scA_{\cH_1}$ that is self-adjoint with respect to the GNS inner product induced by $\rho_T$. 
Then for all $Y\in \scA_{\cH_1}$
\begin{equation}\label{EIGTHMZ1BB}
 \cP(Y\rho_T) =\cP^\dagger(Y)\rho_T \ .
\end{equation}
To see this,  suppose for the moment that $0 < T < \one$ so that $\rho_T$ is a faithful state. Let $X\in \scA_{\cH_1}$ and then
\begin{eqnarray*}
\tau(X^*\cP^\dagger(Y)\rho_T) &=& \langle X,\cP^\dagger(Y)\rangle_{GNS,\rho_T}\\ &=& \langle \cP^\dagger(X),Y\rangle_{GNS,\rho_T} = \tau((\cP(X)^\dagger)^* Y\rho_T) 
= \tau(X^* \cP(Y \rho_T))\ .
\end{eqnarray*}
Since $X$ is arbitrary and $\rho_T$ is faithful,  this proves \eqref{EIGTHMZ1BB} for $0 < T < \one$. Now by a limiting argument on obtains 
\eqref{EIGTHMZ1BB} in general.
Now apply \eqref{EIGTHMZ1BB} with $\cP^\dagger = \Phi^\dagger_{e^{-H},(\one-e^{-2H})T}$ and $Y = H_\aa$ to conclude \eqref{EIGTHMZ1B} from \eqref{EIGTHMZ1}.
\end{proof}

\begin{remark} When $0 < T < \one$, $\{ H_\aa \rho_T\ :\ \aa \in (\{0,1\}\times \{0,1\})^N\}$ is  a basis of $\scA_{\cH_1}$ because 
the density matrix  $\rho_T$ is invertible.  Moreover, it is a
mutually orthogonal set with respect to the {\em dual GNS inner product induced by $\rho_T$} , namely $\tau(X^*Y\rho_T^{-1})$, as a direct consequence of the orthogonality
of $\{ H_\aa\ :\ \aa \in (\{0,1\}\times \{0,1\})^N\}$ with respect to the GNS inner product induced by $\rho_T$. 
However, as explained below Definition~\ref{FermHERMDEF}, when $0$ or $1$ is an eigenvalues of $T$, there will be  
$\aa$ such that $H_\aa\rho_T = 0$. 
\end{remark}

\section{Semigroups of quantum operations on the CAR algebra}\label{SEMIGRCAR}

In this section we are concerned with semigroups of quantum operations, both on $\scA_{\cH_1}$ and $\scG_{\cH_1}$. By a semigroup of quantum operations,
 we always mean a {\em continuous semigroup}. That is, we require that $\cP_0 = \one$, and $\lim_{s\to t}\cP_s = \cP_t$ for all $t$ (the continuity property), and $\cP_s\cP_t = \cP_{s+t}$ for all $s,t\geq 0$ (the semigroup property) . Note in particular that we require $\lim_{t\to 0}\cP_t = \one$. 
This implies the $\cP_t$ is one-to-one for all sufficiently small $t$, and then by the semigroup property $\cP_t$ is one-to-one for all $t$. This is essential for applying Theorem~\ref{STRUTHM}.

\subsection{Semigroups on $\scG_{\cH_1}$ that map ${\mathfrak S}_{GIG}$ into itself.}

\begin{lm}\label{SYMMAPSSG} Let $\{\cP_t\}_{t\geq 0}$ be a semigroup of quantum operations on $\scG_{\cH_1}$, each of which maps ${\mathfrak S}_{GIG}$ into itself. 
Let $\gamma_t$ denote the symbol map of $\cP_t$.  Then for all $t>0$,
\begin{equation}\label{SYMMAPSSG1}
\gamma_t(Q) = R_t + e^{tG}Q e^{tG^*}\ .
\end{equation}
where $\{e^{tG}\}_{t\geq 0}$ is a contraction semigroup on $\cH_1$,  uniquely determined up to a global phase, and where $t\mapsto R_t$ is a differentiable function with 
values in the the space of  operators on $\cH_1$  such that 
\begin{equation}\label{SYMMAPSSG1J}
0 \leq R_t \leq \one - e^{tG}e^{tG^*}
\end{equation}
 and 
 \begin{equation}\label{SYMMAPSSG1K}
\lim_{t\downarrow 0}\frac{1}{t}R_t =: A \geq 0\ .
\end{equation}
 Finally, for all $s,t\geq 0$
\begin{equation}\label{SYMMAPSSG2}
R_{s+t} = R_s + e^{sG}R_te^{sG^*}\ .
\end{equation}
\end{lm}

\begin{proof} Since $\lim_{t\to 0}\cP_t = \one$, $\cP_t$ is invertible for all $t\geq 0$, and hence Theorem~\ref{STRUTHM} applies.  Also because 
$\lim_{t\to 0}\cP_t = \one$, $\lim_{t\to 0}\gamma_t$ is the identity on $\scQ$, Therefore, 
$\gamma_t(\one) \geq \gamma_t(0)$ for all sufficiently small $t$, say, all $t \leq t_0$, $t_0 > 0$.  Then because each $\cP_t$ is completely positive, it follows from Theorem~\ref{STRUTHM}  that for each  $t\leq t_0$, 
$\gamma_t$ has the form
$Q \mapsto R_t + S_tQS_t^*$ for some operator $S_t$ on $\cH_1$ that is determined up to a global phase.  Then because  $\gamma_t(Q)\in \scQ$ for all $Q\in \scQ$,
$0 \leq R_t \leq \one - S_tS_t^*$, and hence $S_t$ is a contraction.    This shows that 
\begin{equation}\label{SYMMAPSSG1P}
\gamma_t(Q) = R_t + S_tQ S_t^*
\end{equation}
for all $t \leq t_0$, and then 
 since $\gamma_t(Q) \in \scQ$ for all $Q\in \scQ$,
\begin{equation}\label{SYMMAPSSG1JP}
0 \leq R_t \leq \one - S_t S_t^*
\end{equation}
for all such $t$. 

By the semigroup property,
for all such $s,t \leq t_0$
\begin{equation}\label{SYMMAPSSG3}
\gamma_{s+t}(Q) = R_s + S_sR_tS_s^* + S_sS_t Q S_t^*S_s^*\ .
\end{equation}
Therefore, $\gamma_t$ has the form $Q \mapsto R_t + S_tQS_t^*$ and \eqref{SYMMAPSSG1JP}  is satisfied  for all $t\leq 2t_0$. By simple induction, this is true  for all $t\geq 0$.

Continuous contraction semigroups are automatically differentiable on the domain of their generators. In this finite-dimensional case, the domain is all of $\scG_{\cH_1}$. Because $t \mapsto\cP_t$ is differentiable for $t>0$ and right differentiable at $t=0$,  $t\mapsto \gamma_t$ is differentiable for $t>0$ and right differentiable at $t=0$. Since $R_t = \gamma_t(0)$, $t\mapsto R_t$ also has these properties. Since $\lim_{t\to 0}\gamma_t(Q) = Q$ for all $Q$, it must be that $\lim_{t\to 0}R_t =0$.  Then by the right differentiability of $R_t$ at $t=0$, $A := \lim_{t\downarrow 0}t^{-1}R_t$ exists, and is positive semi-defnite. This proves \eqref{SYMMAPSSG1K}.

By \eqref{SYMMAPSSG3}, for all $Q\in \scQ$, and all $s,t \geq 0$.
\begin{equation}\label{SYMMAPSSG4}
S_{s+t}QS^*_{s+t} = S_sS_t Q S_t^*S_s^*\ .
\end{equation}
Let $\cC_2$ be the Hilbert space consisting of operators on $\cH_1$ equipped with the Hilbert-Schmidt inner product. 
Define the operator $\Gamma_t$ acting on $\cC_2$ by 
\begin{equation}\label{SYMMAPSSG5}
\Gamma_t(X) = S_t X S_t^*\ .
\end{equation}
Then $\|\Gamma_t\| = \|S_t\|^2$, so that each $\Gamma_t$ is a contraction on $\cC_2$. By \eqref{SYMMAPSSG4} and the differentiability of $t\mapsto \gamma_t$, $\{\Gamma_t\}_{t\geq 0}$
is a continuous contraction semigroup on $\cC_2$.  The generator of $\{\Gamma_t\}_{t\geq 0}$ has the form
\begin{equation}\label{SYMMAPSSG56}
\Lambda_G(X) = GX + XG^*
\end{equation}
for some operator $G$ on $\cH_1$ uniquely determined up to an additive term  $ic \one$ for some $c\in \R$. We may fix $G$ by requiring $\Im(\tr[G]) =0$, but in any case $G$ is the generator of a contraction semigroup on $\cH_1$. We fix the phase on $S_t$ by requiring $S_t = e^{tG}$ for all $t$. Then \eqref{SYMMAPSSG1P} and \eqref{SYMMAPSSG1JP} become
\eqref{SYMMAPSSG1} and \eqref{SYMMAPSSG1J}
\end{proof}.

The equation \eqref{SYMMAPSSG2} was studied in \cite{FR80}, but not the condition \eqref{SYMMAPSSG1K}. Perhaps for this reason, the authors were not able to slove it in any generality.
(See the remark just above Lemma III.1 in  \cite{FR80}). 
The next theorem provides the general solution of \eqref{SYMMAPSSG2} subject to \eqref{SYMMAPSSG1K}. Before stating the theorem, recall that  $G$ is the generator of a contraction semigroup on $\cH_1$ if and only if for all $\psi\in \cH_1$, $\Re\langle \psi, G \psi\rangle \leq 0$, and this is the same as $-G - G^* \geq 0$.

\begin{thm}\label{GIGSGTHM} Let $\{e^{tG}\}_{t\geq 0}$ be a contraction semigroup on $\cH_1$. Then a differentiable function $R_t$ on $[0,\infty)$ with values in $\scQ$  satisfies
\begin{equation}\label{GIGSGTHM1}
R_{t+h}  = R_h + e^{hG} R_te^{hG^*}   
\end{equation}
for all $t,h\geq 0$ 
together with
\begin{equation}\label{GIGSGTHM2}
 \lim_{t\to 0}\frac{1}{t}R_t = A \geq 0  \quad{\rm and}\quad 0 \leq R_t \leq \one -  e^{tG}e^{tG^*} 
\end{equation}
 if and only if 
\begin{equation}\label{GIGSGTHM3}
R_t = \int_0^t e^{sG} A e^{sG^*}{\rm d}s 
\end{equation}
and
\begin{equation}\label{GIGSGTHM4}
0 \leq A \leq -(G+G^*)\ .
\end{equation}
\end{thm}

\begin{proof} Fix any $A\geq 0$, and define $R_t$ by \eqref{GIGSGTHM3}.  Then for all $t,h>0$
$$
e^{h G}R_t e^{hG^*} = \int_0^t e^{(h+s)G}A e^{(h+s)G^*}{\rm d}s =  \int_h^{h+t} e^{sG} A e^{sG^*}{\rm d}s  = R_{t+h} - R_h\ ,
$$
so that \eqref{GIGSGTHM1} is satisfied.  Evidently, $\lim_{t\to 0}R_t = 0$, and $0 \leq R_t$ for all $t$.  If furthermore \eqref{GIGSGTHM4} is satisfied, 
$$
R_t \leq  \int_0^t e^{sG} (-G - G^*) e^{sG^*}{\rm d}s = \int_0^t -\frac{{\rm d}}{{\rm d}s}e^{sG} e^{sG^*}{\rm d}s = \one - e^{tG}e^{tG^*}
$$
so that all requirements  in \eqref{GIGSGTHM2} are satisfied. 

Conversely, suppose that  \eqref{GIGSGTHM1} and \eqref{GIGSGTHM2}  are satisfied.  Subtracting $R_t$ from both sides in \eqref{GIGSGTHM1}, and dividing by $h>0$ yields
\begin{equation}\label{GIGSGTHM5}
\frac{1}{h}(R_{t+h} - R_t) = \frac{1}{h} R_h  +  \frac{1}{h} (e^{hG} R_te^{hG^*}  - R_t)\ .
\end{equation}
Define
$$A := \lim_{h\downarrow 0}\frac{1}{h} R_h$$ and define $\Lambda_G$ to be the operator on $\cB(\cH_1)$ given by
$$
\Lambda_G(X) = GX + XG^*\ ,
$$ 
noting that $\Lambda_G$ is the generator of the semigroup $X \mapsto e^{tG}Xe^{tG^*}$.  Then taking the limit $h\downarrow 0$ in \eqref{GIGSGTHM5},
$$
\frac{{\rm d}}{{\rm d}t}R_t = A +  \Lambda_G(R_t)\ .
$$
The unique solution of this differential equation is
$$
R_t = e^{t\Lambda}R_0 + \int_0^t e^{(t-s)\Lambda} A {\rm d}s  =  e^{t\Lambda}R_0 + \int_0^t e^{s\Lambda} A {\rm d}s\ .
$$
For $R_0 =0$, this reduces to \eqref{GIGSGTHM3}. We have already seen that when $A \leq -(G+G^*)$, then $R_t \leq \one -  e^{tG}e^{tG^*}$. 
Now suppose that $R_t$ is given by \eqref{GIGSGTHM3}, and that $R_t \leq \one -  e^{tG}e^{tG^*}$. Then for $t$ near $0$,
$R_t = tA + o(t)$ and $e^{tG}e^{tG^*} = \one + tG + t G^* + o(t)$. Therefore,  $A \leq -(G+G^*)$ is necessary for $R_t \leq \one -  e^{tG}e^{tG^*}$ to be true for all $t>0$. 
\end{proof} 

Lemma~\ref{SYMMAPSSG} and Theorem~\ref{GIGSGTHM} lead to  a complete description of semigroups of quantum operations on the Araki-Wyss algebra $\scG_{\cH_1}$ that map ${\mathfrak S}_{GIG}$ into itself. 

\begin{thm}\label{AWQDS} There is a one-to-one correspondence between:

\smallskip
\noindent{\it (1)}  
Semigroups  $\{\cP_t\}_{t\geq 0}$ of quantum operations on $\scG_{\cH_1}$ such that for each $t\geq 0$  $\cP_t$ maps ${\mathfrak S}_{GIG}$ into itself
\medskip

and

\smallskip
\noindent{\it (2)}
Pairs  $\left(G,A\right)$ consisting of 
a contraction semigroup generator $G$  on $\cH_1$  together with an operator $A$ on $\cH_1$ 
satisfying $0 \leq A \leq -G -G^*$.

\medskip

More specifically, let $\{\cP_t\}_{t\geq 0}$ be as in {\it (1)}. 
Then there is a pair $\left(G,A\right)$ as in {\it (2)}   such that with $R_t = \int_0^t e^{sG}Ae^{sG^*}{\rm d}s$, 
\begin{equation}\label{GIGSGTHM21}
\cP_t = \Phi_{e^{tG},R_t}\big|_{\scG_{\cH_1}}\ 
\end{equation}
for all $t\geq 0$, where $\Phi_{S,R}$ is the GIG quantum operation produced by the EHK construction. 

Conversely,  let $\left(G,A\right)$ be as in {\it (2)}. Define $R_t = \int_0^t e^{sG}Ae^{sG^*}{\rm d}s$.
Then for each $t$, $(e^{tG},R_t)$ is a compatible pair so that the EHK construction yields a GIG quantum operation $\Phi_{e^{tG},R_t}$ on $\scA_{\cH_1}$, and with $\cP_t$ defined by
\eqref{GIGSGTHM21}, $\{\cP_t\}_{t\geq 0}$ is a semigroup of quantum operations on $\scG_{\cH_1}$ such that each $\cP_t$ maps ${\mathfrak S}_{GIG}$ into itself. 
\end{thm}

\begin{proof} Let $\{\cP_t\}_{t\geq 0}$ be as in {\it (1)}. Let $\gamma_t$ denote the symbol map of 
$\cP_t$.  By Lemma~\ref{SYMMAPSSG} there is a contraction semigroup $\{e^{tG}\}_{t\geq 0}$ on $\cH_1$, and a differentiable function $t\mapsto R_t$ from $[0,\infty)$ to $\scQ$ such
that for all $Q\in \scQ$, $\gamma_t(Q) = R_t + e^{tG}Qe^{tG^*}$ and such that $\lim_{t\downarrow 0}\frac{1}{t}R_t = A \geq 0$ and such that $0 \leq R_t \leq \one -e^{tG}e^{tG^*}$ so that for each $t$, $(e^{tG},R_t)$ is a compatible pair.  By 
Theorem~\ref{GIGSGTHM},  $0 \leq A \leq -G -G^*$, so that $\left(G,A\right)$ is a pair as in {\it (2)}.

The EHK construction gives us a GIG quantum operation $\Phi_{e^{tG},R_t}$ whose symbol map is $Q\mapsto R_t + e^{tG}Q e^{tG^*} = \gamma_t(Q)$.  By the Araki-Wyss Theorem, Theorem~\ref{AWTHM},
every quantum operation on $\scG_{\cH_1}$  that maps ${\mathfrak S}_{GIG}$ into itself  is completely determined by its symbol map. Therefore, $\cP_t$ satisfies \eqref{GIGSGTHM21}.

Conversely, given a pair $\left(G,A\right)$  as in {\it (2)},  define $R_t$  in terms of 
$G$ and $A$ by \eqref{GIGSGTHM3}. Then by Theorem~\ref{GIGSGTHM}, for all $s,t\geq 0$, $R_{s+t}  = R_s + e^{sG} R_te^{sG^*}$ and $0 \leq R_t \leq \one - e^{tG}e^{tG^*}$. 
In particular, for all $t\geq 0$, $(e^{tG},R_t)$ is a compatible pair, and the EHK construction gives us a GIG quantum operation $\Phi_{e^{tG},R_t}$  on $\scA_{\cH_1}$ with symbol map $\gamma_t(Q) =  R_t + e^{tG}Q e^{tG^*}$. 
Then  for all $Q\in \scQ$,
$$
\gamma_{s+t}(Q) = R_{s+t} + e^{(s+t)G}Q e^{(s+t)G^*} = R_s + e^{sG}\gamma_t(Q)e^{sG^*} = \gamma_s(\gamma_t(Q))\ .
$$  
Because a linear map on $\scG_{\cH_1}$ is completely determined by its action on ${\mathfrak S}_{GIG}$, if we define $\cP_t$ on $\scG_{\cH_1}$ by \eqref{GIGSGTHM21}, then
$$
\cP_{s+t} = \cP_{s}\circ \cP_t\ .
$$
Also, since $\lim_{t\to 0}\gamma_t$ is the identity on $\scQ$, $\lim_{t\downarrow 0}\cP_t = \one$. Therefore, $\{\cP_t\}_{t\geq 0}$ is a semigroup of quantum operations on $\scG_{\cH_1}$ that maps
${\mathfrak S}_{GIG}$ into itself, as in {\it (1)}.
\end{proof}

\section{GIG semigroups on  $\scA_{\cH_1}$} \label{GIGSGA}

Every GIG semigroup on $\scA_{\cH_1}$ leaves $\scG_{\cH_1}$ invariant since the complex span of ${\mathfrak S}_{GIG}$ is $\scG_{\cH_1}$. Therefore, the restriction of any such semigroup to $\scG_{\cH_1}$ is (the dual of) one of the semigroups described in Theorem~\ref{AWQDS}, and hence of the form $\left\{\Phi_{e^{tG},R_t}\big|_{\scG_{\cH_1}}\right\}_{t\geq 0}$ where $G$ is a contraction semigroup generator 
on $\cH_1$, $R_t = \int_0^t e^{sG}A e^{sG^*}{\rm d}s$ and $0 \leq A \leq -G -G^*$. 
  In this section we study the problem of extending such semigroups on $\scG_{\cH_1}$ to all of 
$\scA_{\cH_1}$. 

Given a pair $(G,A)$ as in Theorem~\ref{AWQDS}, let $R_t = \int_0^t e^{sG}Ae^{sG^*}{\rm d}s$ and define an operator $\cL_{G,A}^\dagger$ on $\scA_{\cH_1}$ by
\begin{equation}\label{GENDEFP11}
\cL_{G,A}^\dagger(X) = \lim_{t\downarrow 0}\frac1t \left(\Phi^\dagger_{e^{tG},R_t}X - X\right)\ . 
\end{equation}
(In this finite dimensional setting, the limit clearly exists for  all $X\in \scA_{\cH_1}$, but in fact we will calculate it below).
 By Theorem~\ref{AWQDS}, with $\cP_t^\dagger := \Phi^\dagger_{e^{tG},R_t}\big|_{\scG_{\cH_1}}$, $\{\cP_t^\dagger\}_{t\geq 0}$  is a semigroup on $\scG_{\cH_1}$, and hence 
 $\scG_{\cH_1}$ is invariant under $\cL_{G,A}^\dagger$, and   $\cL_{G,A}^\dagger\big|_{\scG_{\cH_1}}$ is the generator of this semigroup. 
 
Theorem~\ref{GENMAINTHM} below  shows that $\cL_{G,A}^\dagger$ can be written in Lindblad form as an operator on all of $\scA_{\cH_1}$, and even shows how to do this 
explicitly in terms of the data $(G,A)$. 
Since $\scG_{\cH_1}$ is not a factor, it does not follow from 
Lindblad's Theorem \cite{Lin76} that 
$\cL_{G,A}^\dagger\big|_{\scG_{\cH_1}}$ can be written in Lindblad form. While $\scA_{\cH_1}$ is a factor,  so that by Lindblad's Theorem
 the generator of any semigroup of unital CP operators on $\scA_{\cH_1}$ can be written in 
Lindblad form, the results at our disposal so far do not show that $\{\Phi^\dagger_{e^{tG},R_t}\}_{t\geq 0}$ is a semigroup on all of $\scA_{\cH_1}$. 
However, we shall now prove  that, in fact,  $\cL_{G,A}^\dagger$ can be written in Lindblad form. 

\begin{thm}\label{GENMAINTHM} Let $(G,A)$ be a pair consisting of a contraction semigroup generator $G$ on $\cH_1$ and an operator on $A$ on $\cH_1$ satisfying $0 \leq A \leq -G -G^*$, as in Theorem~\ref{AWQDS}.  Define 
\begin{equation}\label{GENMAINTHM1}
H := -\frac12(G +G^*) \quad{\rm and}\quad K := \frac{1}{2i}(G - G^*)
\end{equation}
so that $G^* = -H -iK$.  Define 
\begin{equation}\label{GENMAINTHM2}
T := \lim_{\epsilon\downarrow 0}(2H + \epsilon \one)^{-1/2}A(2H + \epsilon \one)^{-1/2}
\end{equation}
and let $\{\eta_1,\dots,\eta_N\}$ be an orthonormal basis of $\cH_1$ consisting of eigenvectors of $T$, and for $j=1,\dots,N$, define $\lambda_j$ by $T\phi_j = \lambda_j \phi_j$. 
Then define $\{\phi_1,\dots,\phi_N\}$ by 
\begin{equation}\label{GENMAINTHM2B}
\phi_j := (2H)^{1/2}\eta_j\ ,\quad j=1,\dots,N\ .
\end{equation}
Finally, define the operator $\widehat{K}$ on $\scK$ by ``second quantizing'' the operator $K$ defined in \eqref{GENMAINTHM1}. 
Then  the action of $\cL_{G,A}^\dagger$ as defined in \eqref{GENDEFP11} on
$ \scA_{\cH_1}$ is given by 
\begin{eqnarray}
\cL_{G,A}^\dagger(X) &=& \sum_{j=1}^N(1-\lambda_j) \bigl( Z^*(\phi_j)W X W Z(\phi_j)  - \tfrac12 Z^*(\phi_j)Z(\phi_j)X - \tfrac12 X Z^*(\phi_j)Z(\phi_j)\bigr)\nonumber\\
&+& \sum_{j=1}^N \lambda_j \bigl( Z(\phi_j)W X W Z^*(\phi_j)  - \tfrac12 Z(\phi_j)Z^*(\phi_j)X - \tfrac12 X Z(\phi_j)Z^*(\phi_j)\bigr)\nonumber\\
&+& i\tfrac12[\widehat{K},X]\ .\label{GENMAINTHM3}
\end{eqnarray}
where  $W := e^{i\frac{\pi}{2}\cN}$.
\end{thm}

While the proof of Theorem~\ref{GENMAINTHM} is essentially a computation, it is somewhat involved and is distributed over several lemmas.  

Before turning to these lemmas, we formulate a simple corollary that addresses some questions raised in the introduction. 

\begin{cl}\label{GENMAINCL} Let $(G,A)$ be a pair consisting of a contraction semigroup generator $G$ on $\cH_1$ and an operator on $A$ on $\cH_1$ satisfying $0 \leq A \leq -G -G^*$, as in Theorem~\ref{AWQDS}. Let $\cL^\dagger_{G,A}$ be the Lindblad generator on the right side of \eqref{GENMAINTHM3}. Then the dual semigroup  $\left\{e^{t\cL_{G,A}}\right\}_{t\geq 0}$ is a GIG quantum semigroup extending the semigroup $\left\{\Phi_{e^{tG},R_t}\big|_{\scG_{\cH_1}}\right\}_{t\geq 0}$ on $\scG_{\cH_1}$ associated to $(G,A)$  by Theorem~\ref{AWQDS}, so that the symbol map of $e^{t\cL_{G,A}}$ is $Q \mapsto R_t + e^{tG}Q e^{tG^*}$.  Moreover, for all $t\geq 0$, 
\begin{equation}\label{GENMAINTHM4}
e^{t\cL_{G,A}} = \lim_{n\to\infty} \left(\Phi_{e^{(t/n)G},R_{(t/n)}}\right)^n\ .
\end{equation}
In this sense, every GIG quantum semigroup arises from the EHK construction. 
\end{cl} 

\begin{proof}
Since $\cL_{G,A}^\dagger\big|_{\scG_{\cH_1}}$ is the generator of the semigroup $\left\{\Phi^\dagger_{e^{tG},R_t}\big|_{\scG_{\cH_1}}\right\}_{t\geq 0}$ associated to 
$(G,A)$ by Theorem~\ref{AWQDS}, $\left\{e^{t\cL_{G,A}^\dagger}\right\}_{t\geq 0}$ is an extension of this semigroup to all of $\scA_{\cH_1}$. Since all GIG density matrices belong to $\scG_{\cH_1}$,
for any $\rho_Q\in {\mathfrak S}_{GIG}$, 
$$
e^{t\cL_{G,A}}\rho_Q = \Phi_{e^{tG},R_t}\rho_Q
$$
and hence the symbol map of $e^{t\cL_{G,A}}$ is $Q \mapsto R_t + e^{tG}Qe^{tG^*}$.  It is also clear from the explicit form of $\cL_{G,A}^\dagger$ that for all $t$,
$[\cL_{G,A}^\dagger,\alpha_{{\rm G},t}] =0$. Consequently  $[\cL_{G,A},\alpha_{{\rm G},t}] =0$  for all $t$, so that $\{e^{t\cL_{G,A}^\dagger}\}_{t\geq 0}$ is a GIG semigroup of quantum operations. 
Finally, \eqref{GENMAINTHM4} follows from a theorem of Chernoff \cite{C68}.
\end{proof}

For a discussion of Chernoff's Theorem with a simple proof,  see \cite[p. 108]{N69}. While Corollary~\ref{GENMAINCL} shows that every semigroup of quantum operations on $\scG_{\cH_1}$ that maps ${\mathfrak S}_{GIG}$ into itself  has an extension to a 
GIG  semigroup on all of $\scA_{\cH_1}$ that is ``natural'' from the point of view of the EHK construction, it does not say that $\{e^{t\cL_{G,A}^\dagger}\}_{t\geq 0}$ is the 
only such extension, nor does it say that for all $t$, 
$e^{t\cL_{G,A}^\dagger} = \Phi_{e^{tG},R_t}$ although this identity is true restricted to $\scG_{\cH_1}$.

The computation justifying \eqref{GENMAINTHM3} will be carried out using a basis for $\scA_{\cH_1}$ that is now introduced. 
Let $\{\psi_1,\dots,\psi_N\}$ be an orthonormal basis of $\cH_1$ consisting of eigenvectors of $A$; that is $A\psi_j = \nu_j$ for each $j$. Define $Z_j = Z(\psi_j)$, $j=1,\dots,N$.   For $\bb,\gg\in \{0,1\}^N$, define $Z_{\bb,\gg}$ as in \eqref{POLYBAS}. Then 
$\{Z_{\bb,\gg}\ :\ \bb,\gg\in \{0,1\}^N\}$ is a basis of $\scA_{\cH_1}$.

\begin{lm}\label{GENLIM1} Let $(G,A)$ be a pair as in part {\it (2)} of Theorem~\ref{AWQDS}, and for $t\geq 0$ let $R_t = \int_0^t e^{sG}Ae^{sG^*}{\rm d}s$. Then for all $ \bb,\gg\in \{0,1\}^N$,
\begin{equation}\label{GENLIM2}
\lim_{t\downarrow 0}\frac1t \left(\Phi^\dagger_{e^{tG^*},R_t}Z_{\bb,\gg} - Z_{\bb,\gg}\right) = \\
\sum_{k=1}^N \left(\prod_{1\leq j< k}(Z_j^*)^{\beta_j}Z_j^{\gamma_j}\right) F_k \left(\prod_{k<\ell\leq N}(Z_j^*)^{\beta_j}Z_j^{\gamma_j} \right)
\end{equation}
where
\begin{equation}\label{GENLIM3}
F_k := 
\delta_{\beta_k,1}Z^*(G^*\psi_k)^{\beta_k}Z_k^{\gamma_k} + \delta_{\gamma_k,1}(Z^*_k)^{\beta_k} Z(G^*\psi_k) + \delta_{(\beta_k,\gamma_k),(1,1)} \langle \psi_k,A\psi_k\rangle\ .
\end{equation}
In particular, with $\cL^\dagger_{G,A}$ given by \eqref{GENDEFP11}, 
\begin{equation}\label{GENLIM4A}
\cL^\dagger_{G,A}(Z_j^*) = Z^*(G^*\psi_j)
\end{equation}
and
\begin{equation}\label{GENLIM4}
\cL^\dagger_{G,A}(Z_j^*Z_k) = Z^*(G^*\psi_j)Z_k + Z_j^*Z(G^*\psi_k) + \langle \psi_j,A\psi_k\rangle \one\ .
\end{equation}
\end{lm}

\begin{proof}
We use formula \eqref{EHKGIG21} with $S = e^{tG}$, and with $V$ in place of $A$ (which has another meaning at present) and $Z$ in place of $B$
\begin{equation}\label{EHKGIG21B2}
\alpha_{U_{e^{tG}}}(Z(\psi_j)) = - Z((\one -e^{tG}e^{tG^*})^{1/2}\psi_j) + V(e^{tG^*}\psi_j)\ .
\end{equation}
Therefore, $\alpha_{U_{e^{tG}}}(Z_{\bb,\gg})$ is a sum over all the products obtained by selecting either  $V^*(e^{tG^*}\psi_j)$ or $- Z^*((\one -e^{tG}e^{tG^*})^{1/2}\psi_j)$ 
to replace $Z^*_j$ and selecting either   $V(e^{tG^*}\psi_j)$ or $- Z((\one -e^{tG}e^{tG^*})^{1/2}\psi_j)$ 
to replace $Z_j$. The simplest of these terms is the one in which there are no $Z$ terms. This yields the product 
$$
\prod_{j=1}^N (V^*(e^{tG^*}\psi_j))^{\beta_j}(V(e^{tG^*}\psi_j))^{\gamma_j}\ ,
$$
which already belongs to $\scA_{0\oplus \cH_1}$ so there is no conditional expectation to take in  formula \eqref{EHKGIG21}.
  Identifying $\scA_{0\oplus \cH_1}$ with $\scA_{\cH_1}$ (that is, replacing $V$ by $Z$), the contribution of this term to the generator is simply
 $$
 \frac{{\rm d}}{{\rm d}t}\prod_{j=1}^N (V^*(e^{tG^*}\psi_j))^{\beta_j}(V(e^{tG^*}\psi_j))^{\gamma_j}\bigg|_{t=0}\ .
 $$
 This produces the first two of the three terms in each $F_k$ in \eqref{GENLIM3}.  Note that we have already made full use of the $-Z_{\bb,\gg}$ on the left in \eqref{GENLIM2}.

To compute the contribution of the remaining terms, which do involve some $Z$  factors, 
we must 
 bring all of the $Z$ terms together and take their expectation in the state $\rho_{T_{e^{tG},R_t}}$. 
 Since $\rho_{T_{e^{tG},R_t}}$ is a GIG state, the expectations are all determined by the expectations of the form 
$$
\rho_{T_{e^{tG},R_t}}(Z^*((\one -e^{tG}e^{tG^*})^{1/2}\psi_j) Z((\one -e^{tG}e^{tG^*})^{1/2}\psi_k))   = \langle \psi_j, R_t\psi_k\rangle. 
$$
Then since $R_t = \int_0^t e^{sG}Ae^{sG^*}{\rm d}s$, 
$$
\frac1t \langle \psi_j, R_t\psi_k\rangle  = \langle \psi_j,A\psi_k\rangle +\frac{o(t)}{t}  = \delta_{j,k}\langle \psi_k,A\psi_k\rangle+ \frac{o(t)}{t}  \ .
$$
Therefore, the only pairings that contribute to the generator are those coming from pairs in which $\beta_j = \gamma_j =1$.  This gives us the remaining contribution to $F_k$. 
Finally, \eqref{GENLIM4} is a consequence of \eqref{GENLIM2} and \eqref{GENLIM3}
\end{proof}

Our goal now is to write \eqref{GENLIM4} in Lindblad form. The next lemmas provide the means to do this.

\begin{lm}\label{GENMAN} Let $Z$ be an  CAR field over $\cH_1$ acting irreducibly on $\scK$. Let $W = e^{i\pi \cN}$, and for any non-zero vector $\phi\in\cH_1$, define the Lindblad generators
$\cL^\dagger_{Z^*(\phi)}$ and $\cL^\dagger_{Z(\phi)}$
by
$$
\cL^\dagger_{Z^*(\phi)}(X) = 2Z^*(\phi)W X W Z(\phi)  - Z^*(\phi)Z(\phi)X - X Z^*(\phi)Z(\phi)
$$
and
$$
\cL^\dagger_{Z(\phi)}(X) = 2Z(\phi)W X W Z^*(\phi)  - Z(\phi)Z^*(\phi)X - X Z(\phi)Z^*(\phi)\ .
$$
Then for any $\psi\in \cH_1$,  the following are valid
\begin{equation}\label{GENMAN1} 
\cL^\dagger_{Z^*(\phi)}(Z(\psi))  =  \cL^\dagger_{Z(\phi)}(Z(\psi))  = Z(|\phi\rangle\langle\phi|\psi)\ ,
\end{equation}
\begin{equation}\label{GENMAN2} 
\cL^\dagger_{Z^*(\phi)}(Z^*(\psi)Z(\psi)) = -\langle \phi,\psi\rangle Z^*(\phi)Z(\psi) - \langle\psi,\phi\rangle Z^*(\psi)Z(\phi)
\end{equation}
and
\begin{equation}\label{GENMAN3} 
\cL^\dagger_{Z(\phi)}(Z^*(\psi)Z(\psi)) = 2|\langle \phi,\psi\rangle|^2\one - \langle\phi,\psi\rangle Z^*(\phi)Z(\psi) - \langle \psi,\phi\rangle Z^*(\psi)Z(\phi)\ .
\end{equation}
\end{lm}

\begin{proof} Since $W$ is a self-adjoint unitary commuting with $Z(\psi)$ (and hence $Z^*(\psi)$) for all $\psi$,
\begin{multline*}
2Z^*(\phi)W Z(\psi) W Z(\phi) = -2Z^*(\phi) Z(\psi) Z(\phi) =\\ (\langle\psi,\phi\rangle\one - Z(\psi)Z^*(\phi))Z(\phi) + Z^*(\phi)Z(\phi)Z(\psi) \ .
\end{multline*}
Therefore, $\cL^\dagger_{Z^*(\phi)}(Z(\psi)) = -\langle\psi,\phi\rangle Z(\phi) = Z(\langle\phi,\psi\rangle \phi) = Z(|\phi\rangle\langle\phi|\psi)$. 
This proves the first formula in \eqref{GENMAN1}. All of the other formulas are proved 
making similar computations using the CAR. 
\end{proof}

Now let $\{\phi_1,\dots,\phi_N\}$ be a set of $N$ vectors in $\cH_1$, not necessarily a basis. 
Let $\{\lambda_1,\dots,\lambda_N\}$ be a set of $N$ numbers in $[0,1]$. Define  positive semidefinite   operators $H_0$, $H_1$ and $K$ on $\cH_1$ by 
$$
H_0 := \sum_{j=1}^N (1-\lambda_j) |\phi_j\rangle\langle \phi_j|\ ,\quad H_1 := \sum_{j=1}^N \lambda_j |\phi_j\rangle\langle \phi_j| \quad{\rm and}\quad  2H = H_0+H_1
$$
and define two Lindblad generators $\cK^\dagger_0$ and $\cK^\dagger_1$ by 
\begin{equation}\label{GENMAN31} 
\cK^\dagger_0 := \sum_{j=1}^N (1-\lambda_j) \cL^\dagger_{Z^*(\phi_j)} \quad{\rm and}\quad \cK^\dagger_1 := \sum_{j=1}^N \lambda_j \cL^\dagger_{Z(\phi_j)} 
\end{equation}
Then because 
$$\langle \phi,\psi\rangle Z^*(\phi)Z(\psi) =  Z^*(|\phi\rangle\langle\phi|\psi)Z(\psi)\quad{\rm and}\quad \langle\psi,\phi\rangle Z^*(\psi)Z(\phi) = Z^*(\psi)Z(|\phi\rangle\langle \phi|\psi)\ ,$$
$$
\cK^\dagger_0(Z^*(\psi)Z(\psi)) = -Z^*(H_0\psi)Z(\psi) - Z^*(\psi)Z(H_0\psi) 
$$
and
$$
\cK^\dagger_1(Z^*(\psi)Z(\psi)) =2 \langle\psi,H_1\psi\rangle \one  -Z^*(H_1\psi)Z(\psi) - Z^*(\psi)Z(H_1\psi) 
$$
Therefore,
\begin{equation}\label{MIXTURE31}
(\cK^\dagger_0+ \cK^\dagger_1)(Z^*(\psi)Z(\psi)) = 2\langle\psi,H_1\psi\rangle \one -Z^*(2H\psi)Z(\psi) - Z^*(\psi)Z(2H \psi) 
\end{equation}

We now show that we may choose $\{\phi_1,\dots,\phi_N\}$ and $\{\lambda_1,\dots,\lambda_N\}$ so that $2H = -G -G^*$ and $A = H_1$. This is a consequence of the following lemma. 

\begin{lm}\label{MIXTURE} Let $0 \leq H_1 \leq 2H$ be positive semidefinite operators on $\cH_1$. Then there is a set of vectors $\{\phi_1,\dots,\phi_N\}$ in $\cH_1$ and a set of numbers 
$\{\lambda_1,\dots,\lambda_N\}$ in $[0,1]$ such that 
$$
\sum_{j=1}^N\lambda_j |\phi_j\rangle\langle \phi_j| = H_1 \quad{\rm and}\quad \sum_{j=1}^N |\phi_j\rangle\langle \phi_j| = 2H\ .
$$
\end{lm}

\begin{proof} Let $T := \lim_{\epsilon\downarrow 0}(2H+\epsilon\one)^{-1/2} H_1 (2H+\epsilon\one)^{-1/2}$. Then $0 \leq T \leq \one$. Let $\{\eta_1,\dots,\eta_N\}$ be an orthonormal basis of $\cH_1$ consisting of eigenvectors of $T$, and for each $j$, define $\lambda_j$ by $T\eta_j = \lambda_j\eta_j$. Then $T = \sum_{j=1}^N \lambda_j |\eta_j\rangle \langle \eta_j|$, and 
$\one = \sum_{j=1}^N \ |\eta_j\rangle \langle \eta_j|$
Then from the definition of $T$, it follows that $H_1 = (2H)^{1/2}T(2H)^{1/2}$ and of course $2H = (2H)^{1/2}\one (2H)^{1/2}$. Therefore,
$$
H_1  = \sum_{j=1}^N \lambda_j |(2H)^{1/2}\eta_j\rangle \langle (2H)^{1/2}\eta_j| \quad{\rm and}\quad H  = \sum_{j=1}^N  |(2H)^{1/2}\eta_j\rangle \langle (2H)^{1/2}\eta_j|\ .
$$
Therefore we take $\phi_j = (2H)^{1/2}\eta_j$ for $j=1,\dots,N$. 
\end{proof}

We now apply Lemma~\ref{MIXTURE} with 
\begin{equation}\label{MIXTURE32}
H_1 = A \quad {\rm and}\quad 2H = -G -G^*\ ,
\end{equation}
 which we may do since  $0 \leq A \leq -G -G^*$. Using the choice of $\{\phi_1,\dots,\phi_N\}$ and $\{\lambda_1,\dots,\lambda_N\}$ provided by the Lemma for this choice if $H_1$ and $H$,  
\eqref{MIXTURE31} becomes
\begin{equation}\label{MIXTURE33}
(\cK^\dagger_0+ \cK^\dagger_1)(Z^*(\psi)Z(\psi)) = 2\langle\psi,A\psi\rangle \one -Z^*(2H\psi)Z(\psi) - Z^*(\psi)Z(2H \psi) 
\end{equation}

Now define the Lindblad generator $\cL^\dagger$ by
\begin{equation}\label{MIXTURE35}
\cL^\dagger(X) = \frac12(\cK^\dagger_0(X)+ \cK^\dagger_1(X) +i[\widehat{K},X])\ 
\end{equation}
where $K$ is specified in \eqref{GENMAINTHM1} so that $G^* = -H -iK$.
Note that $\cL^\dagger(X)$ is exactly what appears on the right side of \eqref{GENMAINTHM3}.   Therefore, to prove Theorem~\ref{GENMAINTHM3} we must show that 
$\cL^\dagger(X)$ as specified in \eqref{MIXTURE35} agrees with the right side of \eqref{GENLIM2} in Lemma~\ref{GENLIM1}.

Recalling that $[\widehat{K},Z(\psi)] = Z(K\psi)$,
\begin{equation}\label{MIXTURE41}
\cL^\dagger (Z^*(\psi)Z(\psi)) = \langle\psi,A\psi\rangle \one -Z^*(G^*\psi)Z(\psi) - Z^*(\psi)Z(G^* \psi) \ .
\end{equation}

Even more simply, using \eqref{GENMAN1}, one finds. 
\begin{equation}\label{MIXTURE42}
\cL^\dagger (Z(\psi)) = Z(G^* \psi) \ .
\end{equation}

Since $\cL^\dagger$ is Hermitian, and evidently satisfied $\cL^\dagger \one =0$, comparison with Lemma~\ref{GENLIM1} shows that $\cL^\dagger Z_{\bb,\gg}$ is given by the right side of \eqref{GENLIM2} whenever $|\bb| + |\gg| \leq 2$.

\begin{lm}\label{CDC} Let $L\in \scA_{\cH_1}$, and let $\cL_L^\dagger$ denote the Lindblad generator acting on $\scA_{\cH_1}$ according to
\begin{equation}\label{LINDGENDEF1}
\cL_L^\dagger(X) := 2LXL^* - LL^*X - XLL^*  = [L,X]L^* - L[L^*,X]\ .
\end{equation}
Then for $X,Y\in \scA_{\cH_1}$,
\begin{equation}\label{LINDGENDEF2}
\cL_L^\dagger(XY ) - \cL_L^\dagger(X)Y - X\cL_L^\dagger(Y)  = -2[L,X][L^*,Y]\ .
\end{equation}
\end{lm}

\begin{remark} Probabilists will recognize the left side of \eqref{LINDGENDEF2} as the {\em carr\'e du champ} bilinear form as defined by P.-A.~Meyer \cite{PAM76} which is fundamental to the study of classical Markovian semigroups, especially in the diffusive case \cite{BGL}. 
\end{remark}

\begin{proof}[Proof of Lemma~\ref{CDC}]
For any operator $X$,
\begin{equation}\label{LINDGENDEF3}
\cL_L^\dagger(X) := 2LXL^* - LL^*X - XLL^*  = [L,X]L^* - L[L^*,X]\ .
\end{equation}
Next,
\begin{eqnarray*}
LXYL^* &=& LXL^*Y + LX[Y,L^*] = LL^*XY + L[X, L^*]Y +   LX[Y,L^*]\\
&=& LL^*XY + L[X, L^*]Y +   XL[Y,L^*] + [L,X][Y,L^*]\ .
\end{eqnarray*}
Likewise,
\begin{eqnarray*}
LXYL^* &=& XLYL^* + [L,X]YL^* = XYLL^* + [L,X]YL^* + X[L,Y]L^*\\
&=&   XYLL^* + [L,X]L^*Y + X[L,Y]L^* + [L,X][Y,L^*]\ .
\end{eqnarray*}
Combining the last two calculations with \eqref{LINDGENDEF3} yields \eqref{LINDGENDEF2}.
\end{proof}

We apply this to the Lindblad generators  $\cK^\dagger_0$ and $\cK^\dagger_1$ defined in 
\eqref{GENMAN31} to obtain
\begin{equation}\label{CDC0}
\cK_0^\dagger(XY) - \cK_0^\dagger(X)Y - X\cK_0^\dagger(Y)  = -2\sum_{j=1}^N (1-\lambda_j)[Z^*(\phi_j)W,X][WZ(\phi_j),Y]
\end{equation}
and
\begin{equation}\label{CDC1}
\cK_1^\dagger(XY) - \cK_1^\dagger(X)Y - X\cK_1^\dagger(Y) = -2\sum_{j=1}^N \lambda_j[Z(\phi_j)W,X][WZ^*(\phi_j),Y] \ .
\end{equation}

To carry out an inductive proof of Theorem~\ref{GENMAINTHM}, it is convenient to rewrite the {\em carr\'e du champ} operators in \eqref{CDC0} and \eqref{CDC1} in terms of a differential calculus of Gross, Shale and Stinespring.  The differential operators specified in the next definition were introduced in \cite[Section 3]{SS65} in connection with what Shale and Stinespring called the {\em holomorphic spinor representation} of the CAR; it is this representation that  corresponds to the GIG quantum states that they studied. The explicit formulas in the definition are due to Gross \cite{Gr75}, as is 
Lemma~\ref{GSSL} and (in an equivalent form) Lemma~\ref{GSSLZ} below.  (Lemma~\ref{GSSLZ}  says  that these differentiation operators do not raise the degree of polynomials.)  Gross \cite{Gr75} utilized this differential calculus to study certain GIG semigroups, and we now give a further demonstration of its utility in this same context. 

\begin{defn}\label{GSSDC} For all $\phi\in \cH_1$, define the operators $\partial_{\phi}$ and $\bar\partial_\phi$ on $\scA_{\cH_1}$ by 
$$
\partial_\phi(X) := Z(\phi)X - \Gamma(X)Z(\phi) \quad{\rm and}\quad  \bar\partial_{\phi}(X) :=  Z^*(\phi_j)X - \Gamma(X)Z^*(\phi_j)\ .
$$
where $\Gamma$ is the principle automorphism; that is, $\Gamma(X) = WXW$. 
\end{defn}

\begin{lm}\label{GSSL} The operators $\partial_{\phi}$ and $\bar\partial_\phi$ are skew-derivations on $\scA_{\cH_1}$. That is, for all $XY\in \scA_{\cH_1}$, 
\begin{equation}\label{GSSL1} 
\partial_{\phi}(XY) = \partial_{\phi}(X) Y + \Gamma(X)\partial_{\phi}(Y) 
\end{equation}
and
\begin{equation}\label{GSSL2} 
\bar\partial_{\phi}(XY) = \bar\partial_{\phi}(X) Y + \Gamma(X)\bar\partial_{\phi}(Y) \ .
\end{equation}
\end{lm}

\begin{proof} Subtracting and adding $\Gamma(X)Z(\phi)Y$ in the definition of  $\partial_{\phi}(XY)$ yields
\begin{eqnarray*}
\partial_{\phi}(XY) &=&  Z(\phi)XY - \Gamma(X)Z(\phi)Y + \Gamma(X)Z(\phi)Y - \Gamma(XY)Z(\phi)\\
&=&  \bigl(Z(\phi)X - \Gamma(X)Z(\phi)\bigr)Y + \Gamma(X)\bigl(Z(\phi)Y - \Gamma(Y)Z(\phi)\bigr)\ .
\end{eqnarray*}
This proves \eqref{GSSL2}, and the proof of \eqref{GSSL2} is essentially the same. 
\end{proof} 

\begin{lm}\label{GSSLZ}  For all $\phi,\psi\in\cH_1$, 
\begin{equation}\label{GSSL3} 
\partial_{\phi}Z(\psi) =0 \quad{\rm and}\quad \partial_{\phi}Z^*(\psi)  = \langle\phi,\psi\rangle\one\ .
\end{equation}
and
\begin{equation}\label{GSSL4} 
\bar \partial_{\phi}Z^*(\psi) =0  \quad{\rm and}\quad \bar\partial_{\phi}Z(\psi)  = \langle\psi,\phi\rangle\one\ .
\end{equation}
Moreover, for $\aa\in \{0,1\}^N$ and $j=1,\dots,N$, define $\aa'_j$ to be the same as $\aa$, but with the $j^{{\rm th}}$ entry set equal to $0$. Then 
\begin{equation}\label{GSSL5} 
\partial_\phi Z_{\bb,\gg} = \sum_{j=1}^N  \langle \phi,\psi_j\rangle \beta_j (-1)^{\sum_{i< j}(\beta_i+\gamma_i)}Z_{\bb'_j,\gg}
\end{equation}
and
\begin{equation}\label{GSSL6} 
\bar \partial_\phi Z_{\bb,\gg} = \sum_{j=1}^N \gamma_j \langle \phi,\psi_j\rangle \gamma_j (-1)^{\beta_j +\sum_{i< j}(\beta_i+\gamma_i)} Z_{\bb,\gg'_j}
\end{equation}
\end{lm}

\begin{proof}  Since 
$Z_{\bb,\gg} = \prod_{j=1}^N(Z^*(\psi_j)^{\beta_j} Z(\psi_j)^{\gamma_j}$ is a product, we may apply the skew-derivation property from Lemma~\ref{GSSL}.  Using the  evident identity $\partial_\phi \one =0$ and using  the first part of \eqref{GSSL3}  from Lemma~\ref{GSSLZ},
one obtains 
$$
\partial_\phi Z_{\bb,\gg} = \sum_{j=1}^N  \Gamma \left(\prod_{i=1}^{j-1}(Z^*(\psi_i)^\beta_i Z(\psi_i)^\gamma_i\right)  
\left(\partial_\phi (Z^*(\psi_j)^\beta_j ) 
Z(\psi_j)^\gamma_j\right) \left( \prod_{k=j+1}^N(Z^*(\psi_k)^\beta_k Z(\psi_k)^\gamma_k\right)\ .
$$
By the second part of \eqref{GSSL3}  from Lemma~\ref{GSSLZ}, $\partial_\phi (Z^*(\psi_j)^\beta_j )  =  \langle \phi,\psi_j\rangle \beta_j \one$.  Then since
$$
\Gamma \left(\prod_{i=1}^{j-1}(Z^*(\psi_i)^\beta_i Z(\psi_i)^\gamma_i\right)   = (-1)^{\sum_{i< j}(\beta_i+\gamma_i)}\prod_{i=1}^{j-1}(Z^*(\psi_i)^\beta_i Z(\psi_i)^\gamma_i\ ,
$$
this proves \eqref{GSSL5}. The proof of \eqref{GSSL6} is essentially the same. 
\end{proof}

\begin{remark} Gross used the differential calculus to build a non-commutative  theory of  the Dirichlet forms analogous to the theory of Dirichlet forms associated to the generators of symmetric Markov semigroups in classical probability. 
He proved \cite{Gr75} a  logartithmic Sobolev inequality involving such a Dirichlet form  associated to the Maximally symmetric Fermionic Mehler semigroup with parameter $\lambda = \frac12$, and conjectured that the sharp version has the same constants that he had found in  the classical case \cite{Gr75B}. His conjecture was later proved in \cite{CL93}. The theory of non-commutative Dirichlet forms has been further developed in  \cite{Cip97,Cip08,CS03}.  Moreover, the differential calculus defined in Definition~\ref{GSSDC} plays an important role in the theory of non-commutative optimal transport associated the the Fermionic Mehler semigroups \cite{CM13,CM17,CM20}. 

\end{remark}

\begin{lm}\label{CDCLEMMA} Let $\cL^\dagger$ be given by \eqref{MIXTURE35}. Then the corresponding {\em carr\'e du champ} operator is given by
\begin{multline}
\cL^\dagger(XY) - \cL^\dagger(X)Y - X\cL^\dagger(Y) =\\ -2\sum_{j=1}^N\left(  (1-\lambda_j)\bar\partial_{\phi_j}\Gamma(X)\partial_{\phi_j}Y + 
\lambda_j \partial_{\phi_j}\Gamma(X)\bar\partial_{\phi_j} Y \right)\ .\label{GSSL7} 
\end{multline}
\end{lm}

\begin{proof}
Because $X \mapsto i[\widehat{K},X]$ is a derivation on $\scA_{\cH_1}$, it does not contribute to the {\em carr\'e du champ}.  
Then since $X = \Gamma(\Gamma(X))$, 
\begin{eqnarray*}
[Z^*(\phi_j)W,X][WZ(\phi),Y] &=& (Z^*(\phi_j)\Gamma(X)- XZ^*(\phi_j)) (Z(\phi_j)Y - \Gamma(Y)Z(\phi_j))\\
&=& \bar\partial_\phi\Gamma(X)\partial_\phi Y\ .
\end{eqnarray*}
Likewise, $[Z(\phi_j)W,X][WZ^*(\phi_j),Y]  = \partial_\phi\Gamma(X)\bar\partial_\phi Y$.  Combining these computations with \eqref{CDC0} and \eqref{CDC1} yields \eqref{GSSL7}.
\end{proof}

\begin{proof}[Proof of Theorem~\ref{GENMAINTHM}] we proceed by induction on $|\bb|+|\gg|$, recalling at the outset  that the identity has been proved for 
$|\bb| + |\gg| \leq 2$. Let $n$ be an integer with $2 \leq n < 2N$.
We suppose that $\cL^\dagger Z_{\bb,\gg} = \cL^\dagger_{G,A} Z_{\bb,\gg} $ has been proved whenever $|\bb|+|\gg| \leq n$. Let $(\bb,\gg)$ be such that 
$|\bb|+ |\gg| = n+1$. We must prove $\cL^\dagger  Z_{\bb,\gg}  = \cL^\dagger_{GA} Z_{\bb,\gg} $ for this choice of 
$(\bb,\gg)$. 

Suppose for some $j$, $|\beta_j|+|\gamma_j| =1$. Define $j_0$ to be the least such $j$, so that for $j < j_0$, $|\beta_j|+|\gamma_j|\in \{0,2\}$. Either $\beta_{j_0} =1$ and $\gamma_{j_0} =0$ or 
$\beta_{j_0} =0$ and $\gamma_{j_0} =1$. We shall suppose that 
$\beta_{j_0} =1$; the other case is similar. Then with $\bb'_{j_0}$ defined as in Lemma~\ref{GSSLZ},
\begin{equation}\label{GENLINFOR2}
Z_{\bb,\gg} = Z(\psi_{j_0})^* Z_{\bb'_{j_0},\gg} \ .
\end{equation}
because in  commuting the factor of $Z^*(\psi_{j_0})$ through to the left, it moves past an even number of terms with which it anti-commutes. Then because $\bar\partial_{\phi_k}Z^*(\psi_{j_0}) =0$
for all $k$ by Lemma~\ref{GSSLZ}, Lemma~\ref{CDCLEMMA} yields 
\begin{eqnarray}\label{GSSL71} 
\cL^\dagger (Z_{\bb,\gg} ) &=& \cL^\dagger(Z^*(\psi_{j_0})) Z_{\bb'_{j_0},\gg} + Z^*(\psi_{j_0}) \cL^\dagger(Z_{\bb'_{j_0},\gg})\nonumber\\
 &+& \sum_{k=1}^N \lambda_k \partial_{\phi_k} (Z^*(\psi_{j_0}) )\bar \partial_{\phi_k}(Z_{\bb'_{j_0},\gg})\ .
\end{eqnarray}
By Lemma~\ref{GSSLZ}, $\partial_{\phi_k}(Z^*(\psi_{j_0})) = \langle \psi_{j_0},\phi_k\rangle\one$, and 
\begin{equation}\label{GSSL72}
\bar \partial_{\phi_k} Z_{\bb'_{j_0},\gg} = \sum_{\ell=1}^N \gamma_\ell \langle \phi_k,\psi_\ell\rangle \gamma_\ell (-1)^{\beta_\ell +\sum_{i< \ell}(\beta_i+\gamma_i)} Z_{\bb'_{j_0},\gg'_\ell}\ .
\end{equation}
Therefore since
\begin{equation}\label{GSSL73}
\gamma_\ell\sum_{k=1}^N \lambda_k \langle \psi_{j_0},\phi_k\rangle \langle \phi_k,\psi_\ell\rangle = \gamma_\ell\langle\psi_{j_0},A\psi_\ell\rangle = \gamma_\ell\delta_{j_0,\ell}\langle\psi_{j_0},A\psi_{j_0}\rangle\ ,
\end{equation}
and since $\gamma_{j_0} =0$,  $\gamma_\ell\delta_{j_0,\ell}=0$ for all $\ell$. The {\em carr\'e du champ} term vanishes and  
\begin{eqnarray*}
\cL^\dagger (Z_{\bb,\gg} ) &=& \cL^\dagger(Z^*(\phi_{j_0})) Z_{\bb'_{j_0},\gg} + Z^*(\psi_{j_0}) \cL^\dagger(Z_{\bb'_{j_0},\gg})\\
&=& Z^*(G^*\phi_{j_0}) Z_{\bb'_{j_0},\gg} + Z^*(\psi_{j_0}) \cL^\dagger_{G,A}(Z_{\bb'_{j_0},\gg})\\
&=& \cL^\dagger_{G,A}(Z_{\bb,\gg} )
\end{eqnarray*}
where in the second equality we have used \eqref{MIXTURE42} and the inductive hypothesis, and in the third equality we have used \eqref{GENLIM2} and the fact that the third term in
$F_{j_0}$ as specified in  
\eqref{GENLIM3}  is zero since $\beta_{j_0}+\gamma_{j_0} =1$.  The case in which  $\gamma_{j_0} =1$ instead of $\beta_{j_0}$ is essentially the same.

Next, consider the case in which $|\beta_j|+ |\gamma_j| \in \{0,2\}$ for all $j$. The case in which $|\beta_j|+ |\gamma_j|= 0$ for all $j$ is trivial. Therefore, fix $j_0$ to be 
the least value of $j$ such that 
$|\beta_{j_0}|+ |\gamma_{j_0}|=2$. Then with $\bb'_{j_0}$ as above, we again have the factorization \eqref{GENLINFOR2}, and the identities \eqref{GSSL71},  
\eqref{GSSL72} and \eqref{GSSL73}. The difference now is that $\gamma_{j_0} =1$, and hence the {\em carr\'e du champ} term is $\langle\psi_{j_0},A\psi_{j_0}\rangle Z_{\bb'_{j_0},\gg'_{j_0}}$.
(Note that  $(-1)^{\beta_{j_0} +\sum_{i< {j_0}}(\beta_i+\gamma_i)} =1$ by our choice of $j_0$). Therefore in this case,
\begin{eqnarray*}
\cL^\dagger (Z_{\bb,\gg} ) &=& \cL^\dagger(Z^*(\phi_{j_0})) Z_{\bb'_{j_0},\gg} + Z^*(\psi_{j_0}) \cL^\dagger(Z_{\bb'_{j_0},\gg}) + \langle\psi_{j_0},A\psi_{j_0}\rangle Z_{\bb'_{j_0},\gg'_{j_0}}\\
&=& Z^*(G^*\phi_{j_0}) Z_{\bb'_{j_0},\gg} + Z^*(\psi_{j_0}) \cL^\dagger_{G,A}(Z_{\bb'_{j_0},\gg}) + \langle\psi_{j_0},A\psi_{j_0}\rangle Z_{\bb'_{j_0},\gg'_{j_0}}\\
&=& \cL^\dagger_{G,A}(Z_{\bb,\gg} )
\end{eqnarray*}
just as above, but now the {\em carr\'e du champ} term provides the third term in $F_{j_0}$ which is not zero when $\beta_{j_0} = \gamma_{j_0} =1$.
\end{proof}

In the next subsection, we investigate an important special case in which one can say much more about the semigroup $\{e^{t\cL_{G,A}}\}_{t\geq 0}$

\subsection{The  Fermionic Mehler semigroups}

The dual gauge group operations $\widehat{\alpha}_{{\rm G},t}$ arise through second quantizing the simplest unitary group on $\cH_1$, namely $\{e^{it}\one\}_{t\in \R}$.  We obtain a  family of equally canonical GIG semigroups on $\scA_{\cH_1}$ starting with the contraction semigroup $\{e^{-t}\one\}_{t\geq 0}$. The generator of this semigroup is $-\one$, and hence the condition
$0\leq A \leq -G - G^*$ becomes $0 \leq A \leq 2\one$.  Making the choice $A := 2\lambda \one$, $0 \leq \lambda \leq 1$, Theorem~\ref{GENMAINTHM} 
yields a one parameter family of   semigroups $\{e^{t\cL_{\one,2\lambda \one}}\}_{t\geq 0}$ of GIG quantum operations on $\scA_{\cH_1}$.

Moreover, since  in this case $e^{sG}Ae^{sG^*}$ is simply $2\lambda e^{-2s}\one$. 
$R_t = \int_0^t e^{sG}A e^{sG^*}{\rm d}s$ reduces to $R_t = \lambda (1- e^{-2t})\one$. Therefore,  by Theorem~\ref{AWQDS} and Theorem~\ref{GENMAINTHM}, 
 the restriction of $e^{t\cL_{\one,2\lambda \one}}$ to $\scG_{\cH_1}$ is the restriction of
$\Phi_{e^{-t}\one,\lambda(1-e^{-2t})\one}$ to $\scG_{\cH_1}$.  In this case, as we shall see, we can dispense with the restrictions: $\{\Phi_{e^{-t}\one,\lambda(1-e^{-2t})\one}\}_{t\geq 0}$ is a GIG semigroup of quantum operations on all of $\scA_{\cH_1}$, and not only restricted to the Araki-Wyss algebra $\scG_{\cH_1}$.  This is an immediate consequence of Theorem~\ref{EIGTHMCOR} below.

\begin{defn} Fix $0 \leq \lambda \leq 1$, and for $t\geq 0$  define 
\begin{equation}\label{MSG1}
\cP_{\lambda,t} := e^{t\cL_{\one,2\lambda \one}} 
\end{equation} 
where $\cL_{\one,2\lambda \one}$ is the Lindblad generator specified in Theorem~\ref{GENMAINTHM}. Then $\{\cP_{\lambda,t}\}_{t\geq 0}$ is the 
{\em maximally symmetric Fermionic Mehler semigroup with parameter $\lambda$}. 
\end{defn}

The maximally symmetric Fermionic Mehler semigroups correspond exactly to the classical Mehler semigroups studied by Mehler in 1866 \cite{M66}, as we shall see.  However, it will be convenient to treat them as part of a   somewhat more general class of semigroups that deserve to be called Fermionic Mehler semigroups.  These are the semigroups generated by $\cL_{G,A}$ where $G$ is self-adjoint, so that we write $G = -H$ as in Theorem~\ref{GENMAINTHM}, and where $[H,A]= 0$.

\begin{defn}\label{FMSGDEF} Fix  $H\geq 0$  on $\cH_1$ (so that $-H$ is a contraction semigroup generator, and an operator $A$ on $\cH_1$ 
 such that $0 \leq A \leq 2H$ and  $[H,A]=0$.
For $t\geq 0$  define 
\begin{equation}\label{MSG2}
\cP^{\phantom{\dagger}}_{(H,A),t} :=  e^{t \cL_{-H,A}}\ . 
\end{equation}  
Then $\{\cP_{(H,A),t}\}_{t\geq 0}$ is the 
{\em  Fermionic Mehler semigroup with parameters $H$ and $A$}. 
\end{defn}

Note that $\cP_{\lambda,t}$ given by \eqref{MSG1} is the same as $\cP^{\phantom{\dagger}}_{(\one,\lambda\one),t}$ given by \eqref{MSG2}.

When $G = -H$ with $H\geq 0$ and $[H,A] =0$, the formula $R_t = \int_0^t e^{sG}Ae^{sG^*}{\rm d}s$ becomes $R_t = \left(\int_0^te^{-2sH}{\rm d}s \right)A$. Then with
\begin{equation}\label{GMEH1}
T = \lim_{\epsilon\downarrow 0}(2H + \epsilon\one)^{-1/2}A(2H + \epsilon\one)^{-1/2}
\end{equation}
 as in Theorem~\ref{GENMAINTHM}, 
\begin{equation}\label{GMEH2}
R_t = (\one - e^{-2tH})T\ .
\end{equation}

\begin{thm}\label{EIGTHMCOR} Let $H$ and $A$ be commuting operators on $\cH_1$ such that $0 \leq A \leq 2H$ and $[H,A] =0$.  Then with 
$\cP^{\phantom{\dagger}}_{(H,A),t}$ given by  \eqref{MSG2} and $T$ given by \eqref{GMEH1},
\begin{equation}\label{EIGTHMCOR1} 
\cP^{\phantom{\dagger}}_{(H,A),t} = \Phi_{e^{-tH},(\one - e^{-2tH})T}\ ,
\end{equation}
for all $t\geq 0$. 
Moreover,  for each $t\geq 0$,
$\cP^\dagger_{(H,A),t}$ is self adjoint with respect to the GNS inner product induced by $\rho_T$ which is a steady state of the semigroup; that is 
$\cP^\dagger_{(H,A),t}\rho_T = \rho_T$ for all $t$.

 If $H>0$,  so that ${\rm ker}(H) =0$, then $\rho_T$ is the unique steady state, and and for all states $\rho$, $\lim_{t\to\infty}\cP^{\phantom{\dagger}}_{(H,T),t}\rho = \rho_T$. 
\end{thm} 

\begin{remark} The fact that $\{\Phi_{e^{-tH},(\one - e^{-2tH})T}\}_{t\geq 0}$ is a semigroup when $[H,T] =0$ follows from a result of  Evans (see \cite[(3.8)]{E79}). However his 
approach is quite different and does not involve the generator of this semigroup, which he did not compute.
\end{remark}

\begin{remark} Self-adjointness with respect to the GNS inner product induced by a steady state is a physically important property known in the physical literature as {\em detailed balance}.
It is closely connected  with time-reversal and with the notion of thermodynamic equilibrium. See \cite{AC20,Ar73,DL92,FR15,FU07,KFGV,MS98}.

\end{remark}

\begin{proof}[Proof of Theorem~\ref{EIGTHMCOR}] Temporarily define $\cP_t := \Phi_{e^{-tH},(\one - e^{-2tH})T}$.  We claim that $\{\cP_t\}_{t\geq0 }$ is a semigroup. By Corollary~\ref{GENMAINCL}
$$
e^{t\cL_{-H,A}} = \lim_{n\to \infty}\left(\Phi_{e^{-(t/n)H},(\one - e^{-2(t/n)H})T}\right)^n\ .
$$
However when $\{\cP_t\}_{t\geq0 }$ is a semigroup, 
\begin{equation}\label{GMEH3}
\Phi_{e^{-tH},(\one - e^{-2tH})T} = \left(\Phi_{e^{-(t/n)H},(\one - e^{-2(t/n)H})T}\right)^n
\end{equation}
 for all $n$, and \eqref{EIGTHMCOR1}  follows trivially from this.

To prove that  $\{\cP_t\}_{t\geq0 }$ is a semigroup, it is convenient to work with  $\{\cP^\dagger_t\}_{t\geq0 }$. Let $s,t\geq 0$. 
Then by Theorem~\ref{EIGTHM}, and the definition \eqref{EIGTHMCOR1}, for all $\aa$,
\begin{equation}\label{GMEH4}
\cP^\dagger_{s+t}H_\aa = e^{-(s+t)\sum_{j=1}^N|\alpha_j|\lambda_j}H_\aa = \cP_s^\dagger( e^{-t\sum_{j=1}^N|\alpha_j|\lambda_j}H_\aa) = \cP_s^\dagger \cP_t^\dagger H_\aa\ .
\end{equation}
By Theorem~\ref{EIGTHM} $\{H_\aa \:\ \aa\in (\{0,1\}\times \{0,1\})^N\}$ is a basis for $\scA_{\cH_1}$  and hence for all $s,t$, $\cP^\dagger_{s+t} = \cP^\dagger_s \cP^\dagger_t$.
This is all we need to establish \eqref{GMEH3}, but the continuity that we require of semigroup, namely
$\lim_{t\to 0}\cP_t^\dagger = \one$, is also evident from  the diagonalization of $\cP_t^\dagger$ provided by Theorem~\ref{EIGTHM}, which also yields the self-adjointness with respect to the GNS inner product induced by $\rho_T$.

By Theorem~\ref{EIGTHM}, $\cP^{\dagger}_{(H,T),t}\rho_T = \rho_T$ for all $t$. Also by Theorem~\ref{EIGTHM}, the $1$-eigenspace of $\cP^{\dagger}_{(H,T),t}$ is spanned by $\one$,
and all other eigenvalues lie in $(0,1)$. Therefore, the $1$-eigenspace of $\cP^{\phantom{\dagger}}_{(H,T),t}$ is spanned by $\rho_T$, and all other eigenvalues of 
$\cP^{\phantom{\dagger}}_{(H,T),t}$ lie in $(0,1)$. Therefore, when $H > 0$, $\lim_{t\to\infty} \cP^{\phantom{\dagger}}_{(H,T),t}(\rho) = \rho_T$ for all states $\rho$.
\end{proof}

In the final part of Theorem~\ref{EIGTHMCOR}, the condition that $H>0$ can be dispensed with using the factorization property proved in Theorem~\ref{FACTORIZTHM}. Let $\cK$ denote ${\rm ker}(H)$, and write $\cH_1 = \cK \oplus \cK^\perp$. Then both $\cK$ and $\cK^\perp$ are invariant under $H$ and $T$. By 
Theorem~\ref{FACTORIZTHM} together with Theorem~\ref{EIGTHMCOR},  for all $X\in \scA_\cK$ and all $Y\in \scA_{\cK^\perp}$, 
$\cP^{\phantom{\dagger}}_{(H,T),t}(XY) = \cP^{\phantom{\dagger}}_{(H,T),t}(X) \cP^{\phantom{\dagger}}_{(H,T),t}(Y)$,
and by Theorem~\ref{EIGTHM}, $\cP^{\phantom{\dagger}}_{(H,T),t}(X) = X$.  That is,
\begin{equation}\label{FactorizationNB}
\cP^{\phantom{\dagger}}_{(H,T),t}(XY) = X \cP^{\phantom{\dagger}}_{(H,T),t}(Y)
\end{equation}
Thus we loose noting in restricting our attention to the action of $\cP^{\phantom{\dagger}}_{(H,T),t}$ on $\scA_{\cK^\perp}$, and by definition, $H > 0$ on $\cK^\perp$.  Suppose that $0 < T < \one$ on $\cK^\perp$. Then by Theorem~\ref{EIGTHM}, $\rho_T$ spans the null space of $\cL_{-H,A}$, and $\lim_{t\to\infty}\cP^{\phantom{\dagger}}_{(H,T),t}(Y) = \tau(Y)\rho_T$. 
Since $\tau(Y)X = {\mathbb E}_{\cK,\tau}(XY)$, 
\begin{equation}\label{EIGTHMCOR1Q} 
\lim_{t\to\infty}\cP^{\phantom{\dagger}}_{(H,T),t}\rho = {\mathbb E}_{{\rm ker}(H),\tau}(\rho_0) \rho_T\ .
\end{equation}
It follows that the density matrices $\rho$ in $\scA_{\cH_1}$ satisfying $\cP^{\phantom{\dagger}}_{(H,T),t}\rho= \rho$ for all $t$ are precisely those of the form
\begin{equation}\label{EIGTHMCOR1Q2}
\rho = \sigma \rho_T
\end{equation}
where $\sigma$ is any density matrix in $\scA_{\cK}$. (Since $\rho_T$ is even, $\rho_T\in \scA_{\cK}'$).  In particular, if $\cK$ is not trivial, there are non-Gaussian steady states. 
The density matrix $\rho$ in \eqref{EIGTHMCOR1Q2} is a GIG density matrix  if and only if $\sigma$ is a GIG density  matrix in $\scA_{\cK}$.  It follows that the GIG density matrix  $\rho_Q$ on $\scA_{\cH_1}$ is a steady sate density matrix if and only if 
$$
T \leq Q \leq T+P
$$
where $P$ is the orthogonal projection onto $\cK$ in $\cH_1$.  It is not hard to generalize these considerations to the case in which $0$ or $1$ or both are eigenvalues of the restriction of $T$ to $\cK^\perp$.

We close with some remarks on a further generalization of the Fermionic Mehler semigroups. 
Given a contraction semigroup $\{e^{tG}\}_{t\geq 0}$, even when $G$ is not self-adjoint,  there is a natural one parameter family of ``second quantizations'' of it that we have already encountered in the case $G = -\one$.
For $0\leq \lambda \leq 1$, define 
$$
A_\lambda = \lambda(-G + G^*)\ ,
$$
and then 
$$
R_{\lambda,t} := \int_0^t e^{s G} A_\lambda  e^{tG^*} {\rm d}s = \lambda\left( \one - e^{tG}e^{tG^*}\right)\ .
$$
Then defining $\cP_t := \Phi_{e^{tG},R_{\lambda,t}}\big|_{\scG_{\cH_1}}$,  by Theorem~\ref{AWQDS}, $\{\cP_t\}_{t\geq 0}$ is a semigroup of quantum operations on $\scG_{\cH_1}$ that maps ${\mathfrak S}_{GIG}$ into itself, and such that
the symbol map of $\cP_t$ is $Q\mapsto R_{\lambda,t} + e^{tG}A e^{tG^*}$. 
Taking $G=-\one$, we obtain the maximally symmetric Fermionic Mehler semigroups. as another special case of Theorem~\ref{GENMAINTHM}, $\{\cP_t\}_{t\geq 0}$ has an extension to all of $\scA_{\cH_1}$. The results of the next section provide, among other things,  another proof of the extension property that provides additional information in this special case.

 \subsection{The cone of GIG semigroup generators}

\begin{thm}\label{GIGCONE} The set of GIG semigroup generators is a convex cone. That is, if $\cL_1$ and $\cL_2$ are GIG semigroup generators, then for all $a,b \geq 0$,  
$a\cL_1 + b\cL_2$ is a GIG semigroup generator.
\end{thm}

\begin{proof} This is a consequence of the Trotter Product Formula. 
$$
e^{t(a\cL_1 + b\cL_2)} = \lim_{n\to \infty} \left(e^{(ta/n)t\cL_1} e^{(tb/n)t\cL_2}\right)^n\ .
$$
Then since for each $t$ and $n$, $e^{(ta/n)t\cL_1}$ and $e^{(tb/n)t\cL_2}$ map ${\mathfrak S}_{GIG}$ into itself, so does $e^{t(a\cL_1 + b\cL_2)}$. 
\end{proof}

The Trotter Product Formula argument used to prove Theorem~\ref{GIGCONE} allows us to  compute the symbol of $e^{t(\cL_1 + \cL_2)}\rho_R$ in a significant class of examples.

\begin{thm}\label{TPSYMBOL} Let $T\in \scQ$, and let $H\geq 0$ be an operator on $\cH_1$ such that $[H,T]=0$ and ${\rm ker}(H)\subseteq {\rm ker}(T)$.
Let  $\cL_1^\dagger$  be  the generator of the Fermionic Mehler semigroup $\{\Phi^\dagger_{e^{-tH},(\one - e^{-2tH})T}\}_{t\geq 0}$.
The explicit form of $\cL_1^\dagger$, and hence $\cL_1$, the generator of the dual GIG semigroup, is given by Theorem~\ref{GENMAINTHM}.  Let
$\cL_2 = \cL_2^\dagger $ be the coherent generator $\cL_2(X) = i[\widehat{K},X]$ for some self-adjoint $K$ on $\cH_1$. 

Then  the symbol map of the GIG semigroup $\{e^{t(\cL_1+\cL_2)}\}_{t\geq 0}$ is given by
$$
e^{-t(H+iK)}Q e^{-t(H-iK)} + \int_0^t e^{-s (H+iK)}(HT+TH)e^{-s (H-iK)} {\rm d}s\
$$

\end{thm}

\begin{proof}
For any $Q\in \scQ$, 
\begin{eqnarray*}
Q_{e^{t\cL_1}(\rho_Q)}  &=& e^{-tH}Q e^{-tH} +  (\one - e^{-t2H})T\\
&=& e^{-tH}(Q  t(HT+TH) ) e^{-tH}  + \mathcal{O}(t^2)\ .
\end{eqnarray*}
and
$$
Q_{e^{t\cL_2}(\rho_Q)}  = e^{-itK}Q e^{itK} \ .
$$
It follows that the symbol of
${\displaystyle 
\left(e^{(t/n)t\cL_1} e^{(t/n)t\cL_2}\right)^n\rho_Q}$
is
\begin{multline*}
(e^{-i(t/n)K}e^{-(t/n)H})^n Q (e^{-(t/n)H}e^{i(t/n)K})^n\\ + \sum_{m=0}^{n-1}  \frac{t}{n}\left(e^{-i(t/n)K}e^{-(t/n)H}\right)^m(HT+TH) \left(e^{-(t/n)H}e^{i(t/n)K}\right)^m +{\mathcal O}\left(\frac{t}{n}\right)^2\ .
\end{multline*}
Taking the limit $n\to \infty$ yields
$$
Q_{e^{t(\cL_1 + \cL_2)}\rho_Q} =  e^{-t(H+iK)}Q e^{-t(H-iK)} + \int_0^t e^{-s (H+iK)}Ae^{-s (H-iK)} {\rm d}s\ .
$$
\end{proof}

\begin{thm}\label{EXTENSIONA} Let 
  $\left(G,A\right)$ be a pair consisting of 
a contraction semigroup generator $G$ on $\cH_1$  together with an operator $A$ on $\cH_1$ satisfying $0 \leq A \leq -G -G^*$. 
Define
\begin{equation}\label{EXTENSIONA1}
H := -\frac12(G+G^*) \quad{\rm and}\quad K := \frac{1}{2i}(G-G^*)\ ,
\end{equation}
 and define $R_t := \int_0^t e^{sG}Ae^{sG^*}{\rm d}S$,  as in Theorem~\ref{AWQDS}. Let $T\in \scQ$ be the unique solution of $HT+TH = A$ with ${\rm ker}(H) \subseteq{\rm ker}(T)$. 
 
 By Theorem~\ref{AWQDS} there is a unique semigroup $\{\cP_t\}$ of quantum operations on $\scG_{\cH_1}$ that maps ${\mathfrak S}_{GIG}$ into itself such that for each $t\geq 0$, the symbol map of $\cP_t$ is $Q \mapsto R_t + e^{tG}Qe^{tG^*}$. 

Under the condition that 
\begin{equation}\label{EXTENSIONA2}
[H,A] = 0 \quad{\rm or\ equivalently}\ ,\quad [H,T] =0\ .
\end{equation}
$\{\cP_t\}_{t\geq 0}$ has an  extension to all of $\scA_{\cH_1}$ which is given by  $\{e^{t(\cL_1+\cL_2)}\}_{t\geq 0}$
where $\cL_1$ is the generator of the Fermonic Mehler semigroup $\{\Phi_{e^{-tH},(\one - e^{-2tH}T)}\}_{t\geq 0}$, and where
$\cL_2(X) = i[\widehat{K},X]$, $K = \frac{1}{2i}(G-G^*)$. 
\end{thm}

\begin{proof} Define $T = \int_0^\infty e^{-tH}Ae^{-tH}{\rm d}t$. Then  $HT+TH =  \int_0^\infty e^{-tH}(HA+AH)e^{-tH}{\rm d}t = A $, and evidently this is the unique solution with ${\rm ker}(H) \subseteq {\rm ker}(T)$. 
 When $H$ and $A$ commute, $T = (\one - e^{-2tH})A$, and so $T\in \scQ$. 
 
For $t\geq 0$, define $S_t := e^{tG} = e^{-t(H+iK)}$ so that $\{S_t\}_{t\geq 0}$ is a contraction semigroup on $\cH_1$, and define 
$$
R_t := \int_0^t S_s (HT+TH) S^*_s {\rm d}s\ .
$$
Note that ${\displaystyle S_s R_t S_s^* = \int_0^t S_{s+t} (HT+TH)  S^*_{s+t} {\rm d}s = \int_s^{t+s}S_r (HT+TH) S_r^*{\rm d}r}$, and hence
\begin{equation}\label{Mehler60}
R_{t+s} = S_s R_t S_s^* + R_s \quad{\rm and}\quad R_t \leq \one - S_tS_t^*\ .
\end{equation}
Then 
$$
A = \lim_{t\downarrow 0}\frac{1}{t} R_t = HT+TH\ .
$$
Then $[H,A] = H[H,T]+ [H,T]H$ and hence  under the condition ${\rm ker}(H) \subseteq {\rm ker}(T)$, $[H,A] = 0$ if and only if $[H,T] =0$. 

By Theorem~\ref{TPSYMBOL}, $e^{t(\cL_1+\cL_2)}$ has the symbol map $Q \mapsto R_t + S_tQS_t^*$, and therefore
 the restriction of $e^{t(\cL_1+\cL_2)}$ to $\scG_{\cH_1}$ is the same as the restriction of $\Phi_{S_t,R_t}$ to $\scG_{\cH_1}$. 
\end{proof}

This result extends a result of Evans proved in the appendix to  \cite{E79}. He proved the semigroup property for $\Phi_{S_t,R_t}$ when  each $S_t$ commutes with $T$, though the result is expressed somewhat differently and the argument is completely different. He imposes a requirement that amounts to $[G,T] =0$ whereas we require only $[G+G^*,T] =0$. However, it is not entirely clear that his extension is the same as ours when $T$ commutes with $G$.

\subsection{Embedding in semigroups}

Let $(S,R)$ be any admissible pair with $S$ invertible, which is a necessary condition for $\Phi_{S,R}$ to be invertible. Every invertible square matrix has a logarithm, 
and hence infinitely many logarithms. Thus it is always possible to write an invertible contraction $S$ in the form $S = e^G$.  However, it may be that no choice of $G$ 
generates a contraction semigroup. For example, consider
$L := \left[\begin{array}{cc} -\log(2) & 2\\ 0 & -\log(2)\end{array}\right]$. Then $e^{tL} = 2^{-t} \left[\begin{array}{cc} 1 & 2t\\ 0 & 1\end{array}\right]$. 
Simple computation show that $S=e^{3L}$ is a contraction. All solutions $G$ of $e^G = S$ are of the form $3L + k2\pi i$, $k\in \Z$,  For  such $G$, 
$\|e^{\frac13 G}\| = \|e^{L}\| = 1 +\sqrt{2}$, so that $\{e^{tG}\}_{t\geq 0}$ is not a contraction semigroup for any choice of $G$. 

The question that we take up now is this: 

\medskip

\noindent{$\bullet$}  
Suppose that $S$ is a contraction on $\cH_1$ that can be written as $S = e^G$  where $G+ G^* \leq 0$ so that $G$ is a contraction semigroup generator. 
Then  given $R$ such that $0 \leq R \leq \one - SS^*$, when does there exist
 a semigroup $\{\cP_t\}_{t\geq 0}$ of quantum operations on $\scG_{\cH_1}$ such that $\Phi_{S,R} = \cP_1$? 

\medskip

If we can find an operator $A$ satisfying $0 \leq A \leq -(G+G^*)$ such that 
\begin{equation}\label{RTOAEQ}
R = \int_0^1 e^{t G} A e^{tG^*} {\rm d}t\ ,
\end{equation}
and we then define $R_t$ by \eqref{GIGSGTHM3}, it will follow from Theorem~\ref{AWQDS} that  defining $\cP_t = \Phi_{e^{tG},R_t}$  yields a semigroup $\{\cP_t\}_{t\geq 0}$ of GIG quantum operations  such that  $\Phi_{S,R} := \cP_1$.

\begin{thm}\label{EMBED} Let $(S,R)$ be any admissible pair with $S = e^G$ and $G + G^* \leq 0$.  Then there is a semigroup $\{\cP_t\}_{t\geq 0}$ of GIG quantum operations  such that  $\Phi_{S,R} := \cP_1$
if and only if 
\begin{equation}\label{EMBED1}
0 \leq \sum_{n=0}^\infty S^n (-GR -RG^*) S^{*n} \leq \one - SS^*\ .
\end{equation}
If $G$, or equivalently $S$, commutes with $R$ and is normal, then both inequalities in \eqref{EMBED1} are satisfied, but there exist compatible pairs $(S,R)$ such that $S$ is invertible, but
\eqref{EMBED1} is not satisfied, and thus such that $\Phi_{S,R}$ cannot be embedded in a semigroup of GIG quantum operations. 
\end{thm}

\begin{proof}
The equation \eqref{RTOAEQ} can be solved for $A$: Since $S= e^G$,  
$$S^n \left( \int_0^1 e^{t G} A e^{tG^*} {\rm d}t\right)S^{*n} = \int_n^{n+1} e^{t G} A e^{tG^*} {\rm d}t\ ,$$
and therefore
$$
\sum_{n=0}^\infty S^n R S^{*n} = \int_0^\infty e^{t G} A e^{tG^*} {\rm d}t = (-\Lambda_G)^{-1}(A)\ .
$$
Therefore, 
$$A = -\Lambda_G \left(\sum_{n=0}^\infty S^n R S^{*n}\right) = -\sum_{n=0}^\infty S^n \Lambda_G(R) S^{*n}\ .$$
To embed $\Phi_{S,R}$ in a GIG semigroup, we require $0 \leq A \leq -G -G^*$. 
Because $(S,R)$ is compatible, $0 \leq R \leq \one - SS^*$.  While $(-\Lambda_G)^{-1}$ is positivity preserving, it is not in general the case that $-\Lambda_G$ is positivity preserving.  However, if
$R$ commutes with $G$, or what is the same thing, $R$ commutes with $S$ and  $S^*$,  $-\Lambda_G(R) = R^{1/2}(-G -G^*)R^{1/2} \geq 0$, so that in this case, $A \geq 0$.  
Also in this case, 
$$-\Lambda_G(R) = (-G - G^*)^{1/2}R(-G - G^*)^{1/2} \leq   (-G - G^*)^{1/2}(\one - SS^*)(-G - G^*)^{1/2}\ .$$
Define $B := (-G - G^*)^{1/2}(\one - SS^*)(-G - G^*)^{1/2}$. If furthermore $S$ is normal,  
$$\sum_{n=0}^\infty S^nS^{*n} = \sum_{n=0}^\infty (SS^*)^n  = (\one - SS^*)^{-1}\ .$$
In this case also $G$ is normal and $B$ commutes with $G$ and $G^*$. Therefore,
$$
A \leq  (\one - SS^*)^{-1} B = -G -G^*\ .
$$
However, simple examples show that if $G$ does not commute with $R$, $A$ may not be positive. For example, take 
$G = \left[\begin{array}{cc} -1 & \phantom{-}2\\  \phantom{-}0 & -1\end{array}\right]$.  
Since $G+G^* =   \left[\begin{array}{cc} -2 & \phantom{-}2\\  \phantom{-}2 & -2\end{array}\right]$, $G$ is a contraction semigroup generator. Then $S := e^G = e^{-1}\left[\begin{array}{cc} 1 & 2\\  0 & 1\end{array}\right]$ and
$SS^* =  e^{-2}\left[\begin{array}{cc} 5 & 2\\  2 & 1\end{array}\right]$. Choose $R := \frac{1}{15}(\one -SS^*)^{1/2}\left[\begin{array}{cc} 8 & 7\\  7 & 8\end{array}\right](\one -SS^*)^{1/2}$
so that $0 \leq R \leq \one - SS^*$ is satisfied. The matrix powers are simple to compute; $S^n = e^{-n}\left[\begin{array}{cc} 1 & 2n\\  0 & 1\end{array}\right]$, and hence one can compute $A$ and see that is has one strictly negative eigenvalue.  This gives and example of a compatible pair $(S,R)$ with $S=e^G$, $G$ a contraction semigroup generator, such that $\Phi_{S,R}$ cannot be embedded in GIG quantum semigroup. 
\end{proof}

\section{Invariant states of quantum operations on $\scG_{\cH_1}$}\label{INVSTATES}

We have already studied states left invariant by semigroups of quantum operations that map ${\mathfrak S}_{GIG}$ into itself. While the main focus of this work is on semigroups, it is also of interest to
treat individual quantum operations.  

\begin{thm}\label{INVSTATETHM} Let $(S,R)$ be a compatible pair. Suppose that $\lim_{n\to\infty}S^n = P$ exists. Then $P$ is a projection, but not necessarily orthogonal and $PRP^* =0$. Moreover:

\medskip
\noindent{\it (1)}  The sum 
\begin{equation}\label{INVSTATE0A}
R_\infty := \sum_{n=0}^\infty S^nRS^{*n}
\end{equation} 
 converges and satisfies $0 \leq  R_\infty \leq \one -PQP^*$.

\medskip
\noindent{\it (2)}  A density matrix $\rho\in \scG_{\cH_1}$ satisfies $\Phi_{S,R}\rho = \rho$ if and only if $\rho\in {\mathfrak S}_{GIG}$ and has a symbol of the form
\begin{equation}\label{INVSTATE0}
Q =   R_\infty + PQ'P \quad{\rm for \ some}\quad Q'\in \scQ\ .
\end{equation}
\end{thm}

\begin{proof}
By definition, for any $Q'\in \scQ$, the symbol of $\Phi_{R,S}(\rho_Q')$ is $R + SQ'S^*$. Iterating, the symbol of $\Phi^n_{R,S}(\rho_Q')$ is
\begin{equation}\label{INVSTATE1}
\sum_{m=0}^{n-1} S^m R S^{*m} + S^nQ'S^{*n}\ .
\end{equation}
Because $\lim_{n\to\infty}S^n = P$ (where $P$ may be zero), $P^2 = \lim_{n\to\infty}S^{2n} = P$.
Likewise, $SP = \lim_{n\to\infty}S^{n+1} = PS = P$. Taking adjoints we have altogether,
$$
PS = P = SP \quad{\rm and}\quad P^* S^* = P^* = S^*P^*\ .
$$
Since $0 \leq R \leq \one - SS^*$, $0 \leq PRP^* \leq PP^* - PSS^*P^* = 0$, so that $PRP^* =0$.  
By \eqref{INVSTATE1}, $R_n := \sum_{m=0}^{n-1} S^m R S^{*m}\in \scQ$ for all $n$, and in particular, this sum of non-negative terms is bounded above by $\one$ for all $n$. It follows that $\sum_{m=0}^{\infty} S^m R S^{*m}$ converges to   an element $R_\infty$ of  $\scQ$. Since $(R_n,S^n)$ is compatible for all $n$, so is $(R_\infty,P) = \lim_{n\to\infty}(R_n,S^n)$.

It follows that for any $Q'\in \scQ$,  the symbol of $\Phi^n_{R,S}(\rho_{Q'})$ converges in the limit $n\to\infty$  to $R_\infty + PQ'P$.  Therefore
$$
\lim_{n\to\infty}\Phi_{S,R}^n (\rho_{Q'}) = \rho_{(R_\infty + PQ'P)}\in {\mathfrak S}_{GIG}\ .
$$
Now suppose that the range of $P$ is at least two dimensional. Let $\psi$ and $\phi$ be unit vectors such that $P\psi = \eta$ and $P\phi = \xi$ where $\{\eta,\xi\}$ is linearly independent. 
Define $Q_0 = |\psi\rangle\langle\psi|$ and $Q_1= |\phi\rangle\langle\phi|$.
Then 
$$
\rho_{R_\infty + \frac12 |\eta\rangle\langle\eta|} \quad{\rm and} \quad  \rho_{R_\infty + \frac12 |\xi\rangle\langle\xi|}
$$
are GIG invariant states for $\Phi_{S,R}$. But then so is any non trivial  convex combination of these two states,
and by Wolfe's Lemma, it does not belong to ${\mathfrak S}_{GIG}$. 

However,  if $P= 0$, then for all $Q'\in \scQ$, $\lim_{n\to\infty} \Phi^n_{S,R}(\rho_{Q'}) = \rho_{R_\infty}$, and then since ${\mathfrak S}_{GIG}$ spans $\scG_{\cH_1}$, 
$\lim_{n\to\infty}\Phi^n_{S,R}(\rho) = \rho_{R_\infty}$ for all density matrices $\rho\in \scG_{GIG}$. Thus, $\rho_{R_\infty}$ is the unique invariant state of $\Phi_{S,R}$ in $\scG_{\cH_1}$. 
suppose that $\eta$ spans the range of $P$. 

Finally suppose $P$ is rank one, and that $\eta$ spans the range of $P$. Define 
$$t_0 := \min\{ t\in \R\ :\  R_\infty + t|\eta\rangle\langle\eta| \in \scQ\} \quad{\rm and}\quad  t_1 := \max\{ t\in \R\ :\  R_\infty + t|\eta\rangle\langle\eta| \in \scQ\}\ .$$
Then define $Q_0 := R_\infty + t_0 |\eta\rangle\langle\eta|$ and $Q_1 := R_\infty + t_1 |\eta\rangle\langle\eta|$. 
For any $Q'\in \scQ$,
$$
\lim_{n\to\infty}\Phi_{S,R}^n(\rho_{Q'}) = \rho_{R_\infty +  s_{Q'}|\eta\rangle\langle \eta|}
$$
for some $s_{Q'}$, and evidently $t_0 \leq s_{Q'} \leq t_0$ so there is a $\lambda_{Q'}\in [0,1]$ such that 
$$
R_\infty +  s_{Q'} = (1-\lambda_{Q'})Q_0 + \lambda_{Q'}\ .
$$
Now let $\rho$ be any density matrix in $\scG_{\cH_1}$.  By the Araki-Wyss Theorem, Theorem~\ref{AWTHM}, we may write 
$$
\rho = \sum_{n} s_n \rho_{Q'_n}
$$
where the $s_n$ are real and $\sum_n s_n =1$. Then 
$$
\lim_{n\to\infty}\Phi_{S,R}^n(\rho) = (1-s)\rho_{Q_0} + s\rho_{Q_1}
$$
for some $s\in \R$. We claim $s\in [0,1]$.  Evidently $(1-s)\rho_{Q_0} + s\rho_{Q_1}$ is a density matrix, and its symbol is $(1-s)Q_0 + s Q_1 =  R_\infty + s(t_1- t_0)|\eta\rangle\langle \eta|$.
Since $t_0 \leq 0$, if $s>1$, $s(t_1- t_0) \geq t_1$ which would meant that $(1-s)Q_0 + s Q_1 \notin \scQ$, which is impossible. Since $t_1 > t_0$, if $s < 0$,  $s Q_1 \geq  R_\infty$ would not be satisfied, so this too is impossible. Therefore, $s\in [0,1]$ and hence $(1-s)Q_0 + s Q_1\in {\mathfrak S}_{GIG}$ by Wolfe's Lemma. 

$$
Q := \sum_{m=0}^{\infty} S^m R S^{*m}  + PQ'P^*\in \scQ
$$
This proves {\it (1)}.
Note that for any choice of $Q'\in \scQ$, 
$$R + SQ S^* = R + \sum_{m=0}^{\infty} S^{m+1} R S^{*(m+1)} + SPQ'S^*P^* = Q\ .$$
This shows that the set of symbols specified in \eqref{INVSTATE0} is precisely the set of fixed points of the symbol map of $\Phi_{S,R}$.  Evidently, any state $\rho\in {\mathfrak S}_{GIG}$ with such a symbol is invariant under $\Phi_{S,R}$.  
\end{proof}

\appendix

 \section{The graded tensor product construction}\label{GTPA}
 
 Let $\cH_1$ and $\widetilde{\cH}_1$ be finite dimensional Hilbert spaces. Let $Z$ be a CAR field over $\cH_1$ acting irreducibly on $\scK$. 
 Likewise, let $\widetilde{Z}$ be a CAR field over $\widetilde{\cH}_1$ acting irreducibly on $\widetilde{\scK}$. Out of these ingredients, we now construct a CAR field $\widehat{Z}$ over $\cH_`\oplus \widetilde{\cH}_1$ acting irreducibly on $\scK\otimes \widetilde{\scK}$.
 
 Define $W := e^{i\pi\cN}$, where $\cN$ is the number operator on $\scK$ associated to the CAR field $Z$. Evidently $W$ is unitary, and since the spectrum of $\cN$ consists of integers, 
 it is self-adjoint.
 
For $\psi\oplus \phi \in \cH_1\oplus \widetilde{\cH}_1$, define
\begin{equation}\label{GTP1}
\widehat{Z}(\psi\oplus \phi ) = Z(\psi)\otimes \one + W\otimes \widetilde{Z}(\phi)\ .
\end{equation}
Since $W$ anti-commutes with $Z(\psi)$ for all $\psi\in \cH_1$, one readily checks that $\widehat{Z}$ is a CAR field over $\cH_`\oplus \widetilde{\cH}_1$ that acts irreducibly  on 
$\scK\otimes \widetilde{\scK}$.  This method of ``pasting together'' two CAR fields is known as the {\em graded tensor product construction}, and of course it may be iterated. The use of
$W = e^{i\pi\cN}$ in this context goes back to Jordan and Wigner \cite{JW28} in their original construction of a CAR field. 

There is a variant of \eqref{GTP1} that is sometimes more convenient. Let $\widetilde{\cN}$ be the number operator on ${\widetilde \scK}$ associated to the CAR field $\widetilde{Z}$.
Define $\widetilde{W} := e^{i\pi \widetilde{N}}$. Define
\begin{equation}\label{GTP2}
\widecheck{Z}(\psi\oplus \phi ) = Z(\psi)\otimes \widetilde{W} + \one \otimes \widetilde{Z}(\phi)\ .
\end{equation}
Then this too is a CAR field over $\cH_1\oplus{\widetilde \cH}_1$ acting irreducibly on $\scK\otimes \widetilde{\scK}$.

The graded tensor product construction is also useful for understanding the relation between various $*$-subalgebras of $\scA_{\cH}$. 

\begin{thm}\label{AHCOMTHM} Let $Z$ be a CAR field over $\cH_1$ acting irreducibly on $\scK$. Let $\cH$ be any non-trivial subspace of $\scH_1$, and let $\scA_{\cH} = W^*(\{Z(\psi)\ :  \psi\in \cH\})$, as in \eqref{AHDEF}.   Let $\cN_\cH$ be the number operator associated to the restriction of $Z$ to $\cH$. That is, for any orthonormal basis $\{\psi_1,\dots,\psi_n\}$ of $\cH$,
$\cN_\cH = \sum_{j=1}^n Z^*(\psi_j)Z(\psi_j)$. 

Then $\scA_{\cH}'$, the commutant of $\scA_{\cH}$, is given by
\begin{equation}\label{AHCOMTHM1}
\scA_{\cH}' = W^*(\{e^{i\pi \cN_{\cH}}\Z(\psi)\ :  \psi\in \cH^\perp\})\ .
\end{equation}
\end{thm}

\begin{proof} Let $A$ and $B$ denote the restrictions of $Z$ to $\cH$ and $\cH^\perp$ respectively, reduced so that they act irreducibly on subspaces $\scL$ and $\scM$, respectively, of $\scK$. Define $W_\cH := e^{i\pi \cN_{\cH}}$.  Then ${\rm dim}(\scL) = 2^n$ and ${\rm dim}(\scM) = 2^{N-n}$. Let $\widehat{Z}$ be the CAR field over $\cH\oplus \cH^\perp$ acting on $\scL\otimes \scM$ given by
$$
\widehat{Z}(\psi\oplus\phi) := A(\psi)\otimes \one + W_{\cH}\otimes B(\phi)\ ,
$$
which is the graded tensor product of the fields $A$ and $B$. Evidently $\widehat{Z}$ is unitarily equivalent to $Z$.  
In the obvious way, we may identify $\scA_{\cH_1} = \cB(\scK)$ with the subalgebra $\cB(\scK)\otimes \one$ of $\scA_{\cH_`\oplus \widetilde{\cH}_1} = \cB(\scK \otimes \widetilde{\scK})$.
The commutant of $\cB(\scK)\otimes \one$ in $\cB(\scK \otimes \widetilde{\scK})$ is $\one \otimes \cB(\widetilde{\scK})$.

Then map $Z(\phi) \mapsto \widehat{Z}(\psi\oplus 0)$ extends to a $*$-isomorphism of $\scA_{\cH}$ with the subalgebra $\cB(\scL)\otimes \one$ of $\cB(\cL\otimes \cM)$.  
The commutant of $\cB(\scL)\otimes one$ in $\cB(\cL\otimes \cM)$ is $\one \otimes \cB(\scM)$.  It follows that the dimension of $\scA_{\cH}' $ is $2^{2(N-n)}$. 

Evidently, $W_\cH$ is unitary, and since the spectrum of $\cN_\cH$ consists of integers, $W_\cH$ is self adjoint.  By \eqref{NUMOPDEF2I}, 
\begin{equation}\label{WHDEF2}
W_\cH Z(\varphi)W_\cH =  \begin{cases} -Z(\varphi)  & \varphi\in \cH \\ \phantom{-}Z(\varphi) &\varphi\in \cH^\perp\end{cases}\ ,
\end{equation}
It follows that for every $\phi\in \cH^\perp$, $W_\cH Z(\phi)$  commutes with $Z(\psi)$ and $Z^*(\psi)$ for every $\psi\in \cH$. Therefore, 
$W_\cH Z(\phi)\in \scA_{\cH}'$ for every $\phi\in \cH^\perp$.  Let $\{\phi_1,\dots\phi_{N-n}\}$ be an orthonormal basis of $\cH^\perp$. Then for $\bb,\gg\in \{0,1\}^{M-n}$, define
$Z_{\bb,\gg} = \prod_{j=1}^{N-n} (Z^*(\phi_j))^{\beta_j}(Z(\phi_j))^{\gamma_j}$
Then
\begin{equation}\label{AHCOMTHM2}
\prod_{j=1}^{N-n} (W_\cH Z^*(\phi_j))^{\beta_j}(W_\cH Z(\phi_j))^{\gamma_j} = \begin{cases} \phantom{W_{\cH} } 
Z_{\bb,\gg} & |\bb|+|\gg| \ {\rm is \ even}\\  W_{\cH} Z_{\bb,\gg} & |\bb|+|\gg| \ {\rm is \ odd}\end{cases}\ .
\end{equation}
It follows easily  from the linear independence of $\{Z_{\bb,\gg}\ :\ \bb,\gg\in \{0,1\}^N\}$ that 
$$
\left\{ \ \prod_{j=1}^{N-n} (W_\cH Z^*(\phi_j))^{\beta_j}(W_\cH Z(\phi_j))^{\gamma_j} \ :\ \bb,\gg\in \{0,1\}^{N-n}\ \right\}
$$
is linearly independent in $\scA_{\cH}'$ and since its cardinality is $2^{2(N-n)}$, it also spans $\scA_{\cH}'$.
\end{proof}

\begin{cl}\label{AHCOMCOR} Let $Z$ be a CAR field over $\cH_1$ acting irreducibly on $\scK$. Let $\cH$ be any non-trivial subspace of $\scH_1$, and let $W_\cH$ be defined as in 
Theorem~\ref{AHCOMTHM}. Then 
$A\in \scA_{\cH_1}\in \scA_{\cH}'$ if and only if 
\begin{equation}\label{AHCOMCOR1}
A = B + W_{\cH}C
\end{equation}
where $B$ is an even element on $\scA_{\cH^\perp}$ and $C$ is an odd element of $\scA_{\cH_1}$. Moreover, for every $A\in \scA_{\cH}'$, the decomposition in \eqref{AHCOMCOR1}
is unique. 
\end{cl}

\begin{proof} As a consequence of \eqref{AHCOMTHM2},  $\left\{ \ W_{\cH}^{|\bb|+|\gg|} Z_{\bb,\gg} \ :\ \bb,\gg\in \{0,1\}^{N-n}\ \right\}$ is a basis for $\scA_{\cH}'$, and 
every statement in the corollary is a direct consequence of this. 
\end{proof}

\section{Conditional expectations in the CAR}\label{CONDEXAPP}

 \subsection{Tracial conditional expectations}
 
Define the {\em Majorana fields} on $\cH_1$ as follows:
 \begin{equation}\label{QPDEFF}
Q(\psi) = Z(\psi)+ Z^*(\psi) \quad{\rm and}\quad P(\psi) = \frac{1}{i}(Z(\psi)- Z^*(\psi))
\end{equation}

 Simple computations using the CAR show that for all $\psi,\phi\in \cH_1$, 
 \begin{equation}\label{FIELDOPPS4} 
\{Q(\psi),P(\phi)\} = 0  \quad {\rm and}\quad   \{Q(\psi),Q(\phi)\}  =  \{P(\psi),P(\phi)\} = 2\Re \langle \psi,\phi\rangle \ .
\end{equation}

Now let $\cH$ be any $n$-dimensional  subspace of $\cH_1$, and suppose $1 \leq n \leq N-1$. Let $\{\psi_1,\dots,\psi_N\}$ be any orthonormal basis of $\cH_1$ such that 
$\{\psi_1,\dots,\psi_n\}$ is an orthonormal basis of $\cH$.  Define $\cN_\cH := \sum_{j=1}^n Z^*(\psi_j)Z(\psi_j)$, and 
$W_\cH := e^{i\pi \cN_\cH}$ as in Theorem~\ref{AHCOMTHM}.

By Theorem~\ref{AHCOMTHM}, for all $\varphi\in \cH^\perp$, $W_\cH Z(\varphi)$ and $W_\cH Z^*(\varphi)$ belong to $\scA_{\cH}'$, 
 for all $\varphi\in \cH^\perp$, $W_\cH Q(\varphi)$ and $W_\cH P(\varphi)$ belong to $\scA_{\cH}'$. Therefore,
 \begin{equation}\label{FIELDOPPS4QP} 
 \{ W_\cH Q(\psi_{n+1}),\dots, W_\cH Q(\psi_N),W_\cH P(\psi_{n+1}),\dots, W_\cH P(\psi_N)\}
 \end{equation}
 is a set of $2(N-n)$ anti-commuting self adjoint unitaries in $\scA_\cH'$.  For $\bb,\gg\in \{0,1\}^{N-n}$, define
 $$
 G_{\bb,\gg} := \prod_{j= 1}^{N-n} (W_\cH Q(\psi_{n+j}))^{\beta_j} (W_\cH P(\psi_{n+j}))^{\gamma_j} 
 $$
 
 The  $\cG_\cH := \{ \pm G_{\bb,\gg}\ :\ \bb,\gg\in \{0,1\}^N\}$ is a finite group of unitaries, and is the smallest subgroup of the unitary group on $\scH$ that contains the set
 \eqref{FIELDOPPS4QP}. The linear span of $\scG_\cH$ has dimension $2^{2(N-n)}$, which is the dimension of $\scA_{\cH}'$, and hence the complex linear span of $\cG_\cH$ is  $\scA_\cH'$.
 
 Evidently, for any $A\in \scA_{H_1}$, ${\displaystyle  \frac{1}{|\cG_\cH| }\sum_{G\in \cG_\cH} G A G^*}$ commutes with every $G\in \scG_\cH$, and hence with every $B\in \scA_{\cH}'$. 
 By von Neumann's Double Commutant Theorem, this average then belongs to $\scA_{\cH}$.

 \begin{defn}\label{TRACIALCEDEF}
 The {\em tracial conditional expectation} onto $\scA_\cH$, $ {\mathbb E}_{\cH,\tau}$,  is given by
 \begin{equation}\label{CODEXPDEF}
 {\mathbb E}_{\cH,\tau} (A)  = \frac{1}{|\cG_\cH| }\sum_{G\in \cG_\cH} G A G^*\ .
 \end{equation}
 \end{defn}
 
 \begin{thm}\label{CONDEXPROP} The super-operator $ {\mathbb E}_{\cH,\tau} $ has the following properties:
 
 \smallskip
 \noindent{\it (1)} For all $A\in \scA_{\cH_1}$, ${\mathbb E}_{\cH,\tau}(A) \in \scA_\cH$, and $\|{\mathbb E}_{\cH,\tau}(A)\| \leq \|A\|$. 
 
  \smallskip
 \noindent{\it (2)} For  all $A\in \scA_{\cH_1}$, 
 $$ 
 \tau({\mathbb E}_{\cH,\tau}(A)) =  \tau(A)
 $$
 
  \smallskip
 \noindent{\it (3)} For all $A,B\in \scA_\cH$,  and all $X\in \scA_{\cH_1}$, 
 $$ 
 {\mathbb E}_{\cH,\tau}(AXB) =  A {\mathbb E}_{\cH,\tau}(X)B
 $$

 \smallskip
 \noindent{\it (4)} For all $A\in \scA_\cH$ and  for all $B\in \scA_{\cH^\perp}$,  
 $$ {\mathbb E}_{\cH,\tau}(A) = A \quad{\rm and}\quad    {\mathbb E}_{\cH,\tau}(B) = \tau(B)\one\ .
 $$

  \smallskip
 \noindent{\it (5)} For  all $X,Y\in \scA_{\cH_1}$, 
 $$ 
 \tau(Y^*{\mathbb E}_{\cH,\tau}(X)) =  \tau({\mathbb E}_{\cH,\tau}(Y)^* X)\ .
 $$
 Finally, ${\mathbb E}_\cH$ is the unique operator  from $\scA_{\cH_1}$ into $\scA_\cH$  with properties {\it (2)} and {\it (3)}. In fact, for all $A\in \scA_{\cH_1}$,  the conditional expectation ${\mathbb E}_{\cH,\tau}(A)$ is characterized as the unique element of $\scA_{\cH}$ such that  
$$
\tau (B {\mathbb E}_{\cH,\tau}(A)) = \tau(BA)
$$
for all $B\in \scA_{\cH}$, which means that ${\mathbb E}_{\cH,\tau}$ is the orthongonal projection in $\scA_{\cH_1}$  onto $\scA_\cH$ with respect to the inner product $\langle \cdot,\cdot\rangle_{GNS,\tau}$.
 \end{thm} 
 
 \begin{proof}  By construction, for any $G\in \scG_\cH$ and $A\in \scA_{\cH_1}$, $[G,{\mathbb E}_{\cH,\tau}(A)]= 0$. Since $\scG_\cH$ spans $\scA_\cH'$, ${\mathbb E}_{\cH,\tau}(A)\in \scA_\cH'' = \scA_\cH$. Again by construction,
 ${\mathbb E}_{\cH,\tau}(A)$ is an average over unitary conjugations of $A$, and hence $\|{\mathbb E}_{\cH,\tau}(A)\| \leq \|A\|$ follows by Minkowski's inequality.  This proves {\it (1)}, and 
 {\it(2)} is an immediate consequence of cyclicity of the trace. 
 
 For {\it (3)}, let $A,B\in \scA_\cH$,  and  $X\in \scA_{\cH_1}$. Then for all $G\in \scG_\cH$, $G(AXB)G^* = A(GXG^*)B$ an therefore {\it (3)} follows from \eqref{CODEXPDEF}.

 For {\it (4)}, if $A\in \scA_\cH$, then $GAG^* = A$ for all $G\in \scG_\cH$, and hence  ${\mathbb E}_{\cH,\tau}(A) = A$. Next, let $B\in \scA_{\cH^\perp}$. If $B$ is even; i.e., $\Gamma(B) = B$, then $B\in \scA_\cH'$ and hence
 $GBG^*\in \scA_\cH'$ for all $G\in \scG_\cH$. Therefore, ${\mathbb E}_{\cH,\tau}(B) \in \scA_\cH'$. By part {\it (1)}, ${\mathbb E}_{\cH,\tau}(B)\in  \scA_\cH$, and $\scA_\cH\cap \scA_\cH'$ is spanned by $\one$.
 Thus for even $B\in   \scA_{\cH^\perp}$, ${\mathbb E}_{\cH,\tau}(B) = c\one$ for some constant $c$, and then by {\it (2)}, $c= \tau(B)$.  If $B\in \scA_{\cH^\perp}$ is odd, $W_\cH B\in \scA_\cH'$ and $\tau(B) =0$. Arguing as before,
 ${\mathbb E}_{\cH,\tau}(W_\cH B) \in \scA_\cH'\cap \scA_\cH$. Therefore, ${\mathbb E}_{\cH,\tau}(W_\cH B) = c\one$ where  $c= \tau(W_\cH B)$. Since $W_\cH$ is even, $W_\cH B$ is odd, and hence $\tau(W_\cH B)= 0$.
 Altogether, since $W_\cH^2 = \one$,
 $$
 {\mathbb E}_{\cH,\tau}(B) = W_\cH {\mathbb E}_{\cH,\tau}(W_\cH B) = 0 = \tau(B)\one\ .
 $$
 Since every element of $\scA_{\cH^\perp}$ is the sum of its even and odd parts, this proves {\it (4)}.
 
 By the cyclicity of the trace, $\tau(Y^*G X G^*)=  \tau(G^*Y^*G X )$ and $\sum_{G\in \scG_\cH} G^*Y^*G = \sum_{G\in \scG_\cH} GY^*G^*$ so that {\it (5)} follows from \eqref{CODEXPDEF}.
 
 Finally, let $\Phi:\scA_{\cH_1}\to \scA_\cH$ be such that for all $X\in \scA_{\cH_1}$, $\tau(\Phi(X)) =\tau(X)$ and such that for all $X\in \scA_{\cH_1}$ and all $X,Y\in  \scA_{\cH}$, $\Phi(AXB) = A\Phi(X)B$. 
Since for all $X\in \scA_{\cH_1}$,
\begin{equation}\label{CONDEXODEC}
X = \Phi(X) +  (X-\Phi(X))\ .
\end{equation}
By hypothesis, for all $A\in \scA_{\cH}$, $\tau(A^*(X -\Phi(X)) = \tau(A^*X) -\tau(\Phi(A^*X)) = 0\ .$
Therefore, \eqref{CONDEXODEC} is the decomposition of $A$ into its components in $\scA_\cH$ and $(\scA_\cH)^\perp$ with respect to the inner product $\langle \cdot,\cdot\rangle_{GNS,\tau}$. Hence any linear map $\Phi$ with these properties must be the orthongonal projection onto $\scA_\cH$ with respect to the inner product $\langle \cdot,\cdot\rangle_{GNS,\tau}$.  In particular, there is only one such map.
 \end{proof}
 
 \begin{cl}\label{SEGALPRODPROP} For all subspaces $\cH$ of $\cH_1$, and all $A\in \scA_\cH$ and all $B\in \scA_{\cH^\perp}$,
 \begin{equation}\label{SEGALPRODPROP1}
 \tau(AB) = \tau(A)\tau(B)\ .
 \end{equation}
 \end{cl}
 
 \begin{proof} By part {\it (5)} followed by part {\it (3)} and then part {\it (2)} of Theorem~\ref{CONDEXPROP},
 $$
 \tau(AB)  = \tau({\mathbb E}_{\cH,\tau} (AB)) = \tau(A{\mathbb E}_{\cH,\tau} (B)) = \tau(A \tau(B)\one) = \tau(A)\tau(B)\ .
 $$
 \end{proof}
 
 The following simple extension of Corollary~\ref{SEGALPRODPROP} will be useful here. 
 
 \begin{thm}\label{SEGCOND}
 Let $\cH$ be a subspace of $\cH_1$. Let $\cK$ be a  subspace of $\cH$, and let $\widetilde{\cK}$ be a  subspace of $\cH^\perp$. Let $A\in \scA_{\cH}$ and let 
 $B \in \scA_{\cH^\perp}$. Then
 $$
 {\mathbb E}_{\cK\oplus\widetilde{\cK}}(AB) =   {\mathbb E}_{\cK}(A) {\mathbb E}_{\widetilde{\cK}}(B)\ .
 $$
 \end{thm}
 
 \begin{proof} By the final part of Theorem~\ref{CONDEXPROP},  ${\mathbb E}_{\cK\oplus\widetilde{\cK}}(AB)$ is the unique element of $\scA_{\cK\oplus\widetilde{\cK}}$ such that
 $$
 \tau(Z {\mathbb E}_{\cK\oplus\widetilde{\cK}}(AB)) = \tau(ZAB)
 $$
 for all $Z\in \scA_{\cK\oplus\widetilde{\cK}}$.
 Since $\scA_{\cK\oplus\widetilde{\cK}}$ is spanned by elements of the form $YX$, $X\in \scA_{\cK}$ and $Y\in \scA_{\widetilde{K}}$, it suffices to show that for all such $X,Y$,
 $$\tau (YX AB) = \tau(YX {\mathbb E}_{\cK}(A) {\mathbb E}_{\widetilde{\cK}}(B))\ .$$
 By cyclicity of the trace and and Corollary~\ref{SEGALPRODPROP}, 
 \begin{eqnarray*}
  \tau(YX {\mathbb E}_{\cK}(A) {\mathbb E}_{\widetilde{\cK}}(B)) &=&  \tau(X {\mathbb E}_{\cK}(A) {\mathbb E}_{\widetilde{\cK}}(B)Y) = 
  \tau(X {\mathbb E}_{\cK}(A))\tau(Y {\mathbb E}_{\widetilde{\cK}}(B))\\
  &=& \tau(XA)\tau(YB) = \tau(YX AB)\ .
 \end{eqnarray*}
 \end{proof}

 \begin{lm}\label{CENEST}
Let $\cH$ and $\cH'$  be  non-trivial subspaces of $\cH_1$.  Then 
$$
{\mathbb E}_{\cH\cap \cH',\tau} = {\mathbb E}_{\cH\cap \cH',\tau}\circ {\mathbb E}_{\cH,\tau} = {\mathbb E}_{\cH\cap \cH',\tau}\circ {\mathbb E}_{\cH',\tau}\ .
$$
\end{lm}

\begin{proof} For all $A\in \scA_{\cH_1}$, ${\mathbb E}_{\cH\cap \cH',\tau}(A)$ is characterized by $\tau(B {\mathbb E}_{\cH\cap \cH' ,\tau}) = \tau(BA)$ for all $B\in \scA_{\cH\cap \cH'}$, while
${\mathbb E}_{\cH,\tau}(A)$ is characterized by $\tau(C {\mathbb E}_{\cH,\tau}(A)) = \tau(CA)$ for all $C\in \scA_{\cH}$. Therefore,  for all
$B\in \scA_{\cH\cap \cH'}$, 
$$\tau(B {\mathbb E}_{\cH\cap \cH',\tau}(A)) = \tau(BA) = \tau(B {\mathbb E}_{\cH,\tau}(A)) = \tau(B {\mathbb E}_{\cH\cap \cH',\tau}({\mathbb E}_{\cH,\tau}(A)))\ .$$
\end{proof}

\subsection{Conditional expectations in general}

Let $\cM$ be a $C^*$ algebra, and let $\cN$ be a $C^*$  sub-algebra of $\cM$. Let $\sE$ be any norm-contractive projection 
from $\cM$ onto $\cN$. By a theorem of Tomiyama \cite{To57}, $\sE$ preserves positivity, $\sE \one = \one$,
 and 
\begin{equation}\label{tomi}
 \sE(AXB) = A\sE(X) B \qquad{\rm for\ all} \quad A,B \in \cN, \ X\in \cM\ . 
 \end{equation}
Moreover, as Tomiyama noted, it follows from \eqref{tomi} and the positivity  preserving 
 property of $\sE$ that 
 \begin{equation}\label{tomi2}
 \sE(X)^* \sE(X) \leq \sE(X^*X)  \qquad{\rm for\ all} \ X\in \cM\ ,
 \end{equation}
 In fact, more is true: every norm contractive projection is {\em completely positive}. 
 
A {\em conditional expectation} from $\cM$ onto $\cN$, in the sense of Umegaki \cite{U54,U56,U59,U62}, is  a unital projection $\sE$ from $\cM$ onto $\cN$ that is order preserving and such that 
\eqref{tomi} and \eqref{tomi2} are satisfied.  Notice that this definition does not make any reference to any state. 

\begin{defn}\label{CECMPAT} A {\em conditional expectation} from $\cM$ onto $\cN$ is {\em compatible} with a state $\rho$ on $\cM$ in case 
\begin{equation}\label{STATECOMP}
\rho(\sE(A)) = \rho(A)
\end{equation}
for all $A\in \scA_{\cH_1}$.  
\end{defn}

We are interested in the nature of conditional expectations in $\scA_{\cH_1}$ for  subalgebras $\scA_\cH$ where $\cH$ is a subspace of $\cH_1$.

Given a state $\rho$ on $\scA_{\cH_1}$ and a non-trivial  subspace $\cH$  of $\scH_1$, define $\rho_{\scA_\cH}$ to be the restriction of $\rho$ to $\scA_\cH$.  This state has the density matrix
\begin{equation}\label{RESTRICT}
\rho_{\scA_\cH} = {\mathbb E}_{\cH,\tau}(\rho)
\end{equation}
where $\rho$ is the density matrix of the state $\rho$.  The question arises as to whether there is any conditional expectation $\sE$ that is compatible with $\rho$.  The next lemma says that if a compatible conditional expectation exists, it is unique. 

\begin{lm}\label{CEGNSOP} Let $\rho$ be a state on $\scA_{\cH_1}$, and let  $\sE$ be a linear map from $\scA_{\cH_1}$ to $\scA_\cH$ such that for all $A\in \scA_{\cH_1}$ and all $X\in \scA_\cH$, $\rho(\sE(A)) = \rho(A)$ and
$\sE(XA) = X\sE(A)$. Then $\sE$ is the orthogonal projection of $\scA_{\cH_1}$ onto $\scA_{\cH}$ with respect to the GNS inner product induced by $\rho$. 
\end{lm}

\begin{proof} For $A\in \scA_{\cH_1}$ and all $X\in \scA_\cH$ we compute, using the given properties of $\sE$, 
$$
\langle X,A\rangle_{GNS,\rho} = \rho(X^*A) = \rho(\sE(X^*A)) = \rho( X^*\sE(A)) = 
\langle X, \sE(A)\rangle_{GNS,\rho}\ .
$$
Therefore, $A- \sE(A)$ lies in the orthogonal complement of $\scA_\cH$ with respect to the GNS inner product induced by $\rho$, while $\sE(A)$ lies in $\scA_\cH$. 
\end{proof}

By a theorem of Takesaki, (proved in a  general von Neumann algebra setting)  a $\rho$-compatible conditional expectation exists 
if and only if the modular automorphism group for $\rho$ leaves $\scA_\cH$ invariant. Leaving aside domain considerations that are anyway irrelevant in finite dimensions, for faithful $\rho$, this is the same as
\begin{equation}\label{TEKESAKICON}
\Delta_\rho(X) \in \scA_\cH \quad{\rm for\ all}\quad X\in \scA_\cH\ 
\end{equation}
where $\Delta_\rho$ is the modular operator associated to $\rho$; i.e., in this finite dimensional setting, 
\begin{equation}\label{MODOPDEF}
\Delta_\rho(X) = \rho X \rho^{-1}\ 
\end{equation}
and where on the right side, $\rho$ is the density matrix of the state $\rho$.

\subsection{Conditional expectation with respect to  $\rho\in {\mathfrak S}_{GIG}$}

Let $\rho_Q\in {\mathfrak S}_{GIG}$ have the symbol $Q$. Let $\cH$ and $\cK$  be  non-trivial subspaces of $\cH_1$ that are orthogonal complements of one another, so that $\cH_1 = \cH\oplus \cK$, and such that both are invariant under $Q$.  

Let $n$ denote the dimension of $\cH$. 
when these conditions are satisfied, there is an orthonormal basis $\{\varphi_1,\dots,\varphi_N\}$ of $\cH_1$ 
consisting of eigenvectors of $Q$, $Q\varphi_j = \mu_j\varphi_j$, such that $\{\varphi_1,\dots,\varphi_n\}$ is an orthonormal basis for $\cH$ and $\{\varphi_{n+1},\dots,\varphi_N\}$ is an orthonormal basis  for $\cK$. 

Then   
\begin{equation}\label{COEXGIGS1}
\rho_Q = \prod_{j=1}^N ((1-\mu_j)2Z(\varphi_j)Z^*(\varphi_j) + \mu_j 2Z^*(\varphi_j) Z(\varphi_j))
\end{equation}
and with $Q_\cH$ and $Q_\cK$ denoting the restrictions of $Q$ to $\cH$ and $\cK$ respectively, $\rho_Q$
factors as the product of two commuting gauge invariant Gaussian density matrices
\begin{equation}\label{COEXGIGS11}
\rho_Q = \rho_{Q_\cH} \rho_{Q_{\cK}}
\end{equation}
where $\rho_{Q_\cH }$ is the product of the first $n$ factors in \eqref{COEXGIGS1} and  $\rho_{Q_\cK}$ is the product of the final $N-n$ factors in \eqref{COEXGIGS1}.
Not only do $ \rho_{Q_\cH}$ and $\rho_{Q_{\cK}}$ commute,  $ \rho_{Q_\cH}\in \scA_\cK'$ and $\rho_{Q_{\cK}}\in \scA_\cH'$. 

To apply Takesaki's Theorem, we need an expression for the modular operator $\Delta_\rho$  of a GIG state $\rho$; see \eqref{MODOPDEF}. Consider first any density matrix $\rho$ on $\scK$. 
Let $\Psi$ and $\Phi$ be eigenvectors of $\rho$ with eigenvalues $\kappa$ and $\lambda$ respectively. Then 
$$
\Delta_\rho(|\Psi\rangle\langle \Phi|) = \kappa\lambda^{-1} |\Psi\rangle\langle \Phi|\ .
$$
Thus, $|\Psi\rangle\langle \Phi|$ is an eigenvector of $\Delta_\rho$ with eigenvalue $\lambda\mu$. It follows that $\Delta_\rho$ is positive semidefinite on $\scA_{\cH_1}$, and that diagonalizing $\rho$
provides a diagonalization of  $\Delta_\rho$. 

Let $Q$ be the symbol of a gauge invariant Gaussian state $\rho$, and suppose that $0 < Q < \one$. Let $\{\psi_1,\dots,\psi_N\}$ be an orthonormal basis of $\cH_1$ consisting of eigenvectors of $Q$; $Q\psi_j = \mu_j \psi_j$. 
Define $Z_j := Z(\psi_j)$. Then 
\begin{equation}\label{EXGPTV2}
\rho = 
\prod_{j=1}^N \left((1-\mu_j)2Z_j^{\phantom{*}}Z_j^* + \mu_j 2Z_j^*Z_j^{\phantom{*}}\right)\ ,
\end{equation}
and hence 
\begin{equation}\label{EXGPTV3}
\rho^{-1} = 
\prod_{j=1}^N \left(\frac12(1-\mu_j)^{-1}Z_j^{\phantom{*}}Z_j^* + \frac12 \mu_j^{-1} Z_j^*Z_j^{\phantom{*}}\right)\ ,
\end{equation}

By the CAR, 
\begin{equation}\label{EXGPTV4}
\Delta_\rho(Z_k) = \frac{1-\mu_k}{\mu_k} Z_k\quad{\rm and}\quad \Delta_\rho(Z^*_k) = \frac{\mu_k}{1-\mu_k} Z^*_k\ .
\end{equation}
Evidently,
\begin{equation}\label{EXGPTV5}
\Delta_\rho(Z_j^{\phantom{*}}Z_j^*) = Z_j^{\phantom{*}}Z_j^* \quad{\rm and}\quad \Delta_\rho(Z_j^*Z_j^{\phantom{*}}) = Z_j^*Z_j^{\phantom{*}}\ .
\end{equation}

Return now to the case in which $\cH_1 = \cH\oplus \cK$ where $\cH$ is invariant under $Q$ so that $\rho_Q$ has the factorization \eqref{COEXGIGS11}.
In this case, the modular operator $\Delta_{\rho_Q}$ has a very simple action on $\scA_\cH$: For all $X\in \scA_\cH$, 
$$
\Delta_{\rho_Q}(X) = \rho_Q X \rho_Q^{-1}  = \rho_{Q_{\cH}}\rho_{Q_\cK} X \rho_{Q_\cH}^{-1}\rho_{Q_\cK}^{-1} = 
\rho_{Q_\cH} X \rho_{Q_\cH}^{-1}\in \scA_{\cH}\ .
$$
Therefore, by Takesaki's Theorem, there exists a conditional expectation from $\scA_{\cH_1}$ onto $\scA_\cH$ that is compatible with $\rho_Q$. 

\begin{defn} Let $\cH$ be a non-trivial subspace of $\cH_1$ that is invariant under $Q$. Then the unique conditional expectation from $\scA_{\cH_1}$ onto $\scA_\cH$ that is compatible with $\rho_Q$  is denoted by 
${\mathbb E}_{\cH,\rho_Q}$. 
\end{defn} 

It is easy to give an explicit formula for ${\mathbb E}_{\cH,\rho_Q}$.

\begin{lm}\label{CONDEXRHOQ} Let $\cH$ be a non-trivial subspace of $\cH_1$ that is invariant under $Q$, and let $\cK= \cH^\perp$. Let $\rho_Q = \rho_\cH\rho_\cK$ be the factorization of $\rho_Q$ given in \eqref{COEXGIGS11}.  Then for all $A\in \scA_{\cH_1}$, 
\begin{equation}\label{CONDEXRHOQ1}
{\mathbb E}_{\cH,\rho_Q} = \mathbb{E}_{\cH,\tau}( \rho_{Q_{\cK}} A) = \mathbb{E}_{\cH,\tau}( \rho_{Q_{\cK}}^{1/2} A\rho_{Q_{\cK}}^{1/2})\ .
\end{equation}
\end{lm}

\begin{proof} Evidently, for all $A\in \scA_{\cH_1}$, $\mathbb{E}_{\cH,\tau}( \rho_{Q_{\cK}} A)\in \scA_\cH$, and if $A\in \scA_\cH$, 
$$
\mathbb{E}_{\cH,\tau}( \rho_{Q_{\cK}} A) = \mathbb{E}_{\cH,\tau}( \rho_{Q_{\cK}})\mathbb{E}_{\cH,\tau}(  A) =A\ .
$$
Therefore, $A \mapsto \mathbb{E}_{\cH,\tau}( \rho_{Q_{\cK}} A)$ is a projection onto $\scA_\cH$. 

The second equality in \eqref{CONDEXRHOQ1} is a consequence of the partial cyclicity of the partial trace, and since 
 ${\mathbb E}_{\cH,\tau}$ is completely positive, this shows that $A \mapsto \mathbb{E}_{\cH,\tau}( \rho_{Q_{\cK}} A)$ is completely positive. 

Since $\rho_{Q_\cK}\in \scA_{\cH}'$, for any $X,Y\in \scA_{\cH}$ and any  $A\in \scA_{\cH_1}$,
\begin{equation}\label{COEXGIGS3}
\mathbb{E}_{\cH,\tau}( \rho_{Q_{\cK}} XAY) = X\mathbb{E}_{\cH,\tau}( \rho_{Q_{\cK}} A)Y\ .
\end{equation}
This completes the proof that $A \mapsto \mathbb{E}_{\cH,\tau}( \rho_{Q_{\cK}} A)$ is a conditional expectation. 

To see that is is compatible with $\rho_Q$, note that for any $A\in \scA_{\cH_1}$, 
$$\rho_Q(A) = \tau(\rho_{Q_\cH}\rho_{\cK}A) = \tau(\mathbb{E}_{\cH,\tau}(\rho_{Q_\cH}\rho_{Q_\cK}A)) = \tau((\rho_{Q_\cH}\mathbb{E}_{\cH,\tau}(\rho_{Q_\cK}A))= \rho_Q(\mathbb{E}_{\cH,\tau}( \rho_{Q_{\cK}} A))\ .$$
\end{proof}

Since $\cH_1 = \cH\oplus \cK$, we may regard $\scA_{\cH_1}$ as the graded tensor product of $\scA_{\cH}$ and $\scA_{\cK}$, and then 
thinking of $\rho_{Q\_\cK}$ as a state instead of as a density matrix, we can also write the operator $A \mapsto \mathbb{E}_{\cH,\tau}( \rho_{Q_{\cK}} A)$ as 
$A \mapsto \one\otimes \rho_{Q\_\cK}$.

The construction of ${\mathbb E}_{\cH,\rho_Q}$  depended on the compatibility of $\cH$ and $Q$.  If the subspace $\cH$ of $\cH_1$ were not invariant under $Q$, such a conditional expectation would not exist because then $\Delta_{\rho_Q}$ would not leave $\scA_{\cH}$ invariant.  Of course, if $Q$ is a multiple of the identity, then every subspace $\cH$ of $\cH_1$ is invariant under $Q$, and we always have a compatible conditional expectation. The symbol $Q$ of the tracial state $\tau$ is an example; in this case $Q = \frac12 \one$. The vacuum state $\omega_0$ is another example; in this case $Q = 0$.  We now derive an alternate formula for 
${\mathbb E}_{\cH,\omega_0}$.

Let $\cH$ be a non-trivial subspace of $\cH_1$, and let 
$E_\cH$ denote the  orthogonal projection onto  $\cH$  of  $\cH_1$. Let
$E^\perp_{\cH}$ denote the complementary projection, which is $E_{\cH^\perp}$. Then
$$E_\cH = \lim_{t\to\infty} e^{-t E_{\cH^\perp}}\ .$$
Define
\begin{equation}\label{VACCONDEX1}
\cE_\cH := \lim_{t\to\infty} e^{-t \widehat{E_{\cH^\perp}}}\ .
\end{equation}
Since $E_{\cH^\perp}$ is self adjoint, so is $\widehat{E_{\cH^\perp}}$. Therefore, $\cE_\cH$ is self adjoint. Moreover, for all $t$, $\left(e^{-t \widehat{E_{\cH^\perp}}}\right)^2 = 
e^{-2t \widehat{E_{\cH^\perp}}}$ and hence $\left(\cE_\cH\right)^2 = \cE_\cH$. This proves that $\cE_\cH$ is an orthogonal projection on $\scH$. 

\begin{thm}\label{VACCONDEXTHM} Let $\cH$ be a non-trivial subspace of $\cH_1$, and let $\cE_\cH$ be the orthogonal projection on $\scH$ defined by \eqref{VACCONDEX1}. Then ${\mathbb E}_{\cH,\omega_0}$ is given by
\begin{equation}\label{VACCONDEX0}
{\mathbb E}_{\cH,\omega_0}(A) = \cE_\cH A \cE_\cH\ .
\end{equation}
\end{thm}

\begin{proof} Define $\Psi_\cH:\scA_{\cH_1}\to \scA_{\cH_1}$ by
\begin{equation}\label{VACCONDEX2}
\Psi_\cH(A) = \cE_\cH A \cE_\cH\ .
\end{equation}
Then $\Psi$ is a manifestly completely positive, and to prove that is a conditional expectation, we need only show that for all $X,Y\in \scA_\cH$ and all $A\in \scA_{\cH_1}$,
\begin{equation}\label{VACCONDEX20}
\Psi(XAY) = X\Psi(A)Y\  .
\end{equation}

Let $n := {\rm dim}(\cH)$, and then since $\cH$ is a non-trivial subspace,  $1 \leq n \leq N-1$. Let $\{\varphi_1,\dots,\varphi_N\}$ be an orthonormal basis of $\cH_1$ such that the first $n$ vectors span $\cH$.  Let $Z:\cH_1\to \scA_{\cH_1}$ be a CAR field over $\cH_1$, so that with $Z_j := Z(\varphi_j)$, $j=1,\dots,N$,   $\{Z_1,\dots,Z_N\}$ satisfy the CAR.
For $K\subseteq\{1,\dots,N\}$, define $\aa_K\in \{0,1\}^N$ by $(\aa_K)_j = 1$ if $j\in K$, and $(\aa_K)_j = 0$ if $j\notin K$. Let $\varPhi_K$ denote $\varPhi_{\aa_K}$ where $\varPhi_\aa$ is given by \eqref{PHIKDEF}.   Then  $e^{-t \widehat{E_{\cH^\perp}}}\varPhi_K = e^{-t|K \cap\{1,\dots,n\}|}\varPhi_K$. Therefore, 
$$
\cE_\cH = \sum_{K\subset \{1,\dots,n\}}|\varPhi_K\rangle\langle \varPhi_K|\ .
$$
For $L\subset \{1,\dots,N\}$, define
$$
Z_L := (Z_1^*)^{\boldsymbol{\alpha}^L_1}\cdots  (Z_N^*)^{\boldsymbol{\alpha}^L_N}\ .
$$
Then 
${\displaystyle
Z_L\varPhi_K = \begin{cases} 0 & L\cap K^c \neq \emptyset\\
\pm \varPhi_{K\cap L^c} & L\subset K \end{cases}}$ and hence
$$
\cE_\cH Z^*_{L'}Z_L^{\phantom{*}} \cE_\cH = \begin{cases}  Z^*_{L'}Z_L^{\phantom{*}} & L,L'\subset\{1,\dots,n\}\\ 0 & {\rm otherwise}\end{cases}\ .
$$
Therefore, $\Psi_\cH(A) = \cE_\cH A \cE_\cH\in \scA_\cH$ for all $A\in \scA_{\cH_1}$, and in case $A\in \scA_\cH$,
$\cE_\cH A \cE_\cH = A$. Therefore, $\Psi_\cH$ is a projection in $\scA_{\cH_1}$ onto $\scA_\cH$. Moreover since $\cE_\cH$ itself  is an orthogonal projection on $\scH$, for all $X\in \scA_\cH$, 
$$
\cE_\cH X = \cE_\cH(\cE_\cH X\cE_\cH) = \cE_\cH X\cE_\cH = (\cE_\cH X\cE_\cH)\cE_\cH = X\cE_\cH\ .
$$
In other words, $\cE_\cH\in \scA_\cH'$.  Then for any $A\in \scA_{\cH_1}$ and any $X,Y \in \cB(\scH)$, 
$$\cE_\cH XAY\cE_\cH = X\cE_\cH A \cE_\cH Y\ ,$$
so that \eqref{VACCONDEX20} is satisfied, and $\Psi$ is a conditional expectation.  

To see that $\Psi$ is compatible with $\omega_0$, note that 
 every $A\in \scA_{\cH_1}$ can be written as
$$A = \sum_{L,L'\subseteq\{1,\dots,N\}}a_{L,L'}Z^*_{L'}Z_L^{\phantom{*}}\ ,$$
and $\omega_0(A) = a_\emptyset$. By 
$$
\Psi_\cH(A) = \sum_{L,L'\subseteq\{1,\dots,n\}}a_{L,L'}Z^*_{L'}Z_L^{\phantom{*}}\ ,
$$ 
and hence $\omega_0(\Psi_\cH(A)) = a_0 = \omega_0(A)$ so that $\Psi$ is a conditional expectation that is compatible with $\omega_0$. By the uniqueness of compatible conditional expectations,
$\Psi = {\mathbb E}_{\cH,\omega_0}$. 
\end{proof}

\begin{cl} Let $\cH$ and $\cK$ be subspaces of $\cH_1$ such that $[E_\cH,E_\cK] = 0$ where $E_\cH$ and $E_\cK$ are the orthogonal projections in $\cH_1$ onto $\cH$ and $\cK$ respectively. 
Then $[{\mathbb E}_{\cH,\omega_0}, {\mathbb E}_{\cK,\omega_0}] =0$, and hence
\begin{equation}\label{vaccexINT}
{\mathbb E}_{\cH,\omega_0} {\mathbb E}_{\cK,\omega_0} = {\mathbb E}_{\cK,\omega_0}{\mathbb E}_{\cH,\omega_0}  = {\mathbb E}_{\cH\cap \cK,\omega_0} \ .
\end{equation}
\end{cl}

\begin{proof} Since $E_\cH$ and $E_\cK$ commute on $\cH_1$, so do $E_{\cH^\perp}$ and $E_{\cK^\perp}$. Then by \eqref{AWCOMM}
$\widehat{E_{\cH^\perp}}$ and $\widehat{E_{\cK^\perp}}$ commute on $\scK$. Then $[{\mathbb E}_{\cH,\omega_0}, {\mathbb E}_{\cK,\omega_0}] =0$  follows from \eqref{VACCONDEX1} and Theorem~\ref{VACCONDEXTHM}, and the 
identity \eqref{vaccexINT} follows directly from this. 
\end{proof} 

The vacuum state $\omega_0$  and the tracial state $\tau$ are very special gauge invariant Gaussian states, and to deal with conditional expectation with respect to general gauge invariant Gaussian states $\rho$, an orthonormal basis for $\scA_{\cH_1}$ with respect to the $\langle\cdot,\cdot\rangle_{GNS,\rho}$ inner product is useful on account of the connection between conditional expectations and GNS orthogonal projections established in Lemma~\ref{CEGNSOP}.

\begin{lm}\label{FERMHERM} Let $\{\varphi_1,\dots,\varphi_N\}$ be an orthonormal basis of $\scA_{\cH_1}$ consisting of eigenvectors of $Q$ so that 
 $Q = \sum_{j=1}^N \mu_j |\varphi_j\rangle\langle \varphi_j|$.  Suppose that  $0< \mu_j< 1$ for each $j$. Then with $Z_j = Z(\varphi_j)$,
 define  $\{K_{j,0,0}, K_{j,1,0},K_{j,0,1},K_{j,1,1}\}\subset \scA_{{\rm span}(\varphi_j)}$ by
 \begin{equation}\label{FERMHERM1K}
 K_{j,0,0} := \one\ ,\quad K_{j,1,0} =\frac{1}{\sqrt{1-\mu_j}} Z_j^* \quad{\rm and}\quad  K_{j,0,1} =\frac{1}{\sqrt{\mu_j}} Z_j
 \end{equation}
 and 
 \begin{equation}\label{FERMHERM2K}
 K_{j,1,1} := \frac{1}{\sqrt{\mu_j(1-\mu_j)}}(Z_j^*Z_j^{\phantom{*}} -\mu_j\one)\ .
  \end{equation}
  Let $\aa = (\alpha_1,\dots, \alpha_N)$ denote a generic element of the index set $\{ \{0,1\}\times \{0,1\}\}^N$, 
and define the operator
\begin{equation}\label{krchuk4B}
K_\aa :=  K_{1,\alpha_1}K_{2,\alpha_2} \cdots K_{N,\alpha_N}
\end{equation}
Then 
\begin{equation}\label{krchuk4C}
\{ K_\aa \ :\ \aa\in  \{ \{0,1\}\times \{0,1\}\}^N\ \}
\end{equation}
is an orthonormal basis for $\scA_{\cH_1}$ with respect to the GNS inner product induced by $\rho_Q$.  Moreover, each $K_\aa$ is an eigenvector of $\Delta_{\rho_Q}$:
\begin{equation}\label{krchuk4D}
\Delta_{\rho_Q}(K_\aa) = \left(\prod_{j\ :\ \alpha_j=(1,0)}\frac{\mu_j}{1-\mu_j}\right) \left(\prod_{j\ :\ \alpha_j=(0,1)}\frac{1-\mu_j}{\mu_j}\right) K_\aa\ .
\end{equation}
Finally,  let $\cH$ be a subspace of $\cH_1$ that is invariant under $Q$. We choose the orthonormal basis $\{\varphi_1,\dots,\varphi_N\}$ so that for some $J_\cH \subseteq \{1,\dots,N\}$
$\{ \varphi_j\ :\  j\in J_\cH\}$ is an orthonormal basis of $\cH$. Then
defining $|K_\aa\rangle\langle K_\aa|$ using the GNS inner product induced by $\rho_Q$,
\begin{equation}\label{krchuk4E}
{\mathbb E}_{\cH,\rho_Q} = \sum_{\{\aa \ :\ \alpha_j = (0,0) \ {\rm for}\  j \notin J_\cH\}} |K_\aa\rangle\langle K_\aa|\ .
\end{equation}
\end{lm}

 \begin{proof} The density matrix for $\rho_Q$ is given by
$$
\rho_Q = \prod_{j=1}^N ((1-\mu_j)2Z_j^{\phantom{*}}Z_j^* + \mu_j 2Z_j^*Z_j^{\phantom{*}})\ .
$$
Since even and odd elements of of $\scA_{\cH_1}$ are orthogonal in the GNS inner product for any even state $\rho$, and since for any state $\rho$, 
$$\langle Z_j^{*},Z_j^{\phantom{*}}\rangle_{\rho,GNS} = \tau(Z_j^2\rho) = 0\ ,$$
$\{\one,Z_j^{\phantom{*}},Z_j^*\}$ is  mutually orthogonal for the $\langle \cdot,\cdot\rangle_{GNS,\rho_Q}$ inner product. 
One computes
$$
\langle Z_j^*,Z_j^*\rangle_{GNS,\rho_Q} = \tau(Z_j^{\phantom{*}}Z_j^*\rho_Q) = (1-\mu_j)\quad{\rm and} \quad \langle Z_j,Z_j\rangle_{GNS,\rho_Q} = \tau(Z_j^*Z_j^{\phantom{*}}\rho_Q) = \mu_j\ .
$$
Therefore, $\{K_{j,0,0}, K_{j,1,0},K_{j,0,1}\}$ is an orthonormal with respect to the GNS inner product induced by $\rho_Q$.  

Evidently, $\{K_{j,0,0}, K_{j,1,0},K_{j,0,1}, Z^*_jZ_j^{\phantom{*}}\}$ spans $\scA_{{\rm span}(\varphi_j)}$, and we obtain an orthonormal basis by applying the Gram-Schmidt procedure, which is particularly simple since 
$Z^*_jZ_j^{\phantom{*}}$ being even is orthogonal to $K_{j,1,0}$ and $K_{j,0,1}$.
Evidently $Z_j^*Z_j^{\phantom{*}} - \rho_Q(Z_j^*Z_j^{\phantom{*}}) = Z_j^*Z_j^{\phantom{*}} -\mu_j\one$ is orthogonal to $\one$ in the GNS inner product induced by $\rho_Q$, and 
 $\rho_Q((Z_j^*Z_j^{\phantom{*}} -\mu_j\one)^2) = \mu_j(1-\mu_j)$. 
 This proves that  $\{K_{j,0,0}, K_{j,1,0},K_{j,0,1},K_{j,1,1}\}$ is an orthonormal basis for $\scA_{{\rm span}(\varphi_j)}$ with respect to the inner product induced by $\langle \cdot,\cdot\rangle_{GNS,\rho_Q}$.
 
 Next, one readily checks that
 $$
 \langle K_\bb,K_\aa\rangle_{GNS,\rho_Q} = \prod_{j=1}^N  \langle K_{\beta_j},K_{\alpha_j}\rangle_{GNS,\rho_Q} = \delta_{\bb,\aa}\ .
 $$
Then \eqref{krchuk4D} follows directly from \eqref{EXGPTV4} and \eqref{EXGPTV5}.

Finally, its is evident that $\{K_\aa \ :\  \alpha_j = 0 \quad{\rm for\ all}\  j > n\}$ spans $\scA_\cH$, and hence is an orthonormal basis for it with respect to the 
GNS inner product induced by $\rho_Q$. Then 
\eqref{krchuk4E} is true because ${\mathbb E}_{\cH,\rho_Q}$ is the orthogonal projection onto $\scA_\cH$ in the GNS inner product induced by $\rho_Q$. 
\end{proof}

\begin{cl}\label{COMCONDEXG} Let $\cH$ and $\cK$ be subspaces of $\cH_1$ that are both invariant under $Q$, and are  such that 
$[E_\cH,E_\cK] = 0$ where $E_\cH$ and $E_\cK$ are the orthogonal projections in $\cH_1$ onto $\cH$ and $\cK$ respectively. 
Then $[{\mathbb E}_{\cH,\omega_0}, {\mathbb E}_{\cK,\omega_0}] =0$, and hence
\begin{equation}\label{COMCONDEXG1}
{\mathbb E}_{\cH,\rho_Q} {\mathbb E}_{\cK,\rho_Q} = {\mathbb E}_{\cK,\rho_Q}{\mathbb E}_{\cH,\rho_Q}  = {\mathbb E}_{\cH\cap \cK,\rho_Q} \ .
\end{equation}
\end{cl}

\begin{proof} We choose the orthonormal basis $\{\varphi_1,\dots,\varphi_N\}$ figuring in Thoerem~\ref{FERMHERM} so  that $\{\varphi_1,\dots,\varphi_n\}$ is an orthonormal basis for $\cH$, and 
$\{\varphi_{m+1},\dots,\varphi_N\}$ is  an orthonormal basis for $\cK$ which is possible since $E_\cH$ and $E_\cK$ commute. (In this case $\{\varphi_n,\dots,\varphi_n\}$ is an 
orthonormal basis for $\cH\cap \cK$ which, to avoid trivialities, we assume to be non-zero.) Then by \eqref{krchuk4E}, 
\begin{equation}\label{COMCONDEXG2}
{\mathbb E}_{\cH,\rho_Q} =  \sum_{\{\aa \ :\ \alpha_j = (0,0) \ {\rm for}\ j \notin J_\cH\}} |K_\aa\rangle\langle K_\aa|\ ,
\end{equation}
\begin{equation}\label{COMCONDEXG3}
{\mathbb E}_{\cK,\rho_Q} =  \sum_{\{\aa \ :\ \alpha_j = (0,0) \ {\rm for}\ j \notin J_\cK\}} |K_\aa\rangle\langle K_\aa|
\end{equation}
and
\begin{equation}\label{COMCONDEXG4}
{\mathbb E}_{\cH\cap \cK,\rho_Q} =  \sum_{\{\aa \ :\ \alpha_j = (0,0) \ {\rm for}\ j \notin J_{\cH\cap \cK}\}} |K_\aa\rangle\langle K_\aa|\ .
\end{equation}
The fact that $[{\mathbb E}_{\cH,\omega_0}, {\mathbb E}_{\cK,\omega_0}] =0$ is evident from \eqref{COMCONDEXG2} and \eqref{COMCONDEXG3}, and \eqref{COMCONDEXG1}
is evident from \eqref{COMCONDEXG2}, \eqref{COMCONDEXG3} and \eqref{COMCONDEXG4}.
\end{proof}

It is worth remarking that the orthonormal basis $\{ K_\aa \ :\ \aa\in  ( \{0,1\}\times \{0,1\})^N\ \}$ is also orthogonal for the KMS inner product associated to $\rho_Q$:
$$
\langle K_\bb,K_\aa\rangle_{KMS,\rho_Q} = \tau(\rho_Q^{1/2}K_\bb^* \rho_Q^{1/2}K_\aa) = \tau(\rho_QK_\bb^* \rho_Q^{1/2}K_\aa \rho_Q^{-1/2}) = \langle K_\bb,\Delta^{1/2}_{\rho_Q}(K_\aa)\rangle_{GNS,\rho_Q}\ .
$$
Therefore, by \eqref{krchuk4D}, $\{ K_\aa \ :\ \aa\in  \{ \{0,1\}\times \{0,1\}\}^N\ \}$ is a mutually orthogonal, but not necessarily normalized, basis for $\scA_{\cH_1}$ with respect to the KMS inner product induced by $\rho_Q$. 
Again by \eqref{krchuk4D},
$$
\|K_\aa\|^2_{KMS,\rho_Q} = \left(\prod_{j\ :\ \alpha_j=(1,0)}\frac{\mu_j}{1-\mu_j}\right)^{1/2} \left(\prod_{j\ :\ \alpha_j=(0,1)}\frac{1-\mu_j}{\mu_j}\right)^{1/2}\|K_\aa\|^2_{GNS,\rho_Q}\ .
$$

\section{Proof of Wolfe's Lemma}\label{Wolfe}

\begin{proof}[Proof of Lemma~\ref{WOLFELEM}] 
Let $Z$ be a CAR field over $\cH_1$ acting irreducibly on $\scK$. Suppose that $\rho_\lambda = (1-\lambda)\rho_0 + \lambda \rho_1$ is a gauge invariant Gaussian state. To simplify the notation, we write 
$Q_0$ for $Q_{\rho_0}$ and  $Q_1$ for $Q_{\rho_1}$. By \eqref{SYMDEF}, the symbol of $\rho_\lambda$ is $Q_\lambda = (1-\lambda)Q_0 + \lambda Q_1$.  

Since  $\rho_\lambda = (1-\lambda)\rho_0 + \lambda \rho_1$,
for any vectors $\psi_1,\psi_2,\varphi_1,\varphi_2 \in \C^N$,  
\begin{multline*}
\rho_\lambda(Z^*(\psi_1)Z^*(\psi_2)Z(\varphi_2)Z(\varphi_1))  =\\ (1-\lambda)\rho_0(Z^*(\psi_1)Z^*(\psi_2)Z(\varphi_2)Z(\varphi_1))  + \lambda 
\rho_1(Z^*(\psi_1)Z^*(\psi_2)Z(\varphi_2)Z(\varphi_1))\ .
\end{multline*}

Computing the fourth moments of $\rho_\lambda$ using Wick's Theorem, one finds that 
$$
\langle \psi_1,Q_\lambda,\varphi_1\rangle  \langle \psi_2,Q_\lambda,\varphi_2\rangle - 
\langle \psi_1,Q_\lambda,\varphi_2\rangle  \langle \psi_2,Q_\lambda,\varphi_1\rangle
$$
must equal
\begin{multline*}
 (1-\lambda)(\langle \psi_1,Q_0,\varphi_1\rangle  \langle \psi_2,Q_0,\varphi_2\rangle - 
\langle \psi_1,Q_0,\varphi_2\rangle  \langle \psi_2,Q_0,\varphi_1\rangle)\\
+ \lambda(\langle \psi_1,Q_1,\varphi_1\rangle  \langle \psi_2,Q_1,\varphi_2\rangle - 
\langle \psi_1,Q_1,\varphi_2\rangle  \langle \psi_2,Q_1,\varphi_1\rangle)\ .
\end{multline*}
Define $D = Q_1 - Q_0$. Subtracting the second expression above from the first yields, after some simple computations,
$$
0 = \lambda(1-\lambda) (\langle\psi_2,D\varphi_1\rangle \langle\psi_1,D\varphi_2\rangle - \langle\psi_1,D\varphi_1\rangle \langle\psi_2,D\varphi_2\rangle)
$$
and therefore, 
$$
\langle \psi_1, (\langle\psi_2,D\varphi_2\rangle D - |D\varphi_2\rangle\langle D\psi_2|)\varphi_1\rangle = 0\ .
$$
Since $\psi_1$ and $\varphi_1$ are arbitrary vectors in $\C^N$, 
$\langle\psi_2,D\varphi_2\rangle D = |D\varphi_2\rangle\langle D\psi_2|$
for all choices of $\psi_2$ and $\varphi_2$. Since $D$ is self-adjoint, it must have the form $D = \pm |\eta\rangle\langle \eta|$ for some $\eta\in \C^N$. 

For the converse, suppose that $Q_1 -Q_0 = \pm |\eta\rangle\langle \eta|$. Interchanging $Q_0$ and $Q_1$ if necessary, we may suppose $Q_1 = Q_0+  |\eta\rangle\langle \eta|$.
Let $\{\be_1,\dots,\be_N\}$ be the standard basis of $\C^N$. Let $\Lambda = \{\alpha_1,\dots,\alpha_n\}$ be any subset of $\{1,\dots,N\}$ of cardinality $n$, $1 \leq n \leq N$. Then
$$
\langle \be_{\alpha_j},Q_\lambda \be_{\alpha_k}\rangle 
\langle \be_{\alpha_j},(Q_0 + \lambda |\eta\rangle\langle \eta|)\be_{\alpha_k}\rangle   = 
\langle \be_{\alpha_j},Q_0\be_{\alpha_k}\rangle + \lambda |\eta_\Lambda\rangle\langle \eta_\Lambda |
$$
where ${\displaystyle \eta_\Lambda = \sum_{j=1}^n\langle \be_{\alpha_j},\eta\rangle \be_{\alpha_j}}$. Suppose that $\eta_\Lambda \neq 0$. Computing $\det(\langle \be_{\alpha_j},Q_0\be_{\alpha_k}\rangle + \lambda |\eta_\Lambda\rangle\langle \eta_\Lambda |)$ in an orthonormal basis in which $\|\eta_\Lambda\|^{-1}\eta_\Lambda$ is the first element, one sees that this determinant is as affine function of $\lambda$, and this is trivially true if $\eta_\Lambda = 0$. Therefore 
$$
\det(\langle \be_{\alpha_j},Q_\lambda \be_{\alpha_k}\rangle) = 
(1-\lambda)\det(\langle \be_{\alpha_j},(Q_0 )\be_{\alpha_k}\rangle)  + \lambda \det(\langle \be_{\alpha_j},(Q_1)\be_{\alpha_k}\rangle)\ . 
$$
It follows that the moments of $(1-\lambda)\rho_0 + \lambda \rho_1$ are those of the unique gauge invariant Gaussian state with symbol $Q_\lambda$, and therefore 
$(1-\lambda)\rho_0 + \lambda \rho_1$ is a gauge invariant Gaussian state. 
\end{proof}

\section{Proof of the Araki-Wyss Theorem}\label{AWA}

 In this appendix we prove Theorem~\ref{AWTHM}. What is actually proved in \cite{AW64} is that $\scG_{\cH_1}$ is complex linear span of all $e^{\widehat{X}}$ where $X$ is any operator on 
 $\cH_1$. We need the stronger version in which $X$ is required to be self-adjoint. The proof given here is quite different from the proof of the related theorem in \cite{AW64} which uses the 
 Baker-Campbel-Hausdorff  formula.  We begin with two simple lemmas.

 \begin{lm}\label{ARWYLEM1} Let $\scS$ denote  the complex linear span of ${\mathfrak S}_{GIG}$. Let $H$ be any self-adjoint operator on $\cH_1$.  Then for all $A\in \scS$,
 $[\widehat{H},A]\in \scS$. 
 \end{lm}
 
 \begin{proof} Let $A\in \scS$ and let $H$ be self-adjoint on $\cH_1$. Then by part {\it (2)}  Lemma~\ref{STATESYMCLM}, for all $t\in \R$,
 $$
 e^{-it \widehat{H}} A  e^{it \widehat{H}} \in \scS\ .
 $$
 Differentiating at $t=0$, one finds $[\widehat{H},A]\in \scS$.
 \end{proof} 
 
 Let $\{\varphi_1,\dots,\varphi_N\}$ be an orthonormal basis of $\cH_1$, and for $j=1,\dots,N$, define $Z_j = Z(\varphi_j)$. 
 For $K\subseteq \{1,\dots,N\}$, let $\varPhi_K\in \scK$ be defined as in the proof of Theorem~\ref{VACCONDEXTHM}, so that $\{\varPhi_K\}_{K\subseteq\{0,\dots,N\}}$ is an 
 orthonormal basis for $\scK$. Then $|\varPhi_K\rangle\langle \varPhi_K|$ is the gauge invariant Gaussian state whose symbol if $Q = \sum_{j\in K}|\varphi_j\rangle\langle \varphi_j|$.  
 In particular, for all $K\subseteq \{1,\dots,N\}$, 
 $|\varPhi_K\rangle\langle \varPhi_K|\in {\mathfrak S}_{GIG}$.
 
 \begin{lm}\label{ARWYLEM2}   
 For all $K,L\subseteq\{1,\dots,N\}$ of equal cardinality, $|\varPhi_K\rangle\langle \varPhi_L|\in \scS$.
 \end{lm}
 
 \begin{proof} Consider two sets $K,L\subseteq\{1,\dots,N\}$ of the same cardinality $n$. 
 Then $|L\cap K|$, the cardinality of $L\cap K$, satisfies
 $$0 \leq n - |L\cap K| \leq \min\{n,N-n\}\ $$
 If $n -  |L\cap K| =0$, then $L =K$ and $|\varPhi_K\rangle\langle \varPhi_L|\in \scS$ because it is a gauge invariant Gaussian state. 
 We now prove that $|\varPhi_K\rangle\langle \varPhi_L|\in \scS$ by induction on $m = n -  |L\cap K|$, noting this has been proved if $m=0$. 
 
 Make the inductive assumption that $|\varPhi_K\rangle\langle \varPhi_L|\in \scS$, for all $K,L$ of cardinality $n$ with $n - |L\cap K| \leq m< \min\{n,N-n\}$.
 We now show that $|\varPhi_K\rangle\langle \varPhi_L|\in \scS$ for all
 $K,L$ of cardinality $n$ with $n - |L\cap K| = m+1$. For any such $K,L$, there is some $k\in K$ such that $k\notin L$, and some $\ell\in L$ such that $\ell\notin K$.
 Define $K' = (K\backslash\{k\})\cup \{\ell\}$.
 Then $|K'| = n = |L|$ and  and $|K'\cap L| = |K\cap L|+1$ so that $n - |L\cap K'| = m$. By the inductive hypothesis, $|\varPhi_{K'}\rangle\langle \varPhi_L|\in \scS$
 
 Define $H\in M_N(\C)$ by
 $$H_{p,q} =  \delta_{p,\ell}\delta_{q,k} +  \delta_{q,\ell}\delta_{p,k}\ . $$
 Then $H$ is self-adjoint, and  $\widehat{H} = Z_\ell^*Z_k + Z_k^*Z_\ell$. Since $k\notin K'$, $Z_k \varPhi_{K'} =0$, and hence
 $$
 \widehat{H} \varPhi_{K'} =Z_k^*Z_\ell \varPhi_{K'}  =  \varPhi_{K}\ .
 $$ 
 Likewise, since $k\notin L$, $Z_\ell^*Z_k \varPhi_L =0$ and hence  $\widehat{H} \varPhi_{L}  = Z_k^*Z_\ell\varPhi_L$. Therefore, 
 \begin{eqnarray*}
 [\widehat{H},|\varPhi_{K'}\rangle\langle \varPhi_L|] &=& |\widehat{H} \varPhi_{K'}\rangle\langle \varPhi_L| - |\varPhi_{K'}\rangle\langle \widehat{H}\varPhi_L|\\
 &=& |\varPhi_{K}\rangle\langle \varPhi_L| - |\varPhi_{K'}\rangle\langle Z_k^*Z_\ell\varPhi_L|\\
 \end{eqnarray*}
 By the inductive hypothesis and Lemma~\ref{ARWYLEM1}, $[\widehat{H},|\varPhi_{K'}\rangle\langle \varPhi_L|] \in \scS$. 
 Therefore, $| \varPhi_{K}\rangle\langle \varPhi_L| - |\varPhi_{K'}\rangle\langle Z_k^*Z_\ell\varPhi_L|\in \scS$. How making the same computation with $H$ replaced by 
 $$H_{p,q} =  i\delta_{p,\ell}\delta_{q,k} -i  \delta_{q,\ell}\delta_{p,k}$$
 yields $i| \varPhi_{K}\rangle\langle \varPhi_L| +i |\varPhi_{K'}\rangle\langle Z_k^*Z_\ell\varPhi_L|\in \scS$. Therefore, $| \varPhi_{K}\rangle\langle \varPhi_L|\in \scS$. 
 \end{proof}
 
 Now consider the basis  $\{Z_{\bb,\gg}\ :\ \bb,\gg\in \{0,1\}^N\}$ of $\scA_{\cH_1}$,  as in \eqref{POLYBAS}, induced by some orthonormal basis of $\cH_1$. It is evident that 
 $Z_{\bb,\gg}\in \scG_{\cH_1}$ if and only if $|\bb| = |\gg|$, and moreover, that 
 $$
 \{Z_{\bb,\gg}\ :\ \bb,\gg\in \{0,1\}^N \ :\  |\bb| = |\gg|\}
 $$
 is a basis of $\scG_{\cH_1}$.  For $1\leq n \leq N$, there are $\binom{N}{n}$ ways to choose $K$ and $L$ such that $|K| = |L| =n$, 
\begin{equation}\label{CARMONG2}
{\rm dim}(\scG_{\cH_1}) = \sum_{n=0}^N \binom{N}{n}^2 = \binom{2N}{N}\ 
\end{equation}
where the final inequality is a special case of Vandermonde's identity.

 \begin{proof}[Proof of Theorem~\ref{AWTHM}] By Lemma~\ref{ARWYLEM2}, 
 \begin{equation}\label{ARAKIWYSP1}
 \{ |\varPhi_K\rangle\langle \varPhi_L|\ :\  K,L\subseteq \{1,\dots,N\}\ ,\ |K| = |L|\} \subset \scS\ .
 \end{equation} 
 This set is orthogonal for the Hilbert-Schmidt inner product on $\scH$: For $|K| = |L|$ and $|K'| = |L'|$,
 $$\tr\left( (|\varPhi_{K'}\rangle\langle \varPhi_{L'}|)^*|\varPhi_K\rangle\langle \varPhi_L| \right) = \langle \varPhi_{K'}\varPhi_K\rangle \langle \varPhi_{L}\varPhi_{L'}\rangle = \delta_{K,K'}\delta_{L.L'}\ .$$
 Thus the dimension of $\scS$ is at least ${\displaystyle \sum_{n=0}^N \binom{N}{n}^2 = \binom{2N}{N}}$. By \eqref{CARMONG2}, this is the dimension of $\scG_{\cH_1}$, and then since $\scS \subseteq \scG_{\cH_1}$,  $\scS = \scG_{\cH_1}$. 
 \end{proof}

\section*{Acknowledgements}

A preliminary version of this work was presented at  GOA q-Day, March 20, 2026. I thank Magdalena Musat, Mikael R\o rdam, and Dima Shlyakhtenko for stimulating discussions. I thank Len Gross, Hami Mehrabi and an anonymous referee for their helpful comments on an earlier draft of this paper.

\section*{Conflict of interest satement}
The author states no conflict of  interest.


\begin{thebibliography}{100}
 



\footnotesize{



\bibitem{AC20} E.~Amorim and E. Carlen, \textit{Complete positivity and self-adjointness}, Linear Alg. and Appl. {\bf 611} (2020). 389--439


\bibitem{A87} H.~Araki, \textit{Bogoliubov automorphisms and Fock representations of canonical anticommutation relations}, Contemp. Mathmetatics {\bf 62}  (1987), 23--141.

\bibitem{AW64} H.~Araki and W.~Wyss, \textit{Representations of canonical anticommutation relations}, Helvetica Physica Acta, {\bf 37} (1964), 136--159

\bibitem{Ar73} G.~S.~Agarwal, \textit{Open quantum Markovian systems and the microreversibility}, Z. Physik {\bf 258}, 409-422, 1973.

\bibitem{BGL} D.~Bakry, I.~Gentil, and M. Ledoux, \textit{Analysis and geometry of Markov diffusion operators}, {\bf 348}
of Grundlehren der Mathematischen Wissenschaften, Springer, Berlin, 2014.


\bibitem{BV68} E.~Balsev and A.~Verbeuere, \textit{States on Clifford Algebras}, Comm. Math. Phys., {\bf 7} (1968), 55--76.

\bibitem{BCS57}J.~Bardeen,  L.~N.~Cooper and J.~R.~Schrieffer,  \textit{Microscopic Theory of Superconductivity}. Physical Review. {\bf 106} (1957), 162--164.


\bibitem{BW} R.~Brauer, R., H.~Weyl:  \textit{Spinors in $n$ dimensions}, Am. Jour.
Math., {\bf 57}, 1935 pp. 425-449.

\bibitem{B05} S.~Bravyi, \textit{Lagrangian representation for fermionic linear optics}, Quantum Inf. and Comp., {\bf 5} (2005), 216--238

\bibitem{B58} N.~N.~Bogoliubov, \textit{A new method in the theory of superconductivity, I}, Soviet Phys.~ JETP, {\bf 34} (1958), 41--46.

\bibitem{C25}
E.~A.~Carlen {\it Inequalities in matrix algebras}. Graduate Studies in Math. {\bf 251}, (2025),  Amer. Math. Soc., Providence, RI.

\bibitem{CL93} E.~A.~Carlen and E.~H.~Lieb, \textit{Optimal hypercontractivity for Fermi fields and related non-commutative
integration inequalities}, Comm. Math. Phys. {\bf 155}, 27-46, 1993.    


\bibitem{CM17} E.~A.~Carlen and J.~Maas, 
	\textit{Gradient flow and entropy inequalities for quantum Markov semigroups with detailed balance},
	   Jour.  Func. Analysis {\bf 273},  (2017) 1810--1869.
	   
	   \bibitem{CM20} E.~A.~Carlen and J.~Maas,  \textit{Non-commutative Calculus, Optimal Transport and Functional Inequalities in Dissipative Quantum Systems} Jour. Stat. Phys., {\bf 178} (2020), 319-378


\bibitem{CM13} E.~A.~Carlen and J.~Maas, 
	\textit{An analog of the 2-Wasserstein metric in non-commutative probability under which the fermionic Fokker-Planck equation is gradient flow for the entropy}, 
	Comm. Math. Phys. {\bf 331}, 887--926, 2014.
	
\bibitem{C68} P.~Chernoff, \textit{Note on product formulas for semi-groups}, Jour. Func. Analysis {\bf 2} (1968), 238-242.

\bibitem{Cook53} J.~M.~Cook, \textit{The mathematics of second quantization}, Trans. Amer. Math. Soc. {\bf 74} (1953), 222-245. 



\bibitem{Cip97} F.~Cipriani  \textit{Dirichlet forms and Markovian semigroups on standard forms
of von Neumann algebras}. J. Funct. Anal. {\bf 147}, 259-300, 1997.
 
\bibitem{Cip08} F.~Cipriani, \textit{Dirichlet forms on noncommutative spaces}, in \textit{Quantum 
potential theory}, vol. 1954 of
Lecture Notes in Math., Springer, Berlin, 161-276, 2008.

\bibitem{CS03} F.~Cipriani and J.~L.~Sauvageot, \textit{Derivations as square roots of Dirichlet forms}, 
	J. Funct. Anal., {\bf 201} 78-120, 2003



\bibitem{DL92}  E.~B.~Davies, J.~M.~Lindsay, \textit{Non-commutative symmetric Markov semigroups}, Math. Z. {\bf 210} 379-411, 1992.

\bibitem{DA68} G.~F.~Dell'Antonio, {Structure of the algebras of some free systems}, Comm. Math. Phys. {\bf 9} (1968), 81--117.

\bibitem{E79} D.~E.~Evans, \textit{Completely positive quasi-free maps on the CAR algebra}, Comm. Math.Phys. {\bf 70} (1979), 53--68.

\bibitem{D06} J.~Derezi\'nski,   \textit{Introduction to Representations of the Canonical Commutation and Anticommutation Relations}. In:  \textit{Large Coulomb Systems}, J.~Derezi\'nski and H.~Siedentop eds.,  Lecture Notes in Physics, {\bf 695} (2006) Springer, Berlin.

\bibitem{DFP08} B.~Dierckx,  M.~Fannes and  M.~Pogorzelska, \textit{Fermionic quasifree states and maps in information theory}, J. Math. Phys. {\bf 49} (2008), 032109.



\bibitem{FR15} F.~Fagnola and R.~Reboledo, \textit{Entropy production and detailed balance for a class of quantum Markov semigroups}, Open Syst. Inf. Dyn. {\bf 22}, 1550013 (2015).

\bibitem{FU07} F.~Fagnola and V.~Umanit\`a, \textit{Generators of Detailed Balance Quantum Markov Semigroups}, Infin. Dimens. Anal. Quantum Probab. Relat. Top.,
{\bf 10}, no. 3, 335-363, 2007.

	
\bibitem{FR80} M.~Fannes and F.~Rocca. \textit{A class of dissipative evolutions with applications in thermodynamics
of fermion systems}. J. Math. Phys.,  {\bf 21} (1980) 221--226..
	
 \bibitem{F32} V.~Fock, \textit{Konfigurationsraum und zweite Quantelung}. Zeitschrift f\"ur Physik (in German). {\bf 75}  1932, 9--10. 
 
 \bibitem{GC02} G.~Giedke and J.~I.~Cirac, \textit{The characterization of Gaussian operations and Distillation of Gaussian States}, Phys. Rev. A {\bf 66} (2002), 032316. 
 
 \bibitem{GJ87} J.~Glimm and A.~Jaffe, {Quantum physics: A Functional Integral Point of View}, Springer, Berlin, 1987.


\bibitem{GKS76} V.~Gorini, A.~Kossakowski and E.~C.~G.~Sudarshan, \textit{Completely positve 
dynamical semigroups of $N$-level systems},
J. Math. Phys. {\bf 17}, 821-825,1976.

\bibitem{Gr72} L.~Gross, \textit{Existence and uniqueness of physical ground states}, J. Funct. Anal. {\bf 10}  59-109, 1972.

\bibitem{Gr75} L.~Gross,  \textit{Hypercontractivity and logarithmic Sobolev inequalities for the Clifford-Dirichlet
form}, Duke Math. J. {\bf 42}  383-396, 1975.

\bibitem{Gr75B} L.~Gross,  \textit{ Logarithmic Sobolev inequalities}. Amer. J. Math. {\bf 97} (1975), 1061--1083.

\bibitem{H67} P.~R.~Halmos, \textit{A hilbert space problem book},  Van Nostrand, Princeton, 1967.



\bibitem{HK75} N.~M.~Hugenholtz and R.~V.~Kadison, \textit{Automorphisms and quasi-free state of the CAR algebra}, Comm. Math.Phys. {\bf 43} (1975), 181--197.

\bibitem{JW28} P.~Jordan and E.~Wigner, \textit{ \"Uber das Paulische \"Aquivalenzverbot},  Z. Physik {\bf 47},  (1928) 631--651

\bibitem{JM05} T.~Jacobs and C.~Maes, \textit{ Reversibility and irreversibility within the quantum formalism}, Physicalia Magazine, {\bf 27} (2005), 119--130.




\bibitem{LSM61}E.~Lieb, T.~Schultz and D.~Mattis, \textit{Two soluble models of an antiferromagnetic chain}, Ann.
Phys.  {\bf 16}(1961), 407--466. 


\bibitem{KFGV}
A.~Kossakowski, A.~Frigerio, V.~Gorini, and M.~Verri, \textit{Quantum detailed balance and KMS
condition}, Comm. Math. Phys., {\bf 57}, 97-110, 1977.


\bibitem{Lin76} G.~Lindblad, \textit{On the generators of quantum dynamical semigroups}, Comm. Math. Phys., 
{\bf 48}, 119-130, 1976.

\bibitem{Lu76} L.-E.~Lundberg, \textit{Quasi-free ``second quantization''}, Commun. Math. Phys., {\bf 50} (1976) 103--112 

\bibitem{MS98} W.~A.~Majewski and R.~F.~Streater, \textit{Detailed balance and quantum dynamical
maps}, J. Phys. A: Math. Gen. {\bf 31}, (1998) 7981-7995.

\bibitem{M66}
F.~G~Mehler
\textit{ \"Uber die {E}ntwicklung einer {F}unction von beliebig vielen
  {V}ariablen nach {L}aplaschen {F}unctionen h\"oherer {O}rdnungn.}
 Crelle's Journal {\bf 66}  (1866), 161--176.
 
 \bibitem{PAM76} P.-A.~Meyer \textit{L'Operateur carr\'e du champ}. S\'eminaire de Probabilit\'es X Universit\'e de Strasbourg. Lecture Notes in Mathematics. {\bf 511}, Berlin, Heidelberg: Springer.  142--161.



\bibitem{N69}  E.~Nelson, \textit{Topics in Dynamics I: Flows} Princeton Univ. Press, Princeton, 1969.


\bibitem{N73}  E.~Nelson, \textit{The free Markoff field}, Jour.. Func. Analysis {\bf12} (1973), 211-227.





\bibitem{Paul02}
{V.~Paulsen}, {\em Completely bounded maps and operator algebras}, vol.~78
  of Cambridge Studies in Advanced Mathematics, Cambridge University Press,
  Cambridge, 2002.





\bibitem{PS70} R.~Powers, R. and E.~St\o rmer, \textit{Free states of the canonical anticommutation
relations}, Comm. Math. Phys. {\bf 16} (1970) 1--33. 

\bibitem{Pr08} T.~Prosen, \textit{Third quantization: a general method to solve
master equations for quadratic open Fermi systems}, New J. Phys. {\bf 10} (2008), 043026

\bibitem{Pr} T.~Prosen, \textit{Spectral theorem for the Lindblad equation for
quadratic open fermionic systems}, J. Stat. Mech. (2010) P07020




\bibitem{SU75}  R.~Schrader and R.~D.~Uhlenbrock: \textit{Markov structures on Clifford algebras}, Jour. Func. Analysis., {\bf 18} (1975) 369---413..

\bibitem{S53}  I.~E.~Segal: \textit{A non-commutative extension of abstract
integration}, Annals of Math., {\bf 57} 1953), 401--457, 

\bibitem{S56}   I.~E.~Segal.: \textit{Tensor algebras over Hilbert spaces II}, 
Annals of Math., {\bf 63}, (1956) 160-175, 

\bibitem{S61}   I.~E.~Segal.: \textit{Mathematical characterization of the physical vacuum for a  linear Bose-Einstein field}, 
llinois J. Math., {\bf 6}  (1962), 500--523.


\bibitem{S65}   I.~E.~Segal.: \textit{Algebraic integration theory}, 
Bull. Am. Math. Soc.., {\bf 71}  (1965),  pp. 419-489.

\bibitem{SS65} D.~Shale and W.~F.~Stinespring, \textit{States of the Clifford Algebra}, Annals of Math., {\bf 80} (1964),  365--381.

\bibitem{S74} B.~Simon, The $P(\Phi)_2$ Euclidean (Quantum) Field Theory, Princeton University Press, Princeton, 1974.




\bibitem{To57} J.~Tomiyama, \textit{On the projection of norm. one in W$^*$-algebras}, Proc. Japan Acad. {\bf 33} (1957), 608--612.




\bibitem{U54} H.~Umegaki, \textit{Conditional expectation in an operator algebra}, 
Tokohu Math. J. {\bf 6} (1954),  177--181.

\bibitem{U56} H.~Umegaki,  \textit{Conditional expectation in an operator algebra, II}, 
Tokohu Math. J. {\bf 8} (1956), 86--100.


\bibitem{U59} H.~Umegaki,  \textit{ Conditional expectation in an operator algebra, III}, 
Kodai Math. Sem. Rep. {\bf 11} (1959),  51--64 


\bibitem{U62} H.~Umegaki, {\em Conditional Expectation in an Operator Algebra. IV. Entropy and Information,} Kodai Math. Sem. Rep. {\bf14} (1962), 59--85.

\bibitem{W32} E.~Wigner,
Uber die Operation der Zeitumkehr in der Quantenmechanik,
Nachr. Akad. Ges. Wiss. G\"ottingen {\bf 31} (1932), 546--559.
  
  


\bibitem{W75} J.~C.~Wolfe, \textit{Free states and automorphisms of the Clifford algebra}, Comm. Math. Phys. {\bf 45} (1975),  53--58. 

}



\end{thebibliography}
 \end{document}